\documentclass[a4paper, 10pt]{amsart}

\usepackage{amsmath,amssymb,amscd, xypic}

\usepackage{amsfonts}
\usepackage[english]{babel}
\usepackage[leqno]{amsmath}
\usepackage{amssymb,amsthm}
\usepackage{mathrsfs} 
\usepackage{amscd}
\usepackage{enumerate}
\usepackage{epsfig}
\usepackage{relsize}

\usepackage{layout}
\usepackage{fullpage}

\usepackage[usenames,dvipsnames]{xcolor}

\usepackage{hyperref}

\hypersetup{
 colorlinks,
 citecolor=Green,
 linkcolor=blue,
 urlcolor=Blue}

\usepackage[matrix,arrow,tips,curve]{xy}
\input{xy}
\xyoption{all}

\numberwithin{equation}{section}
\newenvironment{sis}{\left\{\begin{aligned}}{\end{aligned}\right.}

\newcommand{\M}{\mathcal{M}}

\renewcommand{\O}{\mathcal{O}}

\renewcommand{\H}{\mathcal{H}}

\newcommand{\wt}{\widetilde}
\newcommand{\ov}{\overline}
\newcommand{\un}{\underline}

\newcommand{\E}{\mathcal{E}}

\newcommand{\Q}{\mathbb{Q}}
\newcommand{\R}{\mathbb{R}}

\renewcommand{\P}{\mathbb{P}}

\newcommand{\Z}{\mathbb{Z}}
\newcommand{\N}{\mathbb{N}}
\renewcommand{\O}{\mathcal{O}}

\newcommand{\bV}{\mathbb{V}}
\newcommand{\X}{\mathcal{X}}
\newcommand{\bT}{\mathbb{T}}
\newcommand{\bF}{\mathbb{F}}

\DeclareMathOperator{\NS}{NS}

\DeclareMathOperator{\Proj}{Proj}
\DeclareMathOperator{\Aut}{Aut}

\DeclareMathOperator{\Sym}{Sym}
\DeclareMathOperator{\ev}{ev}
\renewcommand{\char}{\text{char}}
\renewcommand{\Im}{\text{Im}}
\DeclareMathOperator{\supp}{supp}

\DeclareMathOperator{\lcm}{lcm}
\DeclareMathOperator{\rk}{rk}

\DeclareMathOperator{\codim}{codim}

\DeclareMathOperator{\BS}{BS}
\DeclareMathOperator{\Nef}{Nef}

\theoremstyle{plain}
\newtheorem{theorem}{Theorem}[section]
\newtheorem{corollary}[theorem]{Corollary}
\newtheorem{lemma}[theorem]{Lemma}
\newtheorem{proposition}[theorem]{Proposition}
\newtheorem{theoremalpha}{Theorem}

\theoremstyle{definition}
\newtheorem{definition}[theorem]{Definition}

\newtheorem{question}[theorem]{Question}
\newtheorem{example}[theorem]{Example}
\newtheorem{remark}[theorem]{Remark}

\newtheorem{setup}[theorem]{Setup}

\usepackage{tikz}									
\usetikzlibrary{matrix}
\usetikzlibrary{patterns}
\usetikzlibrary{matrix}
\usetikzlibrary{positioning}
\usetikzlibrary{decorations.pathmorphing}
\usetikzlibrary{cd}
\usetikzlibrary{calc}
\usetikzlibrary{decorations.markings}

\title{Slope inequalities for KSB-stable and K-stable families}
\author{Giulio Codogni}
\address{Dipartimento di Matematica, Universit\`a degli Studi di Roma Tor Vergata, Via della Ricerca Scientifica, 00133 Roma, Italy.}
\email{codogni@mat.uniroma2.it}

\author{Luca Tasin}
\address{Dipartimento di Matematica F.\ Enriques, Universit\`a degli Studi di Milano, Via Cesare Saldini 50, 20133 Milano, Italy} 
\email{luca.tasin@unimi.it}

\author{Filippo Viviani}
\address{Dipartimento di Matematica e Fisica, Universit\`{a} degli Studi Roma Tre,  Largo San Leonardo Murialdo, 00146 Rome, Italy} 
\email{viviani@mat.uniroma3.it}

\makeatletter
\@namedef{subjclassname@2020}{%
  \textup{2020} Mathematics Subject Classification}
\makeatother


\begin{document}

\maketitle

\begin{abstract}
We prove some higher dimensional generalisations of the slope inequality originally due to G. Xiao, and to M. Cornalba and J. Harris. We give applications to families of KSB-stable and K-stable pairs, as well as to the study of the ample cone of the moduli space of KSB-stable varieties. Our proofs relies on the study of the Harder-Narasimhan filtration, and some generalisations of Castelnuovo's and Noether's inequalities.
\end{abstract}

\section*{Introduction}

The first slope inequality was proven at the same time by G. Xiao in \cite{Xiao} and by M. Cornalba and J. Harris in \cite{CH}. They were looking at a non-constant morphism $f\colon S\to T$ from a smooth projective minimal surface to a smooth projective irreducible curve, whose general fibre has genus $g$ at least $2$. They showed that
\begin{equation}\label{E:motivation}
K_{S/T}^2 \geq 4\frac{g-1}{g}  \deg f_*\O_S(K_{S/T}) \,.
\end{equation}
Since then, the name \emph{slope inequalities} has been used for inequalities of the form
$$ L^{n+1}\geq C \deg f_*\O_X(L)$$
where $f\colon (X,L)\to T$ is a polarized family  over a projective curve satisfying convenient hypotheses, and $C$ is a constant which depends just on the general fiber of $f$. 

In the present work, we prove some new slope inequalities and we give some applications to the study of the ample cone of moduli spaces. Before presenting our main results, let us comment on the motivations and techniques used in the above mentioned works \cite{Xiao} and \cite{CH}, and that served as inspiration for this paper. 

The two original papers about the slope inequality \eqref{E:motivation} had rather different motivations and proofs. M. Cornalba and J. Harris were interested in showing the ampleness of some natural line bundles on the moduli space of stable curves (so interpreted $f$ as a family of curves), a program that was completed in \cite{Cproj} (see also \cite{Alp} for some modern developments).  They deduced their result from the GIT stability of the fiber of $f$, by reducing the slope inequality to the non-negativity of a Hilbert-Mumford weight. Let us stress that pluricanonical smooth curves are GIT stable and that the moduli space of stable curves can be constructed using GIT on the Hilbert or Chow scheme of pluricanonical curves. From  \cite{CH}, we retain the motivation, i.e.  applying slope inequalities to produce ample line bundle on moduli spaces, and the idea that slope inequalities should hold under the same stability assumption used to construct moduli spaces.

G. Xiao's goal was to understand the geometry of surfaces fibred over a curve. From his work, we retain the scheme of proof, which we now briefly recall. He starts off considering the Harder-Narasimhan filtration $\{\E_{\bullet}\}$ of the push-forward $f_*\O_S(K_{S/T})$, and bounds the degree of $K_{S/T}^2$ using the slopes of the filtration. As  $f_*\O_S(K_{S/T})$ is nef, all the slopes are non-negative, and he uses this non-negativity to handle the above mentioned bound. Restricting one $\E_i$ to a fiber, Xiao obtains a linear subsystem of the canonical linear system. He applies Clifford's theorem to this linear system to compare its degree with its rank, and ultimately get the desired slope inequality. Among the novelties that we introduce in this paper, in Section \ref{Sec:Noether} we prove various generalisations of Noether and Castelnuovo inequalities which are then used in place of Clifford's theorem. 

We now present the main results that we obtain in this paper, and that can be divided in three categories: slope inequalities for families of KSB-stable (canonically polarized) pairs, slope inequalities for families of K-stable (log Fano) pairs and slope inequalities for arbitrary polarized families.

\subsection*{Slope inequalities for KSB-stable families}

In this subsection, we collect the slope inequalities that we prove on families of KSB-stable pairs, i.e. pairs with slc singularities and ample log-canonical divisor. 
%Recall that, for log-canonically polarized varieties, KSB stability is equivalent to to the existence of a K\"{a}hler-Einstein metric, and it allows the construction of projective moduli spaces. 

Our first result is a general existence result, which says that the slope of KSB-stable families is bounded away from zero by a constant that only depends on the relative dimension and the coefficient set of the boundary.

\begin{theoremalpha}[Existence of slope inequalities]\label{T:exis-slope}
	Fix an integer $n\geq 1$ and a subset $I$ of $[0,1]$ satisfying the DCC (=descending chain condition).  Then there exists a  constant $s(n,I)>0$ such that
	$$(K_{X/T}+\Delta)^{n+1}\geq s(n,I)\deg f_*\O_X(K_{X/T}+\Delta) $$
	for every  KSB-stable family $f\colon (X,\Delta)\to T$  over a smooth, irreducible, projective curve $T$ such that the relative dimension of $f$ is $n$  and the coefficients of $\Delta$ belong to $I$. 
\end{theoremalpha}

The above Theorem \ref{T:exis-slope}  is a special case of Corollary \ref{C:existence} where the same result is proved more generally for generic slc families (i.e. such that $K_{X/T}+\Delta$ is $\Q$-Cartier and the general fiber $(F,\Delta_F)$ is slc, see set-up \ref{N:KSBgen})  such that $K_{X/T}+\Delta$ is $f$-semiample and $f$-big. 
Let us stress that, even though $n$ and $I$ are fixed, the volume of the fibers of the families are not fixed, hence the fibers of the families considered in the statement vary in an unbounded set. 
As explained in the proof of Corollary \ref{C:existence}, the constant $s(n,I)$ can be taken equal to 
$$
s(n,I)=\frac{1}{b(n,I)^n},
$$
where $b(n,I)>0$ is the constant, whose existence is guaranteed by Hacon-McKernen-Xu \cite{ACC}, such that $b(n,I)(K_Z+\Delta_Z)$ gives a birational map for all lc pairs $(Z,\Delta_Z)$ such that the dimension of $Z$ is $n$, the coefficients of $\Delta_Z$ belong to $I$, and $K_Z+\Delta_Z$ is big.
Example \ref{Ex:Double} shows that the constant $s(n,\emptyset)$  decays at least double exponentially in $n$. 

The second result provides some explicit lower bounds on the slope of KSB-stable families, depending on the geometry of the family (such as the volume of the irreducible components or the Cartier index of the general fiber).

\begin{theoremalpha}\label{T:KSB-fam}
Let $f\colon (X,\Delta)\to T$ be a KSB-stable family over a a smooth, irreducible, projective curve $T$ and denote by $(F, \Delta_F)$ the general fiber of $f$. 
\begin{enumerate}
\item \label{TT:KSB-fam1}
Assume that there exists  $m \in \N_{>0}$ such that  at least one of the following conditions hold true 
\begin{itemize}
\item $m(K_{X/T}+\Delta)$  is Cartier and $f$-globally generated;
\item $\Delta$ is a reduced Weil divisor and $m(K_{F}+\Delta_F)$  is Cartier and globally generated.
\end{itemize}
Let $w\in \Q_{>0}$ such that the volume of the pull-back of $K_F+\Delta_F$ to any irreducible component of the normalisation of $F$ is at least $w$. Then
$$ m^{n+1}(K_{X/T}+\Delta)^{n+1}\geq \frac{2wm^n}{wm^n+n}\deg \left(f_*\mathcal{O}_X(m(K_{X/T}+\Delta))\right).$$

\item \label{T:KSB-fam2} 
Fix $m \in \N_{>0}$ such that  $m(K_{X/T}+\Delta)$  is Cartier and $f$-globally generated. Then 
$$
m^{n+1}(K_{X/T}+ \Delta)^{n+1} \ge \deg f_*\O_X(m(K_{X/T}+\Delta)).
$$

\item \label{T:KSB-fam3} 
Assume that  $\Delta$ is a reduced Weil divisor and let $m,q\in \N_{>0}$ such that at least one of the following conditions holds true 
\begin{itemize}
\item $\phi_{mq(K_{F} + \Delta_F)}$ is generically finite;
\item $mq(K_{F}+\Delta_F)$  is Cartier.
\end{itemize} 
Then 
$$
m^{n+1}(K_{X/T}+ \Delta)^{n+1}  \geq \frac{\deg f_*\O_X(m(K_{X/T}+\Delta))}{q^n}.
$$ 

\item \label{T:KSB-fam4}  Assume that $K_{X/T} + \Delta$ is nef and let $q\in \N_{>0}$ such that either $q(K_F+\Delta_F)$ is Cartier  or $\phi_{q(K_F+\Delta_F)}$ is generically finite. Then
$$(K_{X/T}+\Delta)^{n+1}\geq \frac{\deg\left(f_*\mathcal{O}_X(K_{X/T}+\Delta)\right)}{q^n}.$$
\end{enumerate}
\end{theoremalpha}

The above Theorem \ref{T:KSB-fam} is a special case of Theorem \ref{T:Nef-Fam} which proves similar slope inequalities for generic slc families (i.e. such that $X$ is deminormal, $K_{X/T}+\Delta$ is $\Q$-Cartier and the general fiber $(F,\Delta_F)$ is slc, see set-up \ref{N:KSBgen}). Theorem \ref{T:Nef-Fam} is reduced, by taking the normalisation of $X$, to Proposition \ref{P:slopepair} in which we prove similar slope inequalities for generic lc families  $f\colon (X,\Delta)\to T$
(i.e. such that $X$ is normal, $K_{X/T}+\Delta$ is $\Q$-Cartier and the general fiber $(F,\Delta_F)$ is lc, see set-up \ref{N:KSBgen}). 
We use the condition of being generic lc family through the results of O. Fujino \cite{Fujino} (see also Theorem \ref{T:Fuj} and Corollary \ref{C:Fuj}), which guarantee the nefness of the relative log-canonical bundle and of the push-forwards of pluri-log-canonical bundles. 

As an  application of our slope inequalities for KSB-stable families, we can describe a portion of the ample cone  of the proper DM-stack $\M_{n,v}$ (which, by definition, is equal to the ample cone of its projective coarse moduli space $M_{n,v}$) of $n$-dimensional KSB varieties of volume $v$. We denote by $\lambda_{CM}$ the Chow-Mumford $\Q$-divisor, which is ample by \cite{PX},  and by $\lambda_m$ the $m$-th determinant $\Q$-divisor, which are nef for any $m$ big and divisible enough by \cite{Fujino} (the definitions are recalled in Section \ref{sec:KSBmoduli}). The following result describes infinitely many $2$-dimensional subcones of the ample cone of  $\M_{n,v}$.

\begin{theoremalpha}[Ample cone of KSB moduli spaces]\label{T:ampleKSB}
Fix $n\in \N_{>0}$ and $v\in \Q_{>0}$.
\begin{enumerate}
\item \label{T:ampleKSB1}
Consider a positive integer $m$  such that $mK_V$ is Cartier and globally generated for any $V\in \M_{n,v}$ and let $w\in \Q_{>0}$ such that the volume of the pull-back of $K_F$ to any irreducible component of the normalisation of $V$ is at least $w$. Then the $\Q$-divisor
$$
\lambda_{CM}-\varepsilon \lambda_m
$$
is ample on $\M_{n,v}$ for every rational number $\varepsilon$ in $\left[0,\frac{1}{m^{n+1}}\frac{2wm^n}{wm^n+n}\right)$.
\item \label{T:ampleKSB2}
Consider two positive integers $m$ and $q$ such that, for every $V\in \M_{n,v}$, either $mqK_V$ is Cartier or $\phi_{mqK_V}$ is generically finite. Then the $\Q$-divisor
$$
\lambda_{CM}-\varepsilon \lambda_m
$$
is ample on $\M_{n,v}$ for every rational number $\varepsilon$ in $\left[0,\frac{1}{q^{n}m^{n+1}}\right)$.
\end{enumerate}
\end{theoremalpha}
Theorem \ref{T:ampleKSB} follows from the ampleness of $\lambda_{CM}$ together with the nefness of the divisors considered in Theorem \ref{T:Nef-Mod}. We also establish a variant of Theorem \ref{T:Nef-Mod}, namely Theorem \ref{T:Nef-Away},  in which we prove that some divisors of  $\M_{n,v}$  are ``nef away from the boundary'',  i.e. it intersects non-negatively all the projective integral curves of  $\M_{n,v}$ whose generic point parametrises a normal  KBS-stable variety.

In dimension $n=1$  (in which case $\M_{1,2g-2}$ is the moduli stack of stable curves of genus $g\geq 2$), some of the divisors appearing in Theorem \ref{T:Nef-Mod} are nef but not ample, which implies that  the right extremes of the intervals appearing in above Theorem are sharp.
More precisely, part \eqref{T:ampleKSB1} is sharp if $m=w=1$ and part \eqref{T:ampleKSB2} is sharp if $m=q=1$ (see Remark \ref{R:NefCurves}). 

At the end of Section \S\ref{Sec:KSB}, we introduce the lambda nef cone  $\Nef^{\Lambda}(\M_{n,v})$ as the  intersection of the nef cone $\Nef(\M_{n,v})$ with linear subspace of the rational Neron-Severi vector space  spanned by the Chow-Mumford line bundle $\lambda_{CM}$ and the classes $\lambda_m$ for any $m\geq 1$, 
and we ask for which moduli spaces $\M_{v,n}$,  Theorem \ref{T:ampleKSB}, together with the nefness of $\lambda_m$, is sufficient to describe  $\Nef^{\Lambda}(\M_{n,v})$. This is indeed the case in dimension $n=1$ (see Remark \ref{R:NefCurves}), and it was one of the original motivation of Cornalba-Harris's paper \cite{CH}.

\subsection*{Slope inequalities for K-stable families}

In this subsection, we collect the slope inequalities that we prove on families of K-(semi-,poly-)stable log Fano pairs. 

Recall that K-polystability is a stability condition for log Fano pairs equivalent to the existence of a K\"{a}hler-Einstein metric; it also allows the construction of projective moduli spaces. It is worth recalling that the study of the Harder-Narasimhan filtration, so important for the proof of our slope inequalities, plays also a crucial role in the proof of the projectivity of this moduli space, see \cite{CP,CP2,Posva,Xu_Zhuang}.

In this case, as opposite to the KSB-stable case, the push-forward of the  pluri-anti-canonical bundle is not nef, see Remark \ref{rem:neg}. Its negativity can be however bounded in terms of the stability threshold $\delta(F,\Delta_F)$ of the generic fiber of the family. This threshold is a numerical invariant which measure the stability of a log Fano pair: it is the first time that such an invariant plays a role in a slope inequality.

Let us recall here one of our inequalities, which holds just when the general fiber is K-stable, referring to Section \ref{Sec:Fano} for the notations and some variants.

The result is formulated in terms of an auxiliary line bundle $H_C$; using the projection formula as shown after the statement of Theorem \ref{T:F_slope1} one can get the slope inequality for the anti-canonical line bundle.

\begin{theoremalpha}[see Theorem \ref{T:F_slope1}]\label{T:Fano-fam}
Let  $f\colon (X,\Delta) \to T$ be a fibration from a normal projective irreducible variety $X$ of dimension $n+1$ to  a smooth projective irreducible curve $T$ such that $-K_{X/T}-\Delta$ is $\Q$-Cartier and $f$-ample.

Assume that the there exists a K-stable geometric fiber $(F,\Delta_F)$, i.e. $\delta(F,\Delta_F)>1$. Let $v:=(-K_F-\Delta_F)^n=((-K_{X/T}-\Delta)_{|F})^n$. For any rational number $C>1$ consider the $\Q$-Cartier $\Q$-divisor on $X$
\begin{equation*}
H_C:=-K_{X/T}-\Delta+C\frac{\delta(F,\Delta_F)}{(\delta(F,\Delta_F)-1)v(n+1)}f^* \lambda_{CM}(X/T).
\end{equation*}

\begin{enumerate}
\item Let $q\geq \frac{1}{C-1}$ be a positive integer such that $qH_C$ is Cartier. Then  
$$q^{n+1}H_C^{n+1}\geq \deg f_*\mathcal{O}_X(qH_C).$$

 \item  Let $q\geq \frac{1}{C-1}$ be a positive integer such that $qH_C$ is  Cartier and  $-q(K_F+\Delta_F)$ gives a generically finite map. Then
 $$q^{n+1}H_C^{n+1}\geq 2\frac{h^0(F,q(-K_{F}-\Delta_F))-n}{h^0(F,q(-K_{F}-\Delta_F))}\deg f_*\O_X(qH_C).$$
 
  \item  Let $q\geq \frac{1}{C-1}$ be a positive integer such that $qH_C$ is Cartier and $-q(K_{F}+\Delta_F)$ is globally generated. Then
	$$q^{n+1}H_C^{n+1}\geq 2\frac{q^nv }{q^nv +n} \deg f_*\O_X(qH_C)  \,.$$
\end{enumerate}
\end{theoremalpha}

When the generic fiber is K-polystable but not K-stable, we show in Theorem \ref{T:F_slope2} the statement of Theorem \ref{T:Fano-fam} holds true up to a finite base change and a birational modification  (which does not change the general fiber) of the original family and up to replacing $\delta(F,\Delta_F)$ with the $\bT$-twisted stability threshold $\delta_{\mathbb{T}}(F,\Delta_F)$, for some maximal torus $\mathbb{T}\subseteq \Aut (F,\Delta_F)$. In Example \ref{ex:twist}, we show that such a birational modification is necessary.

In Theorem \ref{thm:Fano_moduli}, we apply the slope inequality of Theorem \ref{T:Fano-fam} in order to prove that some divisors on the moduli stack  $\mathcal{M}_{n,v}^K$ of K-semistable Fano varieties with dimension $n$ and volume $v$ are nef away from the strictly K-polystable locus, i.e. they intersect non-negatively the curves that are generically contained in the open Deligne-Mumford substack $\mathcal{M}_{n,v}^{K,s}\subseteq \mathcal{M}_{n,v}^{K}$  parametrizing K-stable Fano varieties.

\subsection*{Slope inequality for arbitrary divisors}\label{sec:slope-L}

All the previous results follow from some general slope inequalities for a $\Q$-Cartier $\Q$-divisor $L$ on the total space of a family of  $n$-dimensional varieties $f\colon X\to T$ (with $\dim T=1$), as in set-up \ref{setup}. In this case, we do not assume any stability condition, but we rather make some semi-positivity assumptions: namely the nefness of $L$ and $f_*\O_X(L)$. This semi-positivity is usually implied by a suitable stability condition, e.g. KSB-stability or K-stability. For these results, we need to assume that the total space $X$ is normal, contrary to the KSB case where deminormality is enough.
Note that for the applications to families of KSB-stable or K-stable varieties, it is crucial to work with  $\Q$-Cartier $\Q$-divisors, rather than just Cartier or Weil divisors.

The first type slope inequalities that we prove in this general context involve the numerical invariants of the polarized general fiber $(F,L_F)$, and more specifically either $L_F^n$ or $h^0(F,L_F)$.

\begin{theoremalpha}[see Corollaries \ref{C:Xiao-high1}, \ref{C:Xiao-high2}, \ref{C:Xiao-bir1} and  \ref{C:Xiao-bir2}]\label{T:slopeL1}
Let $f: X \to T$ be a fibration , where $X$ is a normal projective variety of dimension $n+1$ and $T$ is a smooth projective curve, and let $F$ be the general fiber of $f$. 

Let $L$ be a $\Q$-Cartier $\Q$-divisor on $X$; denote by $L_F$ its restriction to $F$ and by $\phi_{L_F}$ the rational map induced by $L_F$. 
Assume that $L$ and $f_*\O_X(L)$ are nef.

\begin{enumerate}
\item \label{T:slopeL1i} If  $\phi_{L_F}$ is generically finite, then
	$$
	L^{n+1} \ge 
	\begin{cases} 
	4\frac{h^0(F,L_F)-n}{h^0(F,L_F)}\deg f_*\O_X(L) & \text{ if  either } \dim F\geq 2 \text{ and } \kappa(F)\geq 0, \\
	& \text{ or } \dim F=1 \: \text{ and } L_F\: \text{ is special},\\
	2\frac{h^0(F,L_F)-n}{h^0(F,L_F)}\deg f_*\O_X(L) & \text{ otherwise.} 
	\end{cases}
	$$
	
\item \label{T:slopeL1ii} If $L_F$ is Cartier, globally generated and big, then
	$$
	L^{n+1} \ge
	\begin{cases} 
	4\frac{L_F^n }{L_F^n +2n} \deg f_*\O_X(L) &  \text{ if  either } \dim F\geq 2 \text{ and } \kappa(F)\geq 0, \\
	  & \text{ or } \dim F=1\: \text{ and } L_F \: \text{ is special},\\
	  	2\frac{L_F^n }{L_F^n +n} \deg f_*\O_X(L)  & \text{ otherwise.} 
	\end{cases}
	$$
	
\item \label{T:slopeL1iii} 	Suppose that $\phi_{L_F}$ is birational and $n\ge 2$.  Assume that the singularities of the general fibre $F$ are canonical and let $s \in \mathbb N$ such that $K_{ F}-sL_F \ge 0$.
	Then 
	$$
	L^{n+1} \ge    2(n+s)\frac{h^0(F,L_{F})-n-2}{h^0( F,L_{F})}\deg f_*\O_X(L) .
	$$

\item \label{T:slopeL1iv} 	 Suppose that   $L_F$ is Cartier, globally generated, $\phi_{L_F}$ is birational and $n\ge 2$.  Assume that the singularities of the general fibre $F$ are canonical and let $s \in \mathbb N$ such that $K_{ F}-sL_F \ge 0$. Then
	$$
	L^{n+1} \ge 2(n+s)\frac{L_{F}^n}{L_{F}^n +(n+s)(n+2)} \deg f_*\O_X(L). 
	$$
	\end{enumerate}
\end{theoremalpha}

Note that both part \eqref{T:slopeL1i} and \eqref{T:slopeL1ii} reduce to \eqref{E:motivation} if $n=1$ and $L=K_{X/T}$, under the further assumption that the total space is smooth and the general fiber $F$ has genus at least two. However, while the inequality \eqref{E:motivation} is sharp  (see also Remark \ref{R:curves}), we do not know if the above inequalities in Theorem \ref{T:slopeL1} are sharp for $n\geq 2$ (see also Remark \ref{R:slope-surf}).

Moreover, the special cases of part \eqref{T:slopeL1i} and \eqref{T:slopeL1ii} for $n=2$ and $L=K_{X/T}$  were proved by, respectively, Ohno \cite[Prop. 2.1(1)]{Ohno}  and Hu-Zhang \cite[Thm. 1.7]{HZ}, with the further assumption that $X$ has terminal singularities (which implies that $F$ is smooth of general type) but without assuming that $\phi_{K_F}$ is 
generically finite (see Remark \ref{R:slope-surf}). Notice, however, that if $F$ is singular the assumption that $\phi_{K_F}$ is generically finite cannot be dropped, see Example \ref{Ex:Xiao}.

The second type slope inequalities that we prove in this general context are independent of  the numerical invariants of the polarized general fiber $(F,L_F)$.

\begin{theoremalpha}[see Theorem \ref{T:Barja}] \label{T:slopeL2} 
Let $f:X\to T$ be a fibration as in Theorem \ref{T:slopeL1} and let $L$ be a $\Q$-Cartier $\Q$-divisor on $X$. 
Assume that  $L$ and $f_*\O_X(L)$ are nef.
		\begin{enumerate}
		\item  \label{T:slopeL2i}   Assume  that there exists a  $q \in \N_{>0}$ such that  at least one of the following two conditions holds true 
		\begin{itemize}
			\item $\phi_{qL_F}$ is generically finite;
			\item $qL_F$ is Cartier and big.
		\end{itemize}
		Then 
		$$
		L^{n+1} \ge \frac{\deg f_*\O_X(L)}{q^n}.
		$$
		\item \label{T:slopeL2ii} Assume that there exists a $q\in \N_{>0}$ such that $\phi_{qL_F}$ is generically finite, and either $n=\dim F\geq 2$ and $\kappa(F)\geq 0$  or $\dim F=1$ and $\lfloor qL_F\rfloor$ is special. Then
		$$
		L^{n+1} \ge \frac{2\deg f_*\O_X(L)}{q^n}.
		$$
		\end{enumerate}

In particular, if the assumptions of either item (\ref{T:Barja1}) or item (\ref{T:Barja2}) hold and $\deg f_*\O_X(L)>0$, then $L$ is big.
\end{theoremalpha}

The proof of the above Theorem is inspired by \cite[Page 69, Claim]{Barja_Phd} (see however Remark \ref{R:Barja}).
In Examples \ref{Ex:curves} and  \ref{Ex:Pn}, we show that the inequalities in Theorem \ref{T:slopeL2} are sharp, at least for $q=1$.

\vspace{0.1cm}

\subsection*{Plan of the paper }\label{sec:plan}

The paper is organized as follows.

In Section \ref{Sec:Qdiv}, we discuss some technicalities on rational maps associated to $\Q$-divisors (not necessarily integral nor Cartier) on normal or deminormal varieties.

In Section \ref{Sec:Noether}, we establish several Noether type inequalities (Propositions \ref{P:Noether1}, \ref{P:Noether1bis}, \ref{P:bound1} and \ref{P:bound2}) and Castelnuovo type inequalities (Propositions \ref{P:Castelnuovo2} and \ref{P:Castelnuovo3}), which are crucial in proving our slope inequalities  and also interesting in their own (as we believe).  

In Section \ref{Sec:HN} we study the Harder-Narasimhan filtration of $f_*\O_X(L)$ and the properties of the induced chains of sub-divisors of $L$ (see Propositions \ref{P:Mnef} and  \ref{P:Nnef}). Moreover, we prove the numerical  Lemma \ref{L:inequ} (see also Corollary \ref{C:spec-ineq} and Remark \ref{R:choices}) that is used to bound from below  the top self-intersection of $L$.

In Section \ref{Sec:slope} we prove the slope inequalities stated in subsection \ref{sec:slope-L} for an arbitrary $\Q$-Cartier $\Q$-divisor $L$ on the total space of a family of  $n$-dimensional varieties. Note that Theorem \ref{T:slopeL1} is a consequence of Theorems \ref{T:Xiao-higher} and \ref{T:Xiao-n}, which establish some slope inequalities for the relatively globally generated part $M_{\ell}$ of $L$.  We think that this result is interesting on its own.

Section \ref{Sec:KSB} is divided in two subsections: in subsection \ref{sec:slopeKSB}, we apply the results of Section \ref{Sec:slope} to get slope inequalities for the relative log canonical divisor on families which are generically lc or slc, e.g. families of KSB-stable pairs (see Proposition \ref{P:slopepair} and Theorem \ref{T:Nef-Fam}); in subsection \ref{sec:KSBmoduli}, 
we interpret the slope inequalities for families of KSB-stable varieties as the nefness (or nefness away from the boundary) of suitable $\Q$-divisors  on the moduli stack $\M_{n,v}$ of KSB-stable varieties of dimension $n$ and volume $v$ (see Theorems \ref{T:Nef-Away} and \ref{T:Nef-Mod}). We end subsection \ref{sec:KSBmoduli} with some speculations on the structure of the lambda nef cone $\Nef^{\Lambda}(\M_{n,v})$ of $\M_{n,v}$. 

Section \ref{Sec:Fano} is divided in two subsections: in subsection \ref{sec:slopeFano}, we prove the slope inequality for $\Q$-Gorenstein families of anti-canonically polarized pairs with general fiber which is K-stable (see Theorem \ref{T:F_slope1}) or K-polystable (see Theorem \ref{T:F_slope2}); in subsection \ref{sec:Fanomoduli}, we apply the slope inequalities for families of generically K-stable varieties to prove the nefness, away from the strictly K-polystable locus, of some divisors on the moduli stack  $\mathcal{M}_{n,v}^K$ of K-semistable Fano varieties with dimension $n$ and volume $v$ (see Theorem \ref{thm:Fano_moduli}).

In Section \ref{Sec:BS} we survey and make some comments on a positivity notion introduced and studied by M. Barja and L. Stoppino (\cite{BS1,BS2,BS3,Enea}), namely the $f$-positivity (see Definition \ref{D:f-positive}), which is the strongest slope inequality one can hope for (see Proposition \ref{P:convBS}) and that it holds if either the general polarized fiber $(F,L_F)$  is GIT-stable (see Theorem \ref{T:BS-pos}) 
or $f_*\O_X(L)$ is semistable (see Theorem \ref{T:ss-push}). 

In Section \ref{Sec:Examples} we compute the slope of some natural divisors on interesting families of polarized varieties, namely families of varieties of minimal degree and polarized hyperelliptic varieties (see \S\ref{Sec:fam-mindeg}) and families of hypersurfaces in weighted projective spaces (see \S\ref{Sec:hyperweight}), and we show 
that some of our slope inequalities are sharp. 

\subsection*{Acknowledgments} We thank M. Barja, O. Fujino,  Zs. Patakfalvi, A. Petracci, T. Sano, L. Stoppino and R. Svaldi for useful conversations. We also thank the referee for several useful comments.

Giulio Codogni is funded by the MIUR  ``Excellence Department Project'' MATH@TOV, awarded to the Department of Mathematics, University of Rome Tor Vergata, CUP E83C18000100006, and the  PRIN 2017 ``Advances in Moduli Theory and Birational Classification''.

The authors are members of the GNSAGA group of INdAM.

\section*{Conventions}
We always work over an algebraically closed field $k$ of characteristic zero. By variety we mean a reduced scheme of finite type over $k$, not necessarily irreducible.

\section{Some preliminary results on $\Q$-divisors}\label{Sec:Qdiv}

We briefly collect some facts and notations which are standard for Cartier divisors, but slightly less standard for Weil divisors or $\Q$-divisors.

\subsection{Normal case} Given a normal projective variety $W$ and a divisor  (sometimes called Weil divisor or $\Z$-divisor or integral divisor) $D$  we define the coherent sheaf $\O_W(D)$ by 
$$\O_W(D)(U)=\{f\in K(U) \, | \, ((f)+D)_{|U} \geq 0 \} \quad \text{ for any open } U\subseteq X.$$
Note that  $\O_W(D)$  is a rank one reflexive sheaf and it is invertible if and only if $D$ is Cartier. 

The global non-zero sections $H^0(W,\O_W(D))$ (which we will also denote by $H^0(W,D)$) modulo scalars form the complete linear system $|D|$, which can be identified with the projective space of effective divisors linearly equivalent to $D$. Note that the non-Cartier locus of $D$ is always included in the base locus of $|D|$:  indeed if $p$ is a point of $D$ such that there exists a divisor $E$ linearly equivalent to $D$ which does not pass through $p$, the difference $D-E$ is Cartier (as it is the divisor of a rational function),  and since $E$ is trivial around $p$, $D$ is Cartier at $p$.

Whenever $h^0(W,\O_W(D))\geq 1$, we can consider the rational map associated to $D$
\begin{equation}\label{E:phiD}
\begin{aligned}
\phi_D\colon W & \dashrightarrow \mathbb{P} H^0(W,\O_W(D))^{\vee}=:\P\\
p & \mapsto [f\mapsto f(p)]\\
\end{aligned}
\end{equation}

The above definitions can be extended to a $\Q$-divisor $D$ on $W$ by setting $\O_W(D)=\O_W(\lfloor D\rfloor)$,  $|D|=|\lfloor D \rfloor |$ and $\phi_D=\phi_{\lfloor D \rfloor}$, where $\lfloor D\rfloor$ is the round down of $D$. We denote by $\{D\}:=D-\lfloor D\rfloor$ the fractional part of $D$, which is always an effective $\Q$-divisor. 

Assume now that $D$ is a $\Q$-Cartier $\Q$-divisor.
We can extend the rational map $\phi_D$ over the codimension $1$ points, and then take a \emph{resolution of  indeterminacy} $\mu\colon V\to W$ of $\phi_D$, i.e. a birational projective morphism $\mu\colon V\to W$ such that the composition $\wt \phi_D:=\phi_D\circ \mu:V\to \P$ is a regular morphism.

We want to compare the pull-back $\mu^*(D)$,  which is a well-defined $\Q$-Cartier $\Q$-divisor on $V$, with the pull-back $H$ via $\wt \phi_D$ of any hyperplane divisor on  $\P$, which is Cartier and base point free divisor on $V$ (well-defined up to linear equivalence).

\begin{lemma}\label{lem:nakayama}
Keep the above notation. 
\begin{enumerate}[(i)]
\item \label{lem:nakayama1} The natural map $ \O_W(D) \to \mu_*\O_V(\mu^*D)$ is an isomorphism. In particular, we have an isomorphism $\mu^*: H^0(W,\O_W(D)) \xrightarrow{\cong} H^0(V,O_V(\mu^*D))$, which implies that $\wt\phi_D=\phi_{\mu^*D}$. 

\item \label{lem:nakayama2}  We have a decomposition 
$$
|\mu^*D|=|H|+F,
$$
with $F$ an effective divisor. In particular, $\wt\phi_D=\phi_{H}$ and we have that
$$
\mu^*D \sim H  + E, 
$$
 where $E=F+\{\mu^*(D)\}$ is an effective $\Q$-Cartier $\Q$-divisor.

\end{enumerate}
\end{lemma}
\begin{proof}
Part (\ref{lem:nakayama1}) is \cite[Lemma 2.11]{Nak}, so we only prove (\ref{lem:nakayama2}).

By (\ref{lem:nakayama1}) the map associated to $|\mu^*D|$ is a morphism and so we have a decomposition 
$$
|\mu^*D|=|H|+F
$$	
where $F$ is an effective divisor which is the fixed part of $|\mu^*D|$. We conclude that
$$
\mu^*D = \lfloor \mu^*D \rfloor + \{\mu^*D\} \sim H+ F + \{\mu^*D\}= H + E.
$$ 
\end{proof}

\subsection{Deminormal case}
Let $X$ be a deminormal (i.e. $S_2$ and nodal in codimension one)  variety. Let $D$ be $\Q$-divisor on $X$ such that the support of $D$ contains no irreducible component of the conductor. Then there exists a closed subset $Z \subset X$ of codimension at least $2$ such that $X^0= X \setminus Z$ contains only regular and normal crossing points, and $\lfloor D \rfloor_{|X^0} = \lfloor D_{|X^0} \rfloor$ is a Cartier divisor. The sheaf $\O_X(D)= i_*\O_{X^0}(D_{|X^0})$ is reflexive, where $i: X^0 \hookrightarrow X$ is the inclusion (see \cite[Section 5.6]{Kollar_book}), and it induces a map $\phi_D$.

Let $\eta: W \to X$ be the normalisation of $X$, and set $W^0 = W \setminus \eta^{-1}{Z}$. Note that $\eta^{-1}{Z}$ has codimension at least $2$ in $W$, as $\eta$ is finite. We have a natural injection 
$$H^0(X^0,\lfloor D _{|X^0} \rfloor) \hookrightarrow H^0(W^0, \eta^*\lfloor D_{|X^0} \rfloor)$$ 
induced by  pull-back of sections.  

Assume that $D$ is $\Q$-Cartier and let $j: W^0 \hookrightarrow W$ be the inclusion. Then  $\O_W(\eta^*D)= j_*\O_{W^0}(\eta^*(D_{|X^0}))$ (they are both reflexive sheaves and they coincide on $W^0$) and so
$$
H^0(X,D) = H^0(X^0,D_{|X^0})\hookrightarrow H^0(W^0, \eta^* (D_{|X^0}))= H^0(W, \eta^*D).
$$
We conclude that $\phi_D\circ \eta$ factors through $\phi_{\eta^*D}$. In particular, the following holds.

\begin{lemma}\label{L:gen_fin}
If $\phi_D$ is generically finite, then $\phi_{\eta^*D}$ is generically finite.
\end{lemma}

\section{Noether and Castelnuovo inequalities}\label{Sec:Noether}
\subsection{Noether inequalities}
Noether's inequality states that, on a smooth minimal projective surface of general type $S$, one has
$$
K_S^2 \geq 2h^0(S,K_S)-4 \,.
$$
In this section, we prove a number of generalisations of this formula, which we will later apply to the fibers of our families of varieties.

\begin{proposition}[Noether inequality I]\label{P:Noether1}
Let $F$ be a normal irreducible projective variety of dimension $n \ge 2$ such  $\kappa(F)\geq 0$. 
Suppose that we are given
	\begin{enumerate}[(i)]
		\item a nef $\Q$-Cartier $\Q$-divisors $H$ such that $\dim \phi_H(F)=k$ for some $0\le k \le n$; 
		\item nef $\Q$-Cartier $\Q$-divisors $L_{k+1},\ldots, L_n$ such that $\dim \phi_{L_i}(F)\geq i$ (for any $k+1\leq i \leq n$).
	\end{enumerate}
Then 
	$$
	L_n \cdots  L_{k+1} \cdot H^{k} \ge 2h^0(F,\lfloor H \rfloor)-2k \ge 2.
	$$	
\end{proposition}
Recall that the Kodaira dimension $\kappa(F)$ of $F$ is defined as the Kodaira dimension of any projective smooth model of $F$. 
\begin{proof}	
The last inequality follows from the fact that if $\phi_H(F)\subseteq \P(H^0(F,\lfloor H\rfloor)^{\vee})$ has dimension $k$ then it must hold that $h^0(F,\lfloor H\rfloor )\geq k+1$. 
Let us focus now on the first inequality.

First of all, we make the following 

\un{Reduction}: we may assume that $F$ is smooth and $H, L_{k+1}, \ldots, L_n$  are base point free (Cartier) divisors.

 In fact,  take a common resolution  $\pi: F' \to F$ of $\phi_{\pi^*H}$ and  $\phi_{\pi^*L_{i}}$ (for any $k+1\leq i \leq n$) with $F'$ smooth.
 By Lemma \ref{lem:nakayama},   we can write
$$\pi^*H \sim H' + D \quad \text{ and }  \pi^*L_i \sim L_i' + E_i,$$
where $H'$ and $L_i'$ are base point free Cartier divisors and $D$ and $E_i$ are effective $\Q$-divisors, in such a way that 
  $\phi_{H}\circ \pi=\phi_{\pi^*H}=\phi_{H'}$ and $\phi_{L_i}\circ \pi=\phi_{\pi^*L_i}=\phi_{L_i'}$. In particular, we have that 
 \begin{equation}\label{E:samedim}
\dim \phi_{H'}(F')=\dim \phi_H(F)=k  \quad \text{ and } \quad \dim \phi_{L_i'}(F')=\dim \phi_{L_i}(F)\geq i.
 \end{equation}

Moreover, Lemma  \ref{lem:nakayama} gives also that the pull-back via $\pi$ gives an isomorphism
\begin{equation}\label{E:H0H}
\pi^*:H^0(F,\lfloor H\rfloor)\xrightarrow{\cong} H^0(F',H'). 
\end{equation}
Finally, we compute
\begin{align}\label{E:inters-gg}
L_n \cdots  L_{k+1} \cdot H^{k}= & \pi^*L_n \cdots  \pi^*L_{k+1} \cdot \pi^*(H)^{k}=(L_n'+E_n)\cdots (L_{k+1}'+E_{k+1})\cdot (H'+D)^k\geq  \\ \geq & L_n'\cdots L_{k+1}'\cdot (H')^k \nonumber
\end{align}
where the last inequality follows by the fact that $D$ and $E_i$ are effective and $\pi^*H$ and $ \pi^*L_i$ are nef  (see also \cite[Prop 2.3]{BFJ}).
Combining  \eqref{E:samedim}, \eqref{E:H0H} and \eqref{E:inters-gg}, if we prove the result for the base point free Cartier divisors $H'$ and $L_i'$ on $F'$, then the result follows for the nef $\Q$-Cartier $\Q$-divisors $H$ and $L_i$ on $F$. Hence, the reduction is complete.

\smallskip

We now distinguish two cases.

\un{I Case:} $k=n$.

We proceed by induction on $n\geq 2$. The base case $n=2$ is proved by Shin in  \cite[Theorem 2]{Shin}.  Assume that the result is true in dimension $n-1$ and let us prove it in dimension $n$. Let $D\subset F$ be a general divisor of $|H|$.
Observe that $D$ is a smooth irreducible variety of dimension $n-1$ (by Bertini theorem), and the restriction $H_{|D}$ is base point free with the property that  $\dim \phi_{H_{|D}}(D)=n-1$, which  follows from our assumptions on $H$ and the fact that $D$ is general. As $mK_D=(mK_X+mH)|_D$ by adjunction, $mK_F$ is effective for $m>>0$ (since $\kappa(X)\geq 0$ by hypothesis) and $D$ is general in $|H|$, we also have $\kappa(D)\geq 0$. By the induction hypothesis, we have that 
\begin{equation}\label{E:intersHD}
H^n=(H_{|D})^{n-1}\geq 2h^0(D,H_{|D})-2(n-1).
\end{equation}
From the exact sequence 
$$0\to H(-D)=\O_X\xrightarrow{\cdot D} H\to H_{|D}\to 0,$$
we deduce that 
\begin{equation}\label{E:H0-HD}
h^0(D,H_{|D})\geq h^0(F,H)-h^0(F,\O_F)=h^0(F,H)-1.
\end{equation}
We conclude by putting together \eqref{E:intersHD} and \eqref{E:H0-HD}.  

\un{II Case:} $0\leq k <n$. 

Consider the morphism $\phi_H: F \twoheadrightarrow B \subseteq \mathbb P^r\cong \P(H^0(F,H)^{\vee})$, where $r=h^0(F,H)-1$. The image of a complete base point free linear system is a non-degenerate integral variety and hence we have (see e.g. \cite{EH})
\begin{equation}\label{E:degB}
\deg B \ge r- \dim B +1 = h^0(F,H) - k. 
\end{equation}
Let $\phi_H:F \xrightarrow{\psi} \tilde B \xrightarrow{\pi} B$ be the Stein factorisation of $\phi_H$, where $\psi: F \to \tilde B$ has connected fibres and $\pi: \tilde B \to B$ is a finite map and let $G$ be a general fibre of $\psi$.
Note that 
\begin{equation}\label{E:interLH}
L_{n}\cdots L_{k+1}\cdot H^k= 
\deg(\pi) \deg(B)(L_n)_{|G}\cdots (L_{k+1})_{|G}.
\end{equation}
The general fibre $G$ of $\psi$ is  smooth (by Bertini's theorem), 
connected, of dimension $n-k\geq 1$  and,  for any $k+1\leq i \leq n$, the divisor ${L_i}_{|G}$ is  base point free with  $\dim \phi_{{L_i}_{|G}}(G)=\dim \phi_{L_i}(F)-k\geq i-k$, 
as it follows from the assumption that $\dim \phi_{L_i}(F) \geq i$ and the fact that the fibers of $\psi$ cover the entire variety $F$ (and $G$ is a general fiber of $\psi$). 
Moreover, as $G$ is general, it is not contained in the base locus of $m K_F$ with $m$ any fixed integer; hence the hypothesis $\kappa(F)\geq 0$, together with the adjunction formula $mK_D=(mK_X+mH)|_D$, implies that $\kappa(G)\geq 0$.  Now, Lemma  \ref{L:LiG} below implies that 
\begin{equation}\label{E:LF2}
 (L_n)_{|G}\cdots (L_{k+1})_{|G} \ge 2.
 \end{equation}
We conclude by combining \eqref{E:degB}, \eqref{E:interLH} and \eqref{E:LF2}.    
\end{proof}

\begin{lemma}\label{L:LiG}
	Let $G$ be a smooth projective irreducible variety of dimension $m\geq 1$ such that $\kappa(G)\geq 0$. Let  $L_1, \ldots ,L_m$ be  base point free divisors such that $\dim \phi_{L{_i}}(G) \ge i$ for any $i=1, \ldots, m$.  Then we have that 
	 $$
	 L_m\cdots L_1\geq 2.
	 $$
\end{lemma}
\begin{proof}
The proof is by induction on $m \ge 1$. 

If $m=1$, then $\deg L_1 \ge 2$, for otherwise $G$ would be isomorphic to $\mathbb P^1$ via $\phi_{L_1}$, which contradicts the assumption that $\kappa(G) \ge 0$. 

Assume now that $m \ge 2$ and that the statement is true in dimension $m-1$. Let $D$ be a connected component of a general element of $|L_1|$, which is smooth by Bertini's theorem. As $D$ is general, it is not contained in the base locus of $m K_G$ with $m$ any fixed integer; hence the hypothesis $\kappa(G)\geq 0$, together with the adjunction formula $mK_D=(mK_G+mD)|_D$, implies that $\kappa(D)\geq 0$. Moreover, since the elements of $|L_1|$ cover $X$ and  $\dim \phi_{L{_i}}(G) \ge i$, we have that $\dim \phi_{L{_i}_{|D}}(D) \ge i-1$ for any $i=2, \ldots, m-1$.
Hence we can apply  induction to the variety $D$ and the divisors ${L_2}_{|D}, \ldots ,{L_m}_{|D}$ and we get
$$
L_m\cdots L_1=(L_m)_{|D}\cdots (L_2)_{|D}\geq 2.
$$
\end{proof}

\begin{proposition}[Noether inequality Ibis]\label{P:Noether1bis} 
	Let $F$ be a normal irreducible projective variety of dimension $n \ge 1$. Let $k$ and $h$ two natural numbers such that $h+k\leq n$.
	Suppose that we are given
	\begin{enumerate}[(i)]
		\item a nef $\Q$-Cartier $\Q$-divisor $H$ such that $\dim \phi_H(F)=k$;
		\item nef $\Q$-Cartier $\Q$-divisors $L_{k+1},\ldots, L_{k+h}$ such that $\dim \phi_{L_i}(F)\geq i$ for any $k+1\leq i \leq k+h$; 
		\item a nef and big Cartier divisor $M$.

	\end{enumerate}
	Then 
	$$
	M^{n-k-h}\cdot L_{k+h} \cdots  L_{k+1} \cdot H^{k} \ge h^0(F,\lfloor H\rfloor)-k.
	$$	
\end{proposition}
\begin{proof}	
		First of all, with the same proof of the reduction step in Proposition \ref{P:Noether1}, 	we can make the following

	\un{Reduction}: we may assume that $F$ is smooth, $H, L_{k+1}, \ldots, L_{k+h}$    are base point free (Cartier) divisors and $M$ is a big and nef (Cartier) divisor.

	\vspace{0.2cm}

	Consider the morphism $\phi_H: F \twoheadrightarrow B \subseteq \mathbb P^r\cong \P(H^0(F,H)^{\vee})$, where $r=h^0(F,H)-1$. Since the  image of $\phi_H$ is a non-degenerate integral variety, then we have  (see e.g. \cite{EH})
	\begin{equation}\label{E:degBbis}
	\deg B \ge r- \dim B +1 = h^0(F,H) - k. 
	\end{equation}
	Let $\phi_H:F \xrightarrow{\psi} \tilde B \xrightarrow{\pi} B$ be the Stein factorisation of $\phi_H$, where $\psi: F \to \tilde B$ has connected fibres and $\pi: \tilde B \to B$ is a finite map and let $G$ be a general fibre of $\psi$.
	Note that 
	\begin{equation}\label{E:interLHbis}
	M^{n-k-h}\cdot L_{k+h} \cdots  L_{k+1} \cdot H^{k} \geq \deg(\pi) \deg(B)(M_{|G})^{n-k-h}\cdot (L_{k+h})_{|G} \cdots  (L_{k+1})_{|G}.
	\end{equation}
	If $k=n$ then we conclude using \eqref{E:degBbis} and \eqref{E:interLHbis}. If $k<n$ then consider the   general fibre $G$ of $\psi$ which  is  smooth (by Bertini's theorem), 
	connected (and hence irreducible) of dimension $n-k\geq 1$. Using the fact that the fibers of $\psi$ cover the entire variety $F$ (and $G$ is a general fiber of $\psi$), we have that 
	\begin{itemize}
		\item for any $k+1\leq i \leq k+h$, the divisor ${L_i}_{|G}$ is  base point free with  
		$$\dim \phi_{{L_i}_{|G}}(G)=\dim \phi_{L_i}(F)-k\geq i-k,$$ 
		\item $M_{|G}$ is nef and big.
	\end{itemize}
	Now, Lemma  \ref{L:LiGbis} below implies that 
	\begin{equation}\label{E:LF2bis}
	(M_{|B})^{n-k-h}\cdot (L_{k+h})_{|B} \cdots  (L_{k+1})_{|B} \ge 1.
	\end{equation}
	We conclude by combining \eqref{E:degBbis}, \eqref{E:interLHbis} and \eqref{E:LF2bis}.    
\end{proof}

\begin{lemma}\label{L:LiGbis}
	Let $G$ be a smooth projective irreducible variety of dimension $m\geq 0$. Let  $L_1, \ldots ,L_h$ be  base point free divisors, for some $0\leq h \leq m$, such that $\dim \phi_{L{_i}}(G) \ge i$ for any $i=1, \ldots, h$ and let $M$ be a nef and big divisor. Then we have that 
	 $$
	 M^{m-h} \cdot L_h\cdots L_1\geq 1.
	 $$
\end{lemma}
\begin{proof}
The proof is by induction on $m \ge 1$. 

If $m=1$, then we conclude since either $\deg L_1 \ge 1$ (because $\phi_{L_1}$ is generically finite) or $\deg M\geq 1$ (because $M$ is big).

Assume now that $m \ge 2$ and that the statement is true in dimension $m-1$. If $h=0$ then we have 
$$
M^m\geq 1
$$
since $M$ is a nef and big divisor. If $h\geq 1$, then we let $D$ to be a connected component of a general element of $|L_1|$, which is smooth by Bertini's theorem. Moreover, since the elements of $|L_1|$ cover $X$, we will have that  ${L_i}_{|D}$ is base point free with $\dim \phi_{L{_i}_{|D}}(D) \ge i-1$ for any $i=2, \ldots, h$ and that $M_{|D}$ is nef and big. 
Hence we can apply  induction to the variety $D$ and the divisors ${L_2}_{|D}, \ldots ,{L_h}_{|D}, M_{|D}$ and we get
$$
M^{m-h}\cdot L_h\cdots L_1=M_{|D}^{(m-1)-(h-1)}\cdot (L_h)_{|D}\cdots (L_2)_{|D}\geq 1.
$$
\end{proof}

\begin{corollary}\label{C:Noether1}
Let $F$ be a normal irreducible projective variety of dimension $n$ and let $H$  be a nef $\Q$-Cartier $\Q$-divisor on $F$ such that $\phi_H$ is generically finite.
\begin{enumerate}[(i)]
\item \label{C:Noether1i} We have that 
$$H^n\ge h^0(F,\lfloor H\rfloor)-n.$$
\item  \label{C:Noether1ii} If, furthermore, $n=\dim F \ge 2$ and $\kappa(F)\geq 0$, then 
	$$
	H^n\ge 2h^0(F,\lfloor H\rfloor)-2n \ge 2.
	$$
\end{enumerate}		
\end{corollary}
Special cases of the above Corollary for $H$ an integral Cartier divisor are known: part \eqref{C:Noether1i} is classical (see \cite{EH}); part  \eqref{C:Noether1ii} was proved by Kobayashi in \cite[Proposition 2.1]{Kobayashi}  under the stronger assumption that  $p_g(F)>0$ and by Shin in \cite[Theorem 2]{Shin} for $n=2$.
\begin{proof}
Part \eqref{C:Noether1i} follows Proposition \ref{P:Noether1bis} with $k=n$ and $h=0$; part \eqref{C:Noether1ii} follows from Proposition \ref{P:Noether1} with $k=n$.
\end{proof}

\begin{remark}\label{R:sharpNoe}
The inequalities in Corollary \ref{C:Noether1} are sharp and the cases where equalities holds are classified, at least if $H$ is an integral Cartier divisor. Indeed:
\begin{enumerate}[(i)]
\item If $(F,H)$ is a pair as in Corollary  \ref{C:Noether1} with $H$ integral and Cartier for which $H^n=h^0(F,H)-n$ then $H$ is base point free and the image of $F$ under $\phi_H$ is a a non-degenerate normal irreducible $n$-dimensional projective variety $Z\subseteq \P^{h^0(F,L)-1}$ of minimal degree, i.e. $\deg Z=h^0(Z,\O_Z(1))-n$
(see \cite{EH}).
\item Kobayashi proved in \cite[Prop. 2.2]{Kobayashi} that if $(F,H)$ is a pair as in Corollary  \ref{C:Noether1} with $H$ integral and Cartier for which $H^n=2h^0(F,H)-2n$ then $H$ is base point free and one of the following two conditions are satisfied:
\begin{enumerate}[(a)]
\item \label{case-a} $\phi_H$ is birational;
\item \label{case-b} $\phi_H$ is a generically finite double cover of a non-degenerate normal irreducible $n$-dimensional projective variety $Z\subseteq \P^{h^0(F,L)-1}$ of minimal degree. 
\end{enumerate}
Both cases do indeed occur (see \cite[Ex. 2.3]{Kobayashi}): case \eqref{case-a} occurs for example if $F$ is a K3 surface and $H$ is a non-hyperelliptic big and nef divisor (we are not aware of similar examples in higher dimensions); case \eqref{case-b} includes the hyperelliptic polarized  varieties studied by T. Fujita in \cite{Fuj}.  
\end{enumerate}
\end{remark}

\begin{proposition}(Noether inequality II) \label{P:bound1}
	Let $F$ be a normal irreducible projective variety of dimension $n=\dim F \ge 1$. Let $L$ and $M$ two Cartier divisors on $F$ such that 
	\begin{itemize}
	\item $L$ is base point free (hence nef) and $\phi_L$ is generically finite;
	\item $M$ is nef;
	\item $L-M$ is effective and $L^n - L^{n-1}\cdot M \ge 1$.
	\end{itemize}
	Then we have that 
	$$
	L^{n-1}\cdot M \ge 
	\begin{cases} 
	2h^0(F,M) - 2 & \quad \text{ if  } \dim F\geq 2 \: \text{ and } \kappa(F)\geq 0;\\
	h^0(F,M) - 1 & \quad \text{ otherwise. } 
	\end{cases}
		$$
\end{proposition}

\begin{proof}
We first make the following 

\un{Reduction}: we may assume that $F$ is smooth. 

In fact, let  $\pi: F' \to F$ be a resolution of singularities. Then $\pi^*L$ and $\pi^*M$ are nef Cartier divisors on $F'$ such that $\pi^*L$ is base point free and $\pi^*L - \pi^*M$ is effective. Moreover, $\phi_{\pi^*L}$ is generically finite and $h^0(F',\pi^*M)=h^0(F,M)$ (cf.\ the reduction step in the proof of Proposition \ref{P:Noether1}). Since $(\pi^*L)^n=L^n$ and $(\pi^*L)^{n-1}\cdot \pi^*M =  L^{n-1}\cdot M$, the reduction is complete.

The proof is now by induction on $n$. The base cases are $n=2$ for the first case and $n=1$ for the second case. We need to distinguish the base cases from the inductive step. 

\un{Base case for the second inequality: $n=1$} 

If $F$ is a curve, then we have that $h^0(F,M)\leq \deg M+1$ (see e.g. \cite[Exercise IV.1.5]{Har}).

%OLD PROOF 
%We distinguish two cases:
%\begin{itemize}
%\item If $M$ is non-special, then we use the Riemann-Roch theorem  to get
%$$h^0(F,M)=\deg M+1-g(F)\leq \deg M+1,
%$$
%where $g(F)\geq 0$ is the genus of the curve $F$.
%\item If $M$ is special then we use the Clifford theorem (and the fact that  $\deg M\geq 0$ since $M$ is nef)  to get
%$$
%h^0(F,M)-1\leq \frac{\deg M}{2}\leq \deg M.
%$$
%\end{itemize}

\un{Base case for the first inequality: $n=2$} 

Consider a general divisor $D\in |L|$. By Bertini theorem (using that $L$ is base point free and $\dim \phi_L(F)=2>1$), we get that $D$ is a smooth and connected curve. 
From the exact sequence 
$$
	0 \to \O_F(M-D) \xrightarrow{\cdot D} \O_F(M) \to \O_D(M_D) \to 0
	$$
and the fact that $M-D\sim M-L$ is non effective (since $L^{n-1}\cdot (M-L)<0$ and $L$ is nef by assumption), we deduce that 
\begin{equation}\label{E:H0DF}
h^0(D,M_D) \ge h^0(F,M).
\end{equation}
Consider now the divisor $M_D$ on $D$ which has  degree $\deg M_D=M\cdot L\geq 0$.
If $M_D$ is special, then   Clifford's theorem gives
\begin{equation}\label{E:1ineqD}
M \cdot L =\deg M_D \ge 2h^0(D,M_D) -2.
\end{equation}
If $M_D$ is not special, then, using Riemann-Roch and the adjunction formula, we compute
\begin{equation}\label{E:2ineqD}
	h^0(D,M_D)= 1 + L \cdot M - g(D) = 1 + L \cdot M -1- \frac{L^2 + K_F \cdot L}{2}\leq \frac{L\cdot M}{2}=\frac{\deg M_D}{2},
\end{equation}
where in the inequality we used that $K_F\cdot L\geq 0$ since $K_F$ is $\Q$-effective and $L$ is nef, and $L^2=L\cdot (M+(L-M))\geq L\cdot M$ since $L-M\geq 0$ and $L$ is nef.
	
We now conclude using \eqref{E:H0DF} and either \eqref{E:1ineqD} or \eqref{E:2ineqD}.

\un{Inductive step} 

Assume that the statement is true in dimension $n-1$ (which is at least $2$ in the first case and at least $1$ in the second case) and let us prove it in dimension $n$. Take a general element $D \in |L|$, which is a smooth connected variety of dimension $n-1$ by Bertini's theorem  (using that $L$ is base point free and $\dim \phi_L(F)=n>1$).
Since $D$ is general, the restrictions $L_D$ and $M_D$ will satisfy the same assumptions of $L$ and $M$. Moreover, if $\kappa(F)\geq 0$, then $\kappa(D)\geq 0$.
%are nef, $L_D$ is base point free and, since $D$ is general, we have that 
%	$\phi_{L_D}$ is generically finite and $\kappa(D)\geq 0$ (using that $\phi_L$ is generically finite and that $\kappa(F)\geq 0$). 
	Hence, we can apply the induction hypothesis to the line bundles $M_D$ and $L_D$ on $D$ in order to deduce that 
	\begin{equation}\label{E:induH0D}
	L^{n-1} \cdot M = L_D^{n-2} \cdot M_D \ge 
	\begin{cases} 
	2h^0(D,M_D) - 2 & \quad \text{ if  }  \kappa(F)\geq 0;\\
	h^0(D,M_D) - 1 & \quad \text{ otherwise. } 
	\end{cases}
	\end{equation}
	We conclude using this and observing that \eqref{E:H0DF} holds true also in the present case (with the same proof). 
\end{proof}

\begin{remark}\label{R:sharpNoe2}
Both the inequalities in Proposition \ref{P:bound1} are sharp, as we now show for any $n\geq 2$ (for $n=1$ it is obvious). 
\begin{enumerate}[(A)]
\item \label{R:sharpNoe2A} 
Let $F=\P^1\times \P^{n-1}$ (with $n\geq 2$) and denote by $p_1$ and $p_2$ the two projections. Given a divisor $D$ on $\P^1$ of positive degree, set $M:=p_1^*D\leq L:=p_1^*D+p_2^*H$ where $H$ is a hyperplane divisor on $\P^{n-1}$. 
Then $M$ is base point free (and hence nef), $L$ is very ample and we have that 
$$L^{n}=n\deg D>L^{n-1}\cdot M=\deg D=h^0(\P^1,D)-1=h^0(F,M)-1.$$ 

\item \label{R:sharpNoe2B} 
Let $\pi: F\to  \P^1\times \P^{n-1}$ (with $n\geq 2$) be a finite double cover ramified along a smooth divisor in  $|2(m_1p_1^*(p)+m_2p_2^*(H))|$, with $m_1\geq 3$ and $m_2\geq n+1$, where $p$ is a point of $\P^1$ and $H$ is a hyperplane divisor on $\P^{n-1}$ and $p_1$ and $p_2$ are the two projections of $\P^1\times \P^{n-1}$.
Note that 
$$K_F=\pi^*(K_{\P^1\times \P^{n-1}}+m_1p_1^*(p)+m_2p_2^*(H))=(m_1-2)\pi_1^*p+(m_2-n)\pi_2^*H,$$
where $\pi_i=p_i\circ \pi$ for $i=1,2$. This shows that $K_F$ is ample, so that $F$ is a variety of general type. 
  Given a divisor $D$ on $\P^1$ of positive degree, set $M:=\pi_1^*D\leq L:=\pi_1^*D+\pi_2^*H$. 
Then $M$ is base point free (and hence nef), $L$ is  ample and base point free and we have that 
$$L^{n}=2n\deg D>L^{n-1}\cdot M=2\deg D=2[h^0(\P^1,D)-1]=2[h^0(F,M)-1].$$ 
\end{enumerate}
The above two examples will be generalized in Example \ref{Ex:Pn}.
\end{remark}

\vspace{0.1cm}

The following result is not used in the current manuscript as it does not cover the case $L^n-L^{n-1}M=1$, but we believe it is interesting and can be applied in situations similar to the one of this work.

\begin{proposition}(Noether inequality III) \label{P:bound2}
	Let $F$ be a normal projective irreducible variety of dimension $n \ge 2$  with $\kappa(F) \ge 0$.

	Let $L$ and $M$ be nef Cartier divisors on $F$ such that  $|M|$ is base point free. Assume that  $\phi_L$ is generically finite and that $L-M$ is effective.  Then 
	
	$$
	L^{n-1}\cdot M \ge  
	\begin{cases}
	2h^0(F,M) - 2 & \text{ if }  L^n - L^{n-1}\cdot M \ge 2, \\
	2h^0(F,L)-2n  & \text{ if }  L^n - L^{n-1}\cdot M =0.
	\end{cases}
	$$
	
\end{proposition}

\begin{proof}
	We distinguish two cases.
	
	\medskip
	\textbf{First case: $\ L^n - L^{n-1}\cdot  M \ge 2$. } 
	
	Since $L$ and $M$ are nef and $ L\geq  M$,  we have the following inequalities 
	
	\begin{equation}\label{E:LMi-inter}
	L^n \ge  M \cdot  L^{n-1} \ge  M^2 \cdot L^{n-2} \ge \ldots \ge  M^{n-1} \cdot  L\geq  M^n. 
	\end{equation}

	The Hodge index theorem \cite[Prop. 2.5.1]{adjunctiontheory} says that (for any  $k=1, \ldots n-1$)
	$$
	(M^k \cdot  L^{n-k})^2 \ge ( M^{k+1} \cdot  L^{n-k-1})( M^{k-1}\cdot  L^{n-k+1}),
	$$
	or in other words that the intersection numbers in \eqref{E:LMi-inter} form a log-concave sequence.
	
	Let $0\leq c := \dim \phi_{ M}( F)\leq n$.
	
	Note $c=0$ implies $M \sim 0$ and  so
	$$ M\cdot  L^{n-1} =0=2h^0(F,M)-2,$$
	and we are done.  Hence, in what follows, we can assume that $c\geq 1$. 
	
	By applying Proposition \ref{P:Noether1} to $L$ and $M$ (using $n\geq 2$), we get 
	\begin{equation}\label{E:Mici}
	M^{c} \cdot  L^{n-c} \geq 2h^0(F,M) -2c \ge 2.
	\end{equation}
	By applying Lemma \ref{lem:concave} to the log-concave sequence \eqref{E:LMi-inter} truncated up to the term $ M^{c} \cdot L^{n-c}\geq 2$ and using the assumptions $L^n - L^{n-1}\cdot M \ge 2$ and $c\geq 1$, 
	we get 
	\begin{equation}\label{E:Midi}
	M\cdot  L^{n-1} \ge  M^{c} \cdot  L^{n-c}+2(c-1).
	\end{equation}
	We now conclude putting together \eqref{E:Mici} and \eqref{E:Midi}.

	\medskip
	\textbf{Second case: $L^n =  L^{n-1}\cdot M$. }  By applying Proposition \ref{P:Noether1} to $ L$ (using $n\geq 2$), we get
	\begin{equation*}\label{eq:L^n}
	L^{n-1}\cdot  M=  L^n \ge 2h^0(F, L)-2n.
	%\underbrace{\ge}_{r_{i+1}\geq r_i+1} 2r_i -2n +2(\ell -i).
	\end{equation*}
	
\end{proof}

The following elementary lemma about log-concave sequences was used in the above proof. 

\begin{lemma}\label{lem:concave}
	Let 
	$$
	d_n \ge d_{n-1} \ge \ldots \ge d_0 \ge 2
	$$	
	be a sequence of positive integers such that $n \ge 2$ and $d_i^2 \ge d_{i+1}d_{i-1}$ for any $0<i<n$.  If $d_n- d_{n-1} \ge 2$, then 
	$
	d_{i}-d_{i-i} \ge 2 
	$
	for any $1 \le i \le n$.  In particular,
	$$
	d_{n-1}\geq d_0+2(n-1).
	$$
\end{lemma}
Notice that the above Lemma is false without the assumption that $d_0\geq 2$, e.g. $d_2=4>d_1=2>d_0=1$. 
\begin{proof}	
	The proof is by descending induction on $i$. Assume by contradiction that $d_{i}-d_{i-1} \le 1 $ for some $i \in \{1,\ldots,n-1\}$. Then 
	$$
	d_i^2 \ge d_{i+1}d_{i-1} \ge (d_i + 2) (d_i-1)= d_i^2 +d_i -2, 
	$$
	which gives $d_i =2$ (since $d_i \ge 2$ for any $i$). We then have $4 \ge 4d_{i-1}$ and so $d_{i-1} = 1$, which contradicts $d_0 \ge 2$.
\end{proof}

\subsection{Castelnuovo inequalities}

Let $C$ be an irreducible non-degenerate curve in $\mathbb P^N$ of degree $d$. \emph{Catelnuovo inequality} says that

\begin{equation}\label{E:Castelnuovo}
p_g(C) \le \binom{A}{2}(N-1) + A\varepsilon,
\end{equation}
where 
\begin{equation}\label{E:Aepsi}
A=\left\lfloor \frac{d-1}{N-1} \right\rfloor  \quad \text{ and } \quad 0\leq \varepsilon= d-1 - A(N-1)<N-1.
\end{equation}

Building on this classical result, we are going to prove some new inequalities that we will later apply to the fibers of our families.

\begin{lemma}\label{L:Castelnuovo}
	Let $C$ be a smooth irreducible curve and let $L$ be a Cartier divisor on $C$ such that $\phi_L$ is birational. Let $p \in \mathbb N$ such that $K_C - pL \ge 0$. Then 
	$$
	\deg L \ge (p+1)(h^0(C,L) -2) +2.
	$$
\end{lemma}
Note that the above inequality is sharp (at least) in the following cases: (1) if $p=0$, $C$ is an elliptic curve and $L$ is a line bundle on $C$ of degree at least $3$ (so that $\phi_L(C)\subset \P^{\deg L-1}$ is an elliptic normal curve); if $p=1$, $C$ is a non-hyperelliptic curve and 
$L=K_C$; if $p=2$ and $L$ is a theta-characteristic on $C$ such that $g(C)=3[h^0(C,L)-1]$ (such a theta characteristic exists on any hyperelliptic curve $C$ of genus $g$ divisible by $3$ by \cite[p. 191]{Mum}).

\begin{proof}
Consider the map $\phi_L: C \to \mathbb P^N$, where $N=h^0(C,L)-1$. Since $\phi_L$ is birational,  the degree $d$ of $\phi_L(C)$ in $\mathbb P^N$ is at most $\deg L$. 
Hence, it is enough to prove that 
\begin{equation}\label{E:disu-d}
d \ge (p+1)(N-1) +2.
\end{equation}
Using the assumption $K_C - pL \ge 0$ and \eqref{E:Castelnuovo}, we get that 
\begin{equation}\label{E:ineq-d}
 \frac{pd}{2} + 1\leq  \frac{p\deg L}{2}+1\leq  \frac{\deg K_C}{2} +1=p_g(C)\leq \binom{A}{2}(N-1) + A\varepsilon. 
\end{equation}
Substituting $d=A(N-1)+\varepsilon+1$ into \eqref{E:ineq-d}, we arrive at the inequality 
 \begin{equation}\label{E:ineq-A}
p+2\leq A(N-1)(A-1-p)+(2A-p)\varepsilon.
\end{equation}

Assume now by contradiction that \eqref{E:disu-d} does not hold, i.e. that $d < (p+1)(N-1) +2$, which in terms of \eqref{E:Aepsi} is equivalent to   
\begin{equation}\label{E:assurdo-A}
\text{ either } A\leq p \quad \text{ or } A=p+1 \text{ and } \varepsilon=0.
\end{equation}

In the first case $A\leq p $, the inequality \eqref{E:ineq-A}, together with the fact that $2A-p\leq A$ and $0\leq \varepsilon<N-1$, gives that 
$$
p+2\leq A(N-1)(A-1-p)+(2A-p)\varepsilon\leq -A(N-1)+A\varepsilon =A(\varepsilon-N+1)\leq 0, 
$$
which is absurd. 
In the second case $A=p+1 $ and $\varepsilon=0$, the inequality \eqref{E:ineq-A} gives the same absurd 
$p+2\leq 0.$

Hence, the inequality \eqref{E:disu-d} must hold, and we are done. 

\end{proof}

The following inequality generalises the bound obtained by J. Harris in \cite[Page 44]{Harris}.

\begin{lemma}\label{L:Harris}
Let $F$ be a smooth irreducible variety of dimension $n \ge 1$ and $L$ a Cartier divisor on $F$ such that $\phi_L$ is birational. Let $p \in \mathbb N$ such that $K_F- pL \ge 0$. Then
 
\begin{equation}
L^n \ge (n+p)(h^0(F,L) -1 -n)+2.
\end{equation}
\end{lemma}

\begin{proof}
Up to resolving the map $\phi_L$ we can assume that $L$ is base point free. 

Let $C$ be the curve obtained intersecting $n-1$ general element of $|L|$. Then $C$ is a smooth irreducible curve and $K_C= (K_F+ (n-1)L)_{|C}$. In particular $K_C - (p+n-1)L_C \ge 0$. 
Since $\phi_{L_C}$ is birational (because $C$ is general), Lemma \ref{L:Castelnuovo} implies that 
$$
L^n = \deg L_C \ge (n+p)(h^0(C,L_C) -2)+2.
$$	
The conclusion follows from the fact that $h^0(C, L_C) \ge h^0(F,L) - (n-1)$.
\end{proof}

\begin{proposition}\label{P:Castelnuovo2}
Let $F$ be a smooth irreducible projective variety of dimension $n\ge 2$. Let $L$ and $M$ two nef Cartier divisors on $F$ such that 
	\begin{itemize}
		\item $\phi_L$ is birational;
		\item $M$ is base point free and $\dim \phi_M(F) =k <n$;
		\item $K_F - pL \ge 0$ for some $p \in \mathbb N$. 		
	\end{itemize}
	Then we have that 
	$$
	L^{n-k}\cdot M^k \ge (n+p-k+2)(h^0(F,M) -k).
	$$	
\end{proposition}

\begin{proof}
Consider  the morphism $\phi_M: F \twoheadrightarrow B \subseteq \mathbb P^r\cong \P(H^0(F,M)^{\vee})$, where $r=h^0(F,M)-1$. Since the  image of $\phi_M$ is a non-degenerate irreducible variety of dimension $k$, then we have  (see e.g. \cite{EH})
\begin{equation}\label{E:1equaz}
\deg B \ge r -k +1 = h^0(F,M) - k. 
\end{equation}
	
Let $\phi_M:F \xrightarrow{\psi} \tilde B \xrightarrow{\pi} B$ be the Stein factorisation of $\phi_M$, where $\psi: F \to \tilde B$ has connected fibres and $\pi: \tilde B \to B$ is a finite map and let $G$ be a general fibre of $\psi$.
	
Using \eqref{E:1equaz}, we get that 
\begin{equation}\label{E:2equaz}
L^{n-k}\cdot M^k \geq \deg(\pi) \deg(B) \cdot L_{G}^{n-k} \ge  L_{G}^{n-k} (h^0(F,M)-k).
\end{equation}
	
Since $G$ is a smooth irreducible variety of dimension $n-k>0$ with $K_G- pL_G \ge 0$ and $\phi_{L_G}$ is birational, by Lemma \ref{L:Harris} we obtain that 
	
\begin{equation}\label{E:3equaz}
L_{G}^{n-k} \ge (n-k+p)(h^0(G, L_G) -1 -(n-k)) + 2.
\end{equation}
	
Moreover, since $\phi_{L_G}$ is birational (onto its image) and $G$ is not rational because $K_G\geq pL_G\geq 0$, we have 
\begin{equation}\label{E:4equaz}
h^0(G, L_G)\geq n-k+2. 
\end{equation}
	
The conclusion follows by putting together \eqref{E:2equaz}, \eqref{E:3equaz} and \eqref{E:4equaz}.
\end{proof}

\begin{proposition}\label{P:Castelnuovo3}
	Let $F$ be a smooth irreducible projective variety of dimension $n=\dim F \ge 2$. Let $L$ and $M$ be two base point free Cartier divisors on $F$ such that 
	\begin{itemize}
		\item $\phi_L$ is birational;
		\item $L-M$ is effective and $L^n - L^{n-1}\cdot M \ge 1$;
		\item $K_F - pL \ge 0$ for some $p\in \N$.
	\end{itemize}
	Then we have that 
	$$
	L^{n-1}\cdot M \ge (n+p)(h^0(F,M)-2) +2.
	$$
\end{proposition}

\begin{proof}	
If 	$ \dim \phi_M(F) = 0$, then $h^0(F,M)=1$ and the statement is trivial.
	
If $ \dim \phi_M(F) = 1$, then by Proposition \ref{P:Castelnuovo2} we have 
$$
L^{n-1} \cdot M \ge (n+p+1)(h^0(F,M)-1) \ge (n+p)(h^0(F,M)-2) + (n+p) \ge (n+p)(h^0(F,M)-2) + 2.
$$

Assume $ \dim \phi_M(F) \ge 2$ and let $C$ be the smooth irreducible curve obtained as intersection of $n-1$ general elements of $|L|$. 

We show by induction on $n$ that $\phi_{M_C}$ is birational. Assume first $n=2$. In this case $\phi_M$ is generically finite. Let $\{p_1,\dots , p_k\}$ be a generic fiber of $\phi_M$, and suppose that $C$ contains $p_i$ for some $i$. Then $C$ does not contain any other point $p_j$  because $\phi_L$ is birational hence a generic section which vanish at $p_i$ does not vanish at $p_j$. To treat the inductive step, we have to show that $ \dim \phi_M(D) \ge 2$ for a generic section $D$ of $L$. As $\phi_L$ is birational, a generic section of $L$ will intersect properly a generic fiber of $\phi_M$, hence the generic fiber of $\phi_M$ restricted to $D$ will have dimension one less than the generic fiber of $\phi_M$.

Therefore, using that 
$$K_C - (n+p-1)M_C = (K_F+(n-1)L)_{|C} - (n+p-1)M_C \geq (n+p-1)(L_C - M_C)  \ge 0,$$	
we can apply Lemma \ref{L:Castelnuovo} to $M_C$ and conclude that
$$
L^{n-1} \cdot M = \deg M_C \ge (p+n)(h^0(F,M_C)-2) +2.
$$

The result follows from the above inequality and the fact that $h^0(C, M_C)  \ge h^0(F,M)$, which one can prove using inductively the exact sequence 
$$
0 \to \O_F(M-D) \xrightarrow{\cdot D} \O_F(M) \to \O_D(M_D) \to 0
$$
and the fact that $M-D\sim M-L$ is non effective (since $L^{n-1}\cdot (M-L)<0$ and $L$ is nef by assumption).
	
\end{proof}

\section{Harder-Narasimhan filtration}\label{Sec:HN}

Assume we are in the following 

\begin{setup}\label{setup}
\noindent
\begin{itemize}
\item Let $T$ be a smooth projective irreducible $k$-curve, $X$ a normal projective irreducible $k$-variety of dimension $n+1$ and $f\colon X \to T$ a fibration, i.e. a  (projective) morphism with $f_*\O_X=\O_T$. In particular, $f$ is  flat and with connected fibers. 
We denote by $F$ a general fiber of $f$ (i.e. the fiber over a closed point of a conveniently small open subset of $T$), which is a normal  projective irreducible $k$-variety of dimension $n$.  
\item Let $L$ be a $\Q$-Cartier $\Q$-Weil divisor on $X$, and we denote by $\O_X(L)$ its associated reflexive sheaf as discussed in Section \ref{Sec:Qdiv}. We will always assume that  the sheaf $f_*\O_X(L)$ (which is always locally free  since $T$ is a smooth curve) is non zero. 
\end{itemize}
\end{setup}

Consider the Harder-Narasimhan(=HN) filtration of $f_*\O_X(L)$:
\begin{equation}\label{E:HNfilt}
0=\E_0\subsetneq \E_1\subsetneq \E_2\subsetneq \ldots \subsetneq \E_{\ell}=f_*\O_X(L),
\end{equation}
where $\ell\geq 1$ is the length of the filtration. Note that $\ell=1$ if and only if $f_*\O_X(L)$ is semi-stable.

For any $1\leq i \leq \ell$,  we denote by $\mu_i:=\mu(\E_i/\E_{i-1})\in \Q$  the slope of the semistable locally free sheaf $\E_{i}/\E_{i-1}$ and by $r_i:=\rk(\E_i)\in \N$ the rank of $\E_i$. 

By definition of the HN filtration, we have
\begin{equation}\label{E:ai}
\mu_+(f_*\O_X(L)):=\mu_1>\ldots >\mu_{\ell}:=\mu_-(f_*\O_X(L))
\end{equation}
\begin{equation}\label{E:ri}
0=:r_0<r_1<\ldots <r_{\ell}.
\end{equation}

By Hartshorne's theorem on the characterisation of nef vector bundles on smooth, projective irreducible curves \cite[Thm. 6.4.15]{Laz2} (which requires $\char(k)=0$), it follows that 
\begin{equation}\label{E:Nef-HN}
 \E_i \text{ is nef } \Leftrightarrow \mu_i\geq 0.
\end{equation}

By generic base change, the rank of $f_*\O_X(L)$ is equal to 
\begin{equation}\label{E:rk-push}
h^0(F,\O_X(L)_{|F})=\rk f_*\O_X(L)=r_{\ell}.
\end{equation}

\begin{lemma}\label{L:Weil}
For a general fiber $F$, we have
\begin{equation}\label{E:LF}
\O_X(L)_{|F}=\O_F(L_{F}). 
\end{equation}
where $F$ is a general fibre of $f$ and $L_F$ is the restriction of $L$ to $F$ as $\Q$-Cartier divisor.
\end{lemma}
\begin{proof}
Both reflexive sheaves and divisors are determined by their restrictions to codimension one points, hence we can prove the statement after removing a codimension two subscheme from $X$. As the total space $X$ is normal, its singularities are in codimension two and we can therefore assume that both $X$ and the general fibers are smooth. After this reduction, it is enough to show that restricting $L$ to  a general fiber commutes with taking the round down. Write $L=\sum a_i D_i$, where $a_i$ are rational numbers and $D_i$ are prime divisors.  By Bertini Theorem, the restriction of every $D_i$ to a general fiber is reduced, hence restriction commutes with round down.
\end{proof}

The degree of $f_*\O_X(L)$ is determined by the numbers $\mu_i$'s and $r_i$'s as in the following 
\begin{lemma}\label{formuletta}
With the notation as above (and the convention that $\mu_{\ell+1}=0$), we have that
 $$\sum_{i=1}^{\ell} r_i(\mu_i-\mu_{i+1})=\sum_{i=1}^{\ell}\mu_i(r_i-r_{i-1})=\deg f_*\O_X(L).$$
\end{lemma}
\begin{proof}
The first equality is just a rearrangement of the terms using that $r_0=0$ and $\mu_{\ell+1}=0$. 

The second equality follows from the fact that $\mu_i(r_i-r_{i-1})$ is  the degree of $\E_i/\E_{i-1}$ and that 
$$
\deg f_*\O_X(L)=\sum_{i=1}^{\ell}\deg(\E_i/\E_{i-1}).
$$
\end{proof}

Variants of the following construction have appeared in many papers, for instance, under the assumption that $L$ is Weil and $\Q$-Cartier, it is discussed in \cite[Lemma 1.1]{Ohno}. For any $1\leq i\leq \ell$, the morphism $f^*\E_i\hookrightarrow f^*f_*\O_X(L)\to \O_X(L)$ induces a rational map $\psi_i:X\dashrightarrow \P_T(\E_i)$ over $T$. By Hironaka's theorem on resolution of singularities  (since $\char(k)=0$), we can pick a birational morphism $\mu:\wt X\to X$ with $\wt X$ a smooth projective irreducible $k$-variety in such a way that 
$\wt{\psi_i}:=\psi_i\circ \mu:\wt X\rightarrow \P_T(\E_i)$ is a regular morphism. On $\wt X$, we define the Cartier divisor
$$M_i:=\wt\psi_i^*L_{\E_i}\,,$$ 
 where $L_{\E_i}$ is any tautological divisor on $\P_T(\E_i)$, i.e. any divisor such that $\O_{\P_T(\E_i)}(L_{\E_i})=\O_{\P_T(\E_i)}(1)$.  
 
As the sheaf $\O_X(L)$ equals $\O_X(\lfloor L \rfloor )$, the divisor $M_i$ depends only on the round down $\lfloor L\rfloor $. If this round down is Cartier, then $M_{\ell}$ is the relative free part of the liner system; if it is just Weil, then $M_{\ell}$ is some Cartier divisor smaller than $\lfloor \mu^*L \rfloor$. 

The inclusion $\E_i\subset  \E_{i+1}$ implies that $M_{i+1}-M_i$ is effective. To summarise, we have a non-decreasing chain of divisors
 \begin{equation}\label{E:Mi}
M_1\leq M_2\leq \ldots \leq M_{\ell}\leq \lfloor \mu^*L \rfloor \leq \mu^*(L) \,,
\end{equation}
and a non-increasing chain of effective integral divisors
\begin{equation}\label{E:Zi}
Z_1\geq Z_2\geq \ldots \geq Z_{\ell}\geq 0.
\end{equation}
such that
\begin{equation}\label{E:MiZi}
M_i\sim_{\Q} \mu^*(L)-Z_i \qquad \textrm{for every} \; i=1, \dots , \ell. 
\end{equation}

\begin{remark}[Relative base loci]
Each piece $\E_i$ of the Harder-Narasimhan filtration of $f_*\O_X(L)$ defines a relative base locus $B_i$. The resolution $\wt X$ makes this base loci divisorial. During the proofs of Section \ref{Sec:slope}, we will intersect this divisorial base loci with nef line bundles, and then just discard them (these computation are often carried out using Lemma \ref{L:inequ} and its corollaries). It would be interesting to study the features of these relative base loci, and let them playing a more prominent role in the slope inequality via asymptotic invariants similar to the $\mu$-invariant introduced in \cite[Definition 4.1]{Xu_Zhuang}.
\end{remark}

We will denote by $\wt f:=f\circ \mu:\wt X\to T$ the induced fibration. A general fiber  $\wt F$  of $\wt f$ is a smooth projective irreducible $k$-variety of dimension $n$, and it is endowed with a birational (projective) morphism $\mu_{\wt F}:=\mu_{|\wt F}:\wt F\to F$ onto a general fiber of $f$, which is a resolution of singularities. 

%\begin{remark}
%Assuming $L$ Cartier, the above construction become slightly more transparent, and  \cite[Lemma 1.1]{Ohno} is not needed: the morphism $\mu$ is the blow-up of the relative base loci of the relative linear systems $\mathcal{E}_i$, 
% followed by a resolution of singularities; the divisors $M_i$ and $Z_i$ are respectively the relative movable and fixed parts of  $\mathcal{E}_i$. In this case, also the proof of Proposition \ref{P:Mnef} can be slightly simplified.
%\end{remark}

We will consider the Cartier divisors $P_i:=(M_i)_{|\wt F}$ on $\wt F$ (well-defined up to linear equivalence), which, by \eqref{E:Mi}, form a non-decreasing chain 
 \begin{equation}\label{E:Pi}
P_1\leq P_2\leq \ldots \leq P_{\ell}\leq \mu^*(L)_{|\wt F}=\mu_{\wt F}^*(L_F).
\end{equation}

We collect the properties of the divisors $M_i$ and their restrictions $P_i$ in the following 

\begin{proposition}\label{P:Mnef} 
Let $1\leq i \leq \ell$.
\begin{enumerate}[(i)]
\item \label{P:Mnef1} The divisor $M_i$ is $\wt f$-globally generated. In particular, $P_i$ is globally generated. 
\item \label{P:Mnef2} The restriction of $\E_i$ to $t$ induces a sub-linear series of $|P_i|$ on $\wt F$; in particular we have that $h^0(\wt F,P_i)\geq r_i.$
%If moreover $\mu_i\geq 2g$, where $g$ is the genus of $T$, then $h^0(\wt F,(M_i)_{\wt F})= r_i$.
\item  \label{P:Mnef3} If $\E_i$ is nef,  then $M_i$ is nef.
 \item \label{P:Mnef4} The pull-back along $\mu_{\wt F}$ induces an isomorphism 
\begin{equation}\label{E:sameH0} 
\mu_{\wt F}^*:H^0(F, L_{F}) \xrightarrow{\cong} H^0(\wt F,P_{\ell}),
\end{equation}
In particular, $h^0(\wt F,P_{\ell})= r_{\ell}$ and  $\phi_{P_{\ell}}=\phi_{L_{F}}\circ \mu_{\wt F}$. 
  \item \label{P:Mnef5}
If $L_F$ is Cartier and  globally generated, then $P_{\ell}= (\mu_{\wt F})^*(L_F)$ (up to linear equivalence). In particular, $P_{\ell}^n=L_F^n$.
%In this case, we have that $h^0(\wt F,P_l)= r_l$ and  $\phi_{P_l}=\phi_{L_F}\circ \mu_{\wt F}$. 
% \Filippo{Conjecture: questi due fatti dovrebbero valere sempre Questi ultimi due fatti valgono anche se non si assume che $L_F$ sia Cartier ne' globalmente generato?}
\end{enumerate}
\end{proposition}
\begin{proof}
By the definition of $M_i$, we have the following commutative diagram
\begin{equation}\label{E:diagpsi}
\begin{tikzcd}
 \wt X\arrow{rr}{\wt \psi_i}  \arrow{dr}[swap]{\wt f}& & \P_T(\E_i)  \arrow{dl}{p_{\E_i}}\\
 & T & 
\end{tikzcd}
\end{equation}
and $\O_{\wt X}(M_i)=\wt \psi_i^*\O_{\P_T(\E_i)}(1)= \wt \psi_i^*\O_{\P_T(\E_i)}(L_{\E_i})$. 

Part \eqref{P:Mnef1} follows from the fact that $\O_{\P_T(\E_i)}(1)$ is $p_{\E_i}$-globally generated. 

Part \eqref{P:Mnef2}: by the definition of $M_i$, it follows that we have an inclusion of torsion-free coherent sheaves on $\P_T(\E_i)$:
$$
\O_{\P_T(\E_i)}(1)\hookrightarrow (\wt\psi_i)_*(\wt \psi_i^*(\O_{\P_T(\E_i)}(1)))=(\wt\psi_i)_*(\O_{\wt X}(M_i)). 
$$
By taking the push-forward via $p_{\E_i}$ we get the inclusion of locally free sheaves on $T$
\begin{equation}\label{E:inclEi}
\E_i=(p_{\E_i})_*(\O_{\P_T(\E_i)}(1))\hookrightarrow (p_{\E_i})_*((\wt\psi_i)_*(M_i))=\wt f_*(\O_{\wt X}(M_i)). 
\end{equation}
By taking ranks, we get 
$$	
r_i:=\rk(\E_i)\leq \rk(\wt f_*\O_{\wt X}(M_i))= h^0(\wt F,P_i).  
$$

Part \eqref{P:Mnef3}: by the definition of the $M_i$, we have a surjection of locally free sheaves on $\wt X$
$$
\wt f^* \E_i\twoheadrightarrow \O_{\wt X}(M_i),
$$
from which the conclusion follows.

Part \eqref{P:Mnef4}: since $M_{\ell}\sim_{\Q} \mu^*(L)-Z_{\ell}$ and $Z_{\ell}\geq 0$ by \eqref{E:MiZi} and \eqref{E:Zi},  we have an injection 
$$
\O_{\wt X}(M_{\ell})\hookrightarrow \O_{\wt X}(\mu^*(L)).
$$

By taking the pushforward along $\wt f=f\circ \mu$, we get 
\begin{equation}\label{E:inclEl}
\wt f_*\O_{\wt X}(M_{\ell})\hookrightarrow \wt f_*\O_{\wt X}(\mu^*(L))=f_*\O_X(L),
\end{equation}
where in the last equality we used that $\mu_*\O_{\wt X}(\mu^*(L))=\O_X(L)$, which follows from  \cite[Lemma 2.11]{Nak}.

Recalling that $\E_{\ell}=f_*\O_X(L)$, by combining \eqref{E:inclEi} and \eqref{E:inclEl} we deduce that 
$$
\wt f_*\O_{\wt X}(M_{\ell})= f_*\O_X(L).
$$
By generic base change and \eqref{E:LF}, we conclude that we have the isomorphism \eqref{E:sameH0}, which implies  the last two assertions. 

Part \eqref{P:Mnef5}: by assumption, and using the generic base-change and Lemma \ref{L:Weil}, one has that the evaluation morphism $f^*\E_{\ell}=f^*f_*\O_X(L)\to \O_X(L)$ is surjective over $F$ and $\O_X(L)_{|F}=\O_F(L_F)$ is line bundle. This implies that the induced 
rational map $\psi_{\ell}:X\dashrightarrow \P_T(\E_{\ell})$ is regular over $F$, and hence that (up to linear equivalence)
$$
P_{\ell}=(M_{\ell})_{|\wt F}=\wt \psi_{\ell}^*(L_{\E_{\ell}})_{|\wt F}=\mu_{\wt F}^*(({\psi_{\ell}}_{|F})^*(L_{\E_{\ell}}))= \mu_{\wt F}^*(L_F).
$$ 
The last assertion follows from the projection formula. 
 \end{proof}

We now want to show that the slopes of the HN filtration of $f_*\O_X(L)$ bound the nefness threshold of $M_i$ with respect to a general fiber $\wt F$. More precisely, consider the ($\Q$-Cartier) $\Q$-divisors $N_i:=M_i-\mu_i\wt F$ (for $1\leq i \leq \ell$) on $\wt X$. By \eqref{E:ai} and \eqref{E:Mi}, the $\Q$-line bundles $N_i$ form a non-decreasing chain 
 \begin{equation}\label{E:Ni}
N_1\leq N_2\leq \ldots \leq N_{\ell}.
\end{equation}
Note that \eqref{E:Nef-HN} implies that
 \begin{equation}\label{E:NiMi}
 \E_i \text{ is nef } \Leftrightarrow N_i\leq M_i.
 \end{equation}

Note that $(N_i)_{|\wt F}=(M_i)_{|\wt F}=P_i$ (up to linear equivalence). 

\begin{proposition}\label{P:Nnef}
For any $1\leq i \leq \ell$, the $\Q$-divisor $N_i$ is nef.
\end{proposition}
\begin{proof}
Observe that, up to linear equivalence,  $N_i$ is the pull-back via $\wt{\psi_i}$ of the $\Q$-divisor $L_{\E_i}-\mu_i p_{\E_i}^{-1}(t)$ on $\P_T(\E_i)$, where $t$ is a general point of $T$ (see diagram \eqref{E:diagpsi}).
Hence it is enough to show that $L_{\E_i}-\mu_i p_{\E_i}^{-1}(t)$ is nef on $\P_T(\E_i)$.

This follows from the Miyaoka's lemma (see e.g. \cite[Lemma 2.1]{Ful}). For the reader's convenience, we also include a direct elementary proof.

Let $\tau \colon T'\to T$ be finite cover whose degree is a multiple of $\rk(\E_i/\E_{i-1})$. By e.g. \cite[Lemma 6.4.12]{Laz2}, the HN filtration of $\tau^*\E_i$ is the pull-back of the HN filtration of $\E_i$, i.e.
$$
0\subsetneq \E_0\subsetneq \E_1\subsetneq \ldots \subsetneq \E_{i-1}\subsetneq \E_i, 
$$
and the slopes get multiplied by $\deg(\tau)$. In particular, using that the nefness of a $\Q$-divisor can be checked after a finite cover, we can assume, up to replacing $T$ with $T'$, that $\mu_i=\mu(\E_i/\E_{i-1})=\mu_-(\E_i)$ is an integer. 
The divisor $L_{\E_i}-\mu_i p_{\E_i}^{-1}(t)$ is now Cartier, and the corresponding sheaf is a quotient of the locally free sheaf $p_{\E_i}^* \E_i(-\mu_i t)$. Since 
$$\mu_-(\E_i(-\mu_i t))=\mu_-(\E_i)-\mu_i=0,$$ 
the sheaf $\E_i(-\mu_i t)$ is nef on $T$ by Hartshorne's theorem  (\cite[Thm. 6.4.15]{Laz2}). We conclude that also $L_{\E_i}-\mu_i p_{\E_i}^{-1}(t)$ is nef .
\end{proof}

We now prove some numerical inequalities that will be crucial in what follows.

\begin{lemma}\label{L:inequ}
Fix the above notation. Assume that we have chosen a nef $\Q$-divisor $N_{\ell+1}$ on $\wt X$ and a rational number $\mu_{\ell+1}$ such that $Z_{\ell+1}:=\mu^*(L)-N_{\ell+1}-\mu_{\ell+1} \wt F\leq Z_{\ell}$.
Set $P_{\ell+1}:=(N_{\ell+1})_{|\wt F}$. Then for any $q \in \{1,\ldots,\ell\}$ and any two sequences of integers
$$ \begin{sis} 
& 1\leq s_1<\ldots <s_q<s_{q+1}=\ell+1, \\ 
%\text{ with } 1\leq q\leq \ell, \\
& 1=m_0\leq m_1\leq \ldots \leq m_{n}\leq m_{n+1}=q+1,
\end{sis}
$$
we have that 
$$
N_{\ell+1}^{n+1}\geq \sum_{i=0}^{n}\sum_{j=m_i}^{m_{i+1}-1} \left(\sum_{k=0}^{i} P_{s_j}^k P_{s_{j+1}}^{i-k} \right)\left(P_{s_{m_{i+1}}}\ldots P_{s_{m_{n}}}\right)(\mu_{s_j}-\mu_{s_{j+1}}).
$$
\end{lemma}
In the above result, if $m_{i+1}-1<m_i$ and hence $j$ belongs to the empty set, by convention the sum is zero. Note that Lemma \ref{L:inequ} is a generalisation of \cite[Lemma 2.2]{Konno} and \cite[Prop. 1.11]{Barja_Phd}, which in turn build on \cite[Lemma 2]{Xiao}.
\begin{proof}
We are going to use several times that the $\Q$-divisors $\{N_h\}_{h=1}^{\ell+1}$ are nef, which follow from Proposition \ref{P:Nnef} if $1\leq h\leq \ell$ and from the assumption on $N_{\ell+1}$ if $h=\ell+1$.  

The Lemma is obtained by summing  the following numerical inequalities for any $0\leq i \leq n$
\begin{equation}\label{E:ineq-i}
\begin{aligned}
\left(N_{s_{m_{i+1}}}^{i+1}-N_{s_{m_i}}^{i+1}\right)\cdot \left(N_{s_{m_{i+1}}}\ldots N_{s_{m_{n}}}\right)( & \geq  \sum_{j=m_i}^{m_{i+1}-1} \wt F\cdot  \left(\sum_{k=0}^{i} N_{s_j}^k N_{s_{j+1}}^{i-k} \right)\cdot \left(N_{s_{m_{i+1}}}\ldots N_{s_{m_{n}}}\right)(\mu_{s_j}-\mu_{s_{j+1}})=\\
& = \sum_{j=m_i}^{m_{i+1}-1} \left(\sum_{k=0}^{i} P_{s_j}^k P_{s_{j+1}}^{i-k} \right)\cdot \left(P_{s_{m_{i+1}}}\ldots P_{s_{m_{n}}}\right)(\mu_{s_j}-\mu_{s_{j+1}}), 
\end{aligned}
\end{equation}
and using that $N_{s_{m_0}}\ldots N_{s_{m_{n}}}\geq 0$ because the divisors $\{N_h\}$ are nef. 

The inequality \eqref{E:ineq-i} follows, using  that $\{N_{s_{m_{i+1}}},\ldots, N_{s_{m_{n}}}\}$ are nef, by the following inequality of $(i+1)$-codimension cycles on $\wt X$
\begin{equation}\label{E:ineq-cyc}
\left(N_{s_{m_{i+1}}}^{i+1}-N_{s_{m_i}}^{i+1}\right)\geq  \sum_{j=m_i}^{m_{i+1}-1} \wt F\cdot  \left(\sum_{k=0}^{i} N_{s_j}^k N_{s_{j+1}}^{i-k} \right)(\mu_{s_j}-\mu_{s_{j+1}})
\end{equation}
The previous inequality \eqref{E:ineq-cyc} follows in turn by summing  the following inequalities of $(i+1)$-codimension cycles on $\wt X$ for 
$m_i\leq j \leq m_{i+1}-1$ 
\begin{equation}\label{E:ineq-j}
N_{s_{j+1}}^{i+1}-N_{s_j}^{i+1}\geq \wt F\cdot  \left(\sum_{k=0}^{i} N_{s_j}^k N_{s_{j+1}}^{i-k} \right)(\mu_{s_j}-\mu_{s_{j+1}}).
\end{equation}

In order to prove inequality \eqref{E:ineq-j}, we write 
\begin{equation}\label{E:potenzeN}
N_{s_{j+1}}^{i+1}-N_{s_j}^{i+1}=(N_{s_{j+1}}-N_{s_j})\cdot  \left(\sum_{k=0}^{i} N_{s_j}^k N_{s_{j+1}}^{i-k} \right).
\end{equation}
Observe that we have
\begin{equation}\label{E:Nh-L}
N_h\sim_{\Q} \mu^*(L)-Z_h-\mu_h \wt F \quad \text{ for any } 1\leq h \leq \ell+1,
\end{equation}
which follows from the definition $N_h:=M_h-\mu_h\wt F$ and  \eqref{E:MiZi} if $1\leq h \leq \ell$, and from the definition of $Z_{\ell+1}$ if $h=\ell+1$. 
By taking the differences of the relations \eqref{E:Nh-L}  for $h=s_{j+1}$ and for $h=s_j$,  we get that 
\begin{equation}\label{E:diffN}
N_{s_{j+1}}-N_{s_j}\sim_{\Q} (Z_{s_j}-Z_{s_{j+1}})+(\mu_{s_j}-\mu_{s_{j+1}})\wt F.
\end{equation}
Combining \eqref{E:potenzeN} and \eqref{E:diffN} and using that $Z_{s_j}\geq Z_{s_{j+1}}$ (by \eqref{E:Zi}, and the assumption $Z_{\ell+1}\leq Z_{\ell}$) and that  $\{N_{s_j},N_{s_{j+1}}\}$ are nef, we get the inequality of cycles \eqref{E:ineq-j}, and we are done. 
\end{proof}

We now collect in the following result some special cases of the above Lemma.

\begin{corollary}\label{C:spec-ineq}
Notation as in Lemma \ref{L:inequ}. 

\begin{enumerate}
\item \label{C:spec-ineq1} For any sequence of integers $1=m_0\leq m_1\leq \ldots \leq m_{n}\leq m_{n+1}=\ell+1$,  we have that 
$$
N_{\ell+1}^{n+1}\geq \sum_{i=0}^{n}\sum_{j=m_i}^{m_{i+1}-1} \left(\sum_{k=0}^{i} P_{j}^k P_{j+1}^{i-k} \right) \left(P_{m_{i+1}}\ldots P_{m_{n}}\right)(\mu_{j}-\mu_{j+1}).
$$
In particular, we have that 
\begin{enumerate}[(A)]
\item \label{C:spec-ineq1A}
$
N_{\ell+1}^{n+1}\geq \sum_{j=1}^{\ell}(P_j+P_{j+1})P_{\ell+1}^{n-1}(\mu_{j}-\mu_{j+1});
$
\item \label{C:spec-ineq1B}
$
N_{\ell+1}^{n+1}\geq \sum_{j=1}^{\ell} \left(\sum_{k=0}^{n} P_{j}^k P_{j+1}^{n-k} \right)(\mu_{j}-\mu_{j+1}).
$
\end{enumerate}
\item \label{C:spec-ineq2} For any sequence of integers $ 1\leq s_1<\ldots <s_q<s_{q+1}=\ell+1$  with  $1\leq q\leq \ell$, we have that 
$$
N_{\ell+1}^{n+1}\geq \sum_{j=1}^{q}(P_{s_j}+P_{s_{j+1}})P_{\ell+1}^{n-1}(\mu_{s_j}-\mu_{s_{j+1}}).
$$
\end{enumerate}
\end{corollary}
Note that  \eqref{C:spec-ineq2} is a special case of \cite[Lemma 1.2]{Ohno} (which generalizes \cite[Lemma 2]{Xiao} from $n=1$ to an arbitrary $n\geq 1$). 
\begin{proof}
Part \eqref{C:spec-ineq1} follows from Lemma \ref{L:inequ} by setting $q=\ell$, which then forces $s_j=j$ for every $1\leq j \leq \ell+1$.
Part  \eqref{C:spec-ineq1A} follows from  \eqref{C:spec-ineq1} by setting $1=m_0=m_1< m_2=\ldots=m_{n+1}=\ell+1$, while part  \eqref{C:spec-ineq1B} follows from  \eqref{C:spec-ineq1} by setting
$1=m_0=\ldots=m_{n}<m_{n+1}=\ell+1$.
 
Part \eqref{C:spec-ineq2} follows from Lemma \ref{L:inequ} by setting $1=m_0= m_1<m_2=\ldots=m_n=q+1$. 
\end{proof}

\begin{remark}\label{R:choices}
The inequalities in Lemma \ref{L:inequ} and Corollary \ref{C:spec-ineq} depend upon the choice of a nef $\Q$-divisor $N_{\ell+1}$ on $\wt X$ and a rational number $\mu_{\ell+1}$ subject to the condition 
$Z_{\ell+1}:=\mu^*(L)-N_{\ell+1}-\mu_{\ell+1} \wt F\leq Z_{\ell}$.
Some natural choices of $(N_{\ell+1}, \mu_{\ell+1})$ are as follows:
\begin{enumerate}[(i)]
\item $N_{\ell+1}:=N_{\ell}$ (which is nef by Proposition \ref{P:Nnef}) and $\mu_{\ell+1}= \mu_{\ell}$, which implies that $Z_{\ell+1}=Z_{\ell}$;
\item under the assumption that $L$ is nef: $N_{\ell+1}:=\mu^*(L)$  and $\mu_{\ell+1}= 0$, which implies that $Z_{\ell+1}=0$;
\item under the assumption that $f_*\O_X(L)$ is nef: $N_{\ell+1}:=M_{\ell}$ (which is nef by Proposition \ref{P:Mnef}\eqref{P:Mnef3}) and $\mu_{\ell+1}=0$, which implies that $Z_{\ell+1}=Z_{\ell}$. 
\end{enumerate}
%In this paper, we use only the second and third choice, never the first one.
\end{remark}

Many of our slope inequalities will depend on the nefness of $L$, together with the nefness of $f_*\O_X(L)$ (see the previous Remark \ref{R:choices}). In the following result, we give a criterion that guarantees the nefness of $L$ together with a numerical consequence of the nefness of $L$ and of $f_*\O_X(L)$. 

\begin{lemma}\label{L:Lnef}
Assume that $f_*\O_X(L)$ is nef.
\begin{enumerate}[(i)]
\item \label{L:Lnef1}
Assume that $L$ is Cartier, $f$-nef  and generically $f$-globally generated (i.e. $L_F$ is globally generated), then $L$ is nef.
\item \label{L:Lnef2}
If $L$ is nef, then $L^{n+1}\geq M_{\ell}^{n+1}$. 
\end{enumerate}
\end{lemma}
\begin{proof}
Part \eqref{L:Lnef1}: as $L$ is $f$-nef, it is enough to show that given a horizontal integral curve $C$, we have $L\cdot C\geq 0$. Let $p$ be the restriction of $f$ to $C$. As $L$ is generically $f$-globally generated, the evaluation map
$$\ev_C\colon p^*f_*\O_X(L)\to \O_X(L)_{|C} $$
is generically surjective, so we can write $L|_C=Q+E$, where $\O_C(Q)$ is a quotient of $p^*f_*\O_X(L)$ and $E$ is  effective. As $f_*\O_X(L)$ is nef, $Q$ is nef and hence $L|_C$ is nef, i.e.  $L\cdot C\geq 0$.

Part \eqref{L:Lnef2}: first of all, note that $L^{n+1}=(\mu^*L)^{n+1}$.  As $\mu^* L\sim_{\Q} M_{\ell}+Z_{\ell}$ with $Z_{\ell}$ effective by \eqref{E:MiZi} and \eqref{E:Zi}, $L$ (and hence also $\mu^*L$) is nef by assumption and $M_{\ell}$ is nef by Proposition \ref{P:Mnef}\eqref{P:Mnef3} (using the assumption that $f_*\O_X(L)$ is nef), we have that (see also \cite[Prop 2.3]{BFJ})
$$(\mu^*L)^i\cdot M_{\ell}^{n+1-i}=(\mu^*L)^{i-1}\cdot (M_{\ell}+Z_l)\cdot M_{\ell}^{n+1-i}\geq (\mu^* L)^{i-1}\cdot M_{\ell}^{n+2-i} \text{ for every } 1\leq i \leq n+1.$$
We conclude by putting together all the above inequalities for every $1\leq i \leq n+1$.

\end{proof}

\section{Slope inequalities}\label{Sec:slope}

In this section, we assume that we are in the set-up  \eqref{setup}, and our goal is to prove some slope inequalities, i.e. inequalities of the form
$$L^{n+1} \geq C \deg f_*\O_X(L) \,,$$ 
for some positive constant $C$, which depends just on the polarized general fiber of $f$. In all our results, we will need to assume that $f_*\O_X(L)$ and $L$ are both nef, so both $L^{n+1}$ and  $\deg f_*\O_X(L)$ are non-negative.

 Moreover, depending on the different slope inequalities that we get, we will need to make some extra assumptions that can be of two types. 
 
 The first kind of assumptions concerns  the $q$-th multiple of the $\Q$-Cartier divisor $L_F$ on $F$  (for some integer $q\geq 1$) and its associated rational morphism  $\phi_{qL_{F}}:F\dashrightarrow \P(H^0(F,qL_{F})^*)$ (we refer to \cite[Chapter II]{Nak} or \S \ref{Sec:Qdiv} for basic properties of morphisms associated to $\Q$-divisors which are not Cartier), and they assume the following possible forms

\begin{enumerate}
\item[($A_q$)]  $qL_F$ is Cartier, globally generated and $\phi_{qL_F}$ is generically finite, or equivalently $qL_F$ is  Cartier, globally generated and big;
\item[($B_q$)]  $\phi_{qL_F}$ is generically finite;
\item[($C_q$)] $qL_F$ is Cartier and big. 
\end{enumerate}
Note that ($A_q$) implies ($B_q$) and ($C_q$), while ($B_q$) and ($C_q$) are independent of each other.

The other kind of assumptions is on a general fiber $F$ and they assume the following possible forms 
\begin{enumerate}[(a)]
\item[($a$)] \label{Ass-F} A general fiber $F$ of $f$ has dimension $n\geq 2$ and $\kappa(F)\geq 0$, i.e. it has non-negative Kodaira dimension.
\item[($b_q$)] \label{Ass-spec} A general fiber $F$ of $f$ has dimension $n=1$ and $\lfloor q L_F\rfloor$ is special, i.e. $h^1(F,\lfloor q L_F\rfloor)\neq 0$. 
\end{enumerate}
The assumption ($a$)  is relevant in order to apply some of the Noether inequalities of \S \ref{Sec:Noether}, while the assumption ($b_q$)  will allow to apply Clifford's theorem which says that $\deg qL_F\geq 2h^0(F,\lfloor qL_F\rfloor)-2$.

The first slope inequality that we prove involves the numerical invariants of the polarized general fiber $(F,L_F)$, and more specifically either $L_F^n$ or $h^0(F,L_F)$, under the assumption that $\phi_{L_F}$ is generically finite. 

\begin{theorem}\label{T:Xiao-higher}
Assume we are in the set-up  \eqref{setup} and suppose that $f_*\O_X(L)$  is nef and that $\phi_{L_F}$ is generically finite, then we have that
\begin{equation}\label{E:ineqXiao1}
M_{\ell}^{n+1} \ge 2\frac{P_{\ell}^n }{P_{\ell}^n +n} \deg f_*\O_X(L) \ge  2\frac{h^0(\wt F,P_{\ell})-n}{h^0(\wt F,P_{\ell})}\deg f_*\O_X(L) .
\end{equation}
 If $n\geq 2$ and the first inequality is an equality, then $\mu_-(f_*\O_X(L))=0$, and hence $f_*\O_X(L)$ is not ample.

Suppose moreover that one of the following two conditions are satisfied:
\begin{itemize}
\item $\dim F\geq 2$ and $\kappa(F)\geq 0$, i.e. ($a$)  holds true;
\item $\dim F=1$ and $\lfloor L_F\rfloor $ is special (which implies that $\kappa(F)=1$ since $\phi_{L_F}$ is generically finite), i.e. ($b_1$)  holds true.
\end{itemize}
Then we also have that 
\begin{equation}\label{E:ineqXiao2}
M_{\ell}^{n+1} \ge 	4\frac{P_{\ell}^n }{P_{\ell}^n +2n} \deg f_*\O_X(L) \ge  4\frac{h^0(\wt F,P_{\ell})-n}{h^0(\wt F,P_{\ell})}\deg f_*\O_X(L).	
\end{equation}
If $n\geq 2$ and the first inequality is an equality, then $\mu_-(f_*\O_X(L))=0$, and hence $f_*\O_X(L)$ is not ample.
\end{theorem}
\begin{proof}
First of all, note that the (Cartier) divisor $P_{\ell}$ on $\wt F$ is globally generated (and hence nef) by Proposition \ref{P:Mnef}\eqref{P:Mnef1} and with generically finite associated morphism $\phi_{P_{\ell}}$ by  Proposition \ref{P:Mnef}\eqref{P:Mnef4} together with assumption that $\phi_{L_F}$ is generically finite. Therefore, the second inequality in \eqref{E:ineqXiao1} follows from the inequality
\begin{equation}\label{E:Pln0}
P_{\ell}^n\geq h^0(\wt F, P_{\ell})-n \,,
\end{equation} 
which holds by Corollary  \ref{C:Noether1}\eqref{C:Noether1i}. Similarly, the second inequality in \eqref{E:ineqXiao2} is a consequence  of the inequality 
\begin{equation}\label{E:Pln}
P_{\ell}^n\geq 2h^0(\wt F, P_{\ell})-2n,
\end{equation} 
which follows from Corollary \ref{C:Noether1}\eqref{C:Noether1ii} if $\dim F\geq 2$ and $\kappa(F)\geq 0$, and from Clifford's theorem if $\dim F=1$ and $\lfloor L_F\rfloor $ is special (using that $P_{\ell}\leq  \lfloor L_F\rfloor$).

Let us focus on the first inequalities in \eqref{E:ineqXiao1} and \eqref{E:ineqXiao2}. 

We first apply Corollary \ref{C:spec-ineq}\eqref{C:spec-ineq1A} with $N_{\ell+1}:=N_{\ell}=M_{\ell}-\mu_{\ell} \wt F$ and $\mu_{\ell+1}=\mu_{\ell}$,
%(which satisfy the assumptions of loc. cit. by Remark \ref{R:choices} since $f_*\O_X(L)$ is nef by hypothesis), 
and we get
\begin{equation}\label{E:furbohnno}
N_{\ell}^{n+1}=M_{\ell}^{n+1}-(n+1)\mu_{\ell}P_{\ell}^n \geq \sum_{i=1}^{\ell} (P_i+P_{i+1})P_{\ell}^{n-1}(\mu_i-\mu_{i+1})=\sum_{i=1}^{\ell-1} (P_i+P_{i+1})P_{\ell}^{n-1}(\mu_i-\mu_{i+1}).
\end{equation}
This inequality implies that 
\begin{equation}\label{E:1ineq}
M_{\ell}^{n+1} \geq \sum_{i=1}^{\ell-1} (P_i+P_{i+1})P_{\ell}^{n-1}(\mu_i-\mu_{i+1})+2\mu_{\ell}P_{\ell}^n,
\end{equation}
with the equality that can occur only if either $n=1$ or $n\geq 2$ and $\mu_{\ell}=0$.

In order to give a lower bound on the right hand side of  \eqref{E:1ineq}, we define
$$
p:= \min\{i \in \{1,\ldots,\ell\} \ | \ P_{\ell}^n -P_{\ell}^{n-1}\cdot P_i=0 \}.
$$
Observe that, since the intersection numbers 
$$P_{\ell}^n -P_{\ell}^{n-1}\cdot P_i=P_{\ell}^{n-1}\cdot (Z_i-Z_{\ell})_{|\wt F}$$ 
are non-increasing in $i$ (because $P_{\ell}$ is nef by Proposition \ref{P:Mnef}\eqref{P:Mnef1}, $Z_i-Z_{i+1}\geq 0$ by \eqref{E:Zi}, and $\wt F$ is a general fiber of $\wt f$), we have that 
\begin{equation}\label{E:intMiF}
\begin{sis} 
P_{\ell}^n -P_{\ell}^{n-1}\cdot P_i \geq 1 & \text{ if } i<p, \\
P_{\ell}^n -P_{\ell}^{n-1}\cdot P_i= 0 & \text{ if } p\leq i. \\
\end{sis} 
\end{equation}
In order to treat the two inequalities  \eqref{E:ineqXiao1} and \eqref{E:ineqXiao2} simultaneously, we define
\begin{equation*}
\epsilon:=
\begin{cases} 
2 & \text{ if either } \dim F\geq 2 \: \text{ and } \kappa(F)\geq 0, \: \text{ or } \dim F=1 \: \text{ and } \lfloor L_F\rfloor  \: \text{ is special.} \\
1 & \text{ otherwise.}  
\end{cases} 
\end{equation*}

\un{Claim:} We have that 
        \begin{equation}\label{E:boundMi}
	P_{\ell}^{n-1}\cdot P_i\geq 
	\begin{cases}
	\epsilon[h^0(\wt F,P_i)-1]\geq \epsilon[r_i - 1] &  \text{ if } 1\leq i<p,  \\
	\epsilon[h^0(\wt F,P_{\ell})-n]\geq \epsilon[r_i-n+(\ell-i)] & \text{ if }  p\leq i\leq \ell.
	\end{cases}
	\end{equation}

Indeed, if $p\leq i\leq \ell$, which implies that $P_{\ell}^n -P_{\ell}^{n-1}\cdot P_i= 0$ by \eqref{E:intMiF}, then \eqref{E:Pln0} and \eqref{E:Pln} give 
$$
P_{\ell}^{n-1}\cdot P_i=P_{\ell}^{n}\geq \epsilon[h^0(\wt F,P_{\ell})-n].
$$
If, instead, $1\leq i<p$, which implies that $P_{\ell}^n - P_{\ell}^{n-1}P_i \ge 1$ by \eqref{E:intMiF}, then we get 
$$
P_{\ell}^{n-1}\cdot P_i \geq \epsilon[h^0(\wt F,P_i)-1],
$$
by applying  
\begin{itemize}
\item  when ($a$)  holds true: Proposition \ref{P:bound1} with $L:=P_{\ell}$ which is nef 
%by Proposition \ref{P:Mnef}\ref{P:Mnef1} 
and with generically finite associated map $\phi_{P_{\ell}}$ 
%by Proposition \ref{P:Mnef}\ref{P:Mnef4} and assumption \ref{Ass-L2}  
and $M:=P_{i}$ which is nef by Proposition \ref{P:Mnef}\eqref{P:Mnef1}  and such that $P_{\ell}-P_i\geq 0$ by  \eqref{E:Mi};  
\item   when ($b_1$)  holds true:  Clifford theorem to $P_i$ which is special since $P_i\leq P_{\ell}\leq \lfloor L_F\rfloor $.
\end{itemize}
We conclude in both cases using that $h^0(\wt F,M_{i}|_{\wt F})\geq r_{i}$ by Proposition \ref{P:Mnef}\eqref{P:Mnef2} and (for the second case) the fact that $r_{i+1}\geq r_i+1$ by \eqref{E:ri}.

\vspace{0.1cm}

By substituting the inequalities given by the above Claim  into \eqref{E:1ineq} and dividing out by $\epsilon$, we get  
\begin{equation}\label{E:1ineqbis}
\begin{aligned}
\frac{M_{\ell}^{n+1}}{\epsilon}& \geq  \sum_{i=1}^{p-2}(r_i-1+r_{i+1}-1)(\mu_i-\mu_{i+1}) + (r_{p-1}-1 + r_{p}-n+(\ell-p))(\mu_{p-1} -\mu_{p}) + \\
& + \sum_{ i=p}^{\ell-1}(r_i-n +(\ell-i) +r_{i+1}-n+(\ell-i-1))(\mu_i-\mu_{i+1})  +  2(r_{\ell}-n)\mu_{\ell} \\
&\underbrace{\ge}_{r_{i+1}\geq r_i+1}   \sum_{i=1}^{p-2}(2r_i-1)(\mu_i-\mu_{i+1}) + (2r_{p-1}-n+\ell-p)(\mu_{p-1} - \mu_{p})  +\\ 
& +2\sum_{ i=p }^{\ell-1}(r_i-n +(\ell-i))(\mu_i-\mu_{i+1})  +  2(r_{\ell}-n)\mu_{\ell}\\
&\underbrace{=}_{\textrm{Lemma } \ref{formuletta}}  2 \deg f_*\O_X(L) -\mu_1 +\mu_{p-1} +(-n+\ell-p)(\mu_{p-1}-\mu_p) +\\
&+ 2(-n+\ell-p)\mu_p-2\mu_{p+1}-\ldots -2\mu_{\ell}=\\
&=  2 \deg f_*\O_X(L) -\mu_1 +(1-n+\ell-p)\mu_{p-1} + (-n+\ell-p)\mu_p-2\mu_{p+1}-\ldots -2\mu_{\ell} \\
&\underbrace{\ge}_{\mu_i\geq \mu_{i+1}}  2 \deg f_*\O_X(L) -\mu_1 +(1-n)\mu_1+(\ell-p)\mu_{p-1} + (-n+\ell-p)\mu_p-2(\ell-p)\mu_{p}\\
& =  2 \deg f_*\O_X(L) -n\mu_1-n\mu_p +(\ell-p)(\mu_{p-1}-\mu_p) \underbrace{\ge}_{\mu_{p-1} \ge \mu_{p}} 2 \deg f_*\O_X(L) -n (\mu_1+\mu_p).
\end{aligned}
\end{equation}
Note that in the above  inequalities we have used that $\mu_{\ell}\geq 0$ by \eqref{E:Nef-HN}  and the assumption that $f_*\O_X(L)$ is nef.

We next apply  Corollary \ref{C:spec-ineq}\eqref{C:spec-ineq2} with $N_{\ell+1}:=N_{\ell}=M_{\ell}-\mu_{\ell}\wt F$, $\mu_{\ell+1}=\mu_{\ell}$, and either $q=2$ and $s_1=1<s_2=p<s_3=\ell+1$ if $1<p$, or $q=1$ and $s_1=1=p<s_2=\ell+1$ if $p=1$:

\begin{equation}\label{E:2ineq}
N_{\ell}^{n+1}=M_{\ell}^{n+1}-(n+1)\mu_{\ell}P_{\ell}^n \ge P_{\ell}^{n-1} \left[(P_1+P_p)(\mu_1-\mu_p) + (P_p+P_{\ell})(\mu_p-\mu_{\ell}) \right]=
\end{equation}
$$ 
= P_{\ell}^{n-1} P_1 (\mu_1-\mu_p)+P_\ell^n (\mu_1+\mu_p-2\mu_{\ell})\geq P_\ell^n (\mu_1+\mu_p-2\mu_{\ell}),
$$ 
where we used in the second equality that $P_{\ell}^{n-1}P_p=P_{\ell}^n$ by the definition of $p$, and in the last inequality that $P_{\ell}^{n-1}P_1 (\mu_1-\mu_p) \ge 0$ since $P_{\ell}$ and $P_1$ are nef by Proposition \ref{P:Mnef}\eqref{P:Mnef1} and $\mu_1\geq  \mu_p$ by \eqref{E:ai}.
The  inequality \eqref{E:2ineq} implies that 
\begin{equation}\label{E:2ineqbis}
M_{\ell}^{n+1} \geq P_\ell^n (\mu_1+\mu_p),
\end{equation}
with the equality that can occur only if either $n=1$ or $n\geq 2$ and $\mu_{\ell}=0$.

Now we conclude as follows: 
\begin{itemize}
\item if 
$$
\mu_1+\mu_p \le\frac{2\epsilon\deg f_*\O_X(L)}{P_\ell^n +\epsilon n}
$$
then the conclusion follows from \eqref{E:1ineqbis}; 
\item if
$$
\mu_1+\mu_p \ge\frac{2\epsilon\deg f_*\O_X(L)}{P_\ell^n +\epsilon n}
$$
then the conclusion follows from \eqref{E:2ineqbis}. 
\end{itemize}	
\end{proof}

\begin{corollary}\label{C:Xiao-high1}
Suppose that $L$ and $f_*\O_X(L)$ are nef, and that $\phi_{L_F}$ is generically finite. 
Then 
	$$
	L^{n+1} \ge 
	\begin{cases} 
	  4\frac{h^0(F,L_F)-n}{h^0(F,L_F)}\deg f_*\O_X(L) & \text{ if  either } \dim F\geq 2 \text{ and } \kappa(F)\geq 0, \\
	  & \text{ or } \dim F=1 \: \text{ and } \lfloor L_F\rfloor \: \text{ is special},\\
	   2\frac{h^0(F,L_F)-n}{h^0(F,L_F)}\deg f_*\O_X(L) & \text{ otherwise.} 
	\end{cases}
       $$
     For $n\geq 2$, if the above inequalities are equalities, then $\mu_-(f_*\O_X(L))=0$ and hence $f_*\O_X(L)$ is not ample.
      \end{corollary}
\begin{proof}
This follows from the second inequality in Theorem \ref{T:Xiao-higher}, using that  $L^{n+1}\geq M_{\ell}^{n+1}$ by Lemma \ref{L:Lnef}\eqref{L:Lnef2} and that 
$h^0(F,L_F)=h^0(\wt F, P_{\ell})$ by Proposition \ref{P:Mnef}\eqref{P:Mnef4}. 
\end{proof}

\begin{corollary}\label{C:Xiao-high2}
Suppose that $L$ and $f_*\O_X(L)$ are nef, and that $L_F$ is Cartier, globally generated and big. Then
	$$
	L^{n+1} \ge
	\begin{cases} 
	4\frac{L_F^n }{L_F^n +2n} \deg f_*\O_X(L) &  \text{ if  either } \dim F\geq 2 \text{ and } \kappa(F)\geq 0, \\
	  & \text{ or } \dim F=1\: \text{ and } \lfloor L_F\rfloor  \: \text{ is special},\\
	  	2\frac{L_F^n }{L_F^n +n} \deg f_*\O_X(L)  & \text{ otherwise.} 
	\end{cases}
	$$
	For $n\geq 2$, if the above inequalities are equalities, then $\mu_-(f_*\O_X(L))=0$ and hence $f_*\O_X(L)$ is not ample.
\end{corollary}
\begin{proof}
This follows from the first inequality in Theorem \ref{T:Xiao-higher}, using that  $L$ is nef, and hence that $L^{n+1}\geq M_{\ell}^{n+1}$ by Lemma \ref{L:Lnef}\eqref{L:Lnef2}, and that $L_F^n=P_{\ell}^n$ by Proposition \ref{P:Mnef}\eqref{P:Mnef5}. 
\end{proof}

Note that the assumptions of Corollary \ref{C:Xiao-high2} are stronger than the assumptions of Corollary \ref{C:Xiao-high1} but, at the same time, the conclusion is also stronger by Corollary \ref{C:Noether1}.

%\begin{remark}
%If $L$ does not contain any irreducible component of reducible fibers, in Corollary \ref{C:Xiao-high2}, using Lemma \ref{L:Lnef}\eqref{L:Lnef1}, we can replace the assumptions $L$ nef and $L_F$ Cartier, with $L$ $f$-globally generated (recall that the restriction to $F$ is done as $\Q$-Cartier divisor). In this second set-up, $L$ is automatically Cartier as we are going to explain: the assumption $L$ $f$-globally generated implies that the restriction of $L$ to any fiber is base point free; as explained in Section \ref{Sec:Qdiv}, the non-Cartier locus is always contained in the base locus, so the restriction of $L$ to any fiber is Cartier, and then $L$ itself is Cartier by \cite[Definition-Lemma 4.20]{Kol2}.
%\end{remark}

\begin{remark}\label{R:curves}
If $f:X\to T$ is a family of curves such that a general fiber $F$ is a smooth projective irreducible curve of genus $g\geq 2$ and $L=K_{X/T}$, both Corollaries  \ref{C:Xiao-high1} and \ref{C:Xiao-high2} reduces to the Xiao-Cornalba-Harris inequality (see \cite[Thm. 2]{Xiao} or \cite[Prop. 4.3]{CH} \footnote{this result is stated in \cite[Prop. 4.3]{CH} under the 
further assumption that $f$ is a family of stable curves. However, the proof can be easily adapted to the general case of a fibration with normal total space.})
\begin{equation}\label{E:slope-curves}
K_{X/T}^2\geq \frac{4g-4}{g} \deg f_*\O_X(K_{X/T}).  
\end{equation}
Moreover, the slope inequality \eqref{E:slope-curves} is sharp since it is attained for non-isotrivial families of hyperelliptic stable curves such that every node of each fiber of $f$ is non-separating (and indeed these are  the only families of stable curves attaining the equality in \eqref{E:slope-curves}), see \cite[Thm. 4.12]{CH}.
\end{remark}

\begin{remark}\label{R:slope-surf}
In the special case $n=2$ and $L=K_{X/T}$, and under the assumptions that $X$  has terminal singularities (and hence isolated singularities since $\dim X=3$) and a general fiber $F$ (which is automatically smooth) is of general type, then
\begin{enumerate}[(i)]
\item the slope inequality in Corollary \ref{C:Xiao-high1} was proved by Ohno \cite[Prop. 2.1]{Ohno} without the assumption that $\phi_{K_F}$ is generically finite;
\item the slope inequality in Corollary \ref{C:Xiao-high2} was proved by Hu-Zhang \cite[Thm. 1.7]{HZ} without the assumption that $K_F$ is globally generated, using methods of positive characteristics (and indeed \cite[Thm. 1.7]{HZ} holds also in positive characteristics). 
\end{enumerate}
Moreover, in this special case, Ohno  proved in \cite[Prop. 2.1]{Ohno} that the equality in  Corollary \ref{C:Xiao-high1} cannot occur unless $f$ is isotrivial, see Proposition \ref{P:slopepair}\eqref{P:slopepair1} for a generalisation. %On the other hand, we do not know if, in this special case or more generally for $n\geq 2$, equality can occur in Corollary \ref{C:Xiao-high2} if $f$ is not isotrivial. 
\end{remark}

The second slope inequality that we prove  is  a refinement (at least for $n\geq 3$) of the one obtained in Theorem \ref{T:Xiao-higher}, under the stronger assumption that $L_F$ gives a birational map and it is subcanonical  (on a resolution of singularities).

\begin{theorem}\label{T:Xiao-n}
	Assume we are in the set-up  \eqref{setup} and suppose that $f_*\O_X(L)$  is nef, $\phi_{L_F}$ is birational and $n\ge 2$. Let $s \in \mathbb N$ such that $K_{\tilde F}-sP_{\ell} \ge 0$.
	Then 
\begin{equation}\label{E:ineqXiao-n}
M_{\ell}^{n+1} \ge 2(n+s)\frac{P_{\ell}^n}{P_{\ell}^n +(n+s)(n+2)} \deg f_*\O_X(L) \ge  2(n+s)\frac{h^0(\wt F,P_{\ell})-n-2}{h^0(\wt F,P_{\ell})}\deg f_*\O_X(L) .
\end{equation}
 If $n\geq 2$ and the first inequality is an equality, then $\mu_-(f_*\O_X(L))=0$ and hence $f_*\O_X(L)$ is not ample.
\end{theorem}

\begin{proof}
First of all, note that the (Cartier) divisor $P_{\ell}$ on $\wt F$ is globally generated (and hence nef) by Proposition \ref{P:Mnef}\eqref{P:Mnef1} and with birational associated morphism $\phi_{P_{\ell}}$ by  Proposition \ref{P:Mnef}\eqref{P:Mnef4} together with assumption that $\phi_{L_F}$ is birational. Therefore, the second inequality in \eqref{E:ineqXiao-n} follows from the inequality
\begin{equation}\label{E:Pln-n}
	P_{\ell}^n\geq (n+s)(h^0(\wt F, P_{\ell})-n-1)+2,
\end{equation} 
which holds by Lemma \ref{L:Harris}. 

We first apply Corollary \ref{C:spec-ineq}\eqref{C:spec-ineq1A} with $N_{\ell+1}:=N_{\ell}=M_{\ell}-\mu_{\ell} \wt F$ and $\mu_{\ell+1}=\mu_{\ell}$,
%(which satisfy the assumptions of loc. cit. by Remark \ref{R:choices} since $f_*\O_X(L)$ is nef by hypothesis), 
and we get
\begin{equation}\label{E:furbohnno-n}
N_{\ell}^{n+1}=M_{\ell}^{n+1}-(n+1)\mu_{\ell}P_{\ell}^n \geq \sum_{i=1}^{\ell} (P_i+P_{i+1})P_{\ell}^{n-1}(\mu_i-\mu_{i+1})=\sum_{i=1}^{\ell-1} (P_i+P_{i+1})P_{\ell}^{n-1}(\mu_i-\mu_{i+1}).
\end{equation}
This inequality implies that 
\begin{equation}\label{E:1ineq-n}
M_{\ell}^{n+1} \geq \sum_{i=1}^{\ell-1} (P_i+P_{i+1})P_{\ell}^{n-1}(\mu_i-\mu_{i+1})+2\mu_{\ell}P_{\ell}^n,
\end{equation}
with the equality that can occur only if either $n=1$ or $n\geq 2$ and $\mu_{\ell}=0$.

	In order to give a lower bound on the right hand side of  \eqref{E:1ineq-n}, we define
	$$
	p:= \min\{i \in \{1,\ldots,\ell\} \ | \ P_{\ell}^n -P_{\ell}^{n-1}\cdot P_i=0 \}.
	$$
	Observe that, since the intersection numbers 
	$$P_{\ell}^n -P_{\ell}^{n-1}\cdot P_i=P_{\ell}^{n-1}\cdot (Z_i-Z_{\ell})$$ 
	are non-increasing in $i$ (because $P_{\ell}$ is nef by Proposition \ref{P:Mnef}\eqref{P:Mnef1}, $Z_i-Z_{i+1}\geq 0$ by \eqref{E:Zi}, and $\wt F$ is a general fiber of $\wt f$), we have that 
	\begin{equation}\label{E:intMiF-n}
	\begin{sis} 
	P_{\ell}^n -P_{\ell}^{n-1}\cdot P_i \geq 1 & \text{ if } i<p, \\
	P_{\ell}^n -P_{\ell}^{n-1}\cdot P_i= 0 & \text{ if } p\leq i. \\
	\end{sis} 
	\end{equation}

	\un{Claim:} We have that 
	\begin{equation}\label{E:boundMi-n}
	P_{\ell}^{n-1}\cdot P_i\geq 
	\begin{cases}
	(n+s)(h^0(\wt F,P_i)-2)+2\geq (n+s)(r_i - 2) +2&  \text{ if } 1\leq i<p,  \\
	(n+s)(h^0(\wt F,P_{\ell})-n-1)+2\geq (n+s)[r_i+(\ell-i)-n-1] +2& \text{ if }  p\leq i\leq \ell.
	\end{cases}
	\end{equation}
	
	Indeed, if $p\leq i\leq \ell$, which implies that $P_{\ell}^n -P_{\ell}^{n-1}\cdot P_i= 0$ by \eqref{E:intMiF-n}, then \eqref{E:Pln-n} gives 
	$$
	P_{\ell}^{n-1}\cdot P_i=P_{\ell}^{n}\geq (n+s)(h^0(\wt F,P_{\ell})-n-1) +2.
	$$
	If, instead, $1\leq i<p$, which implies that $P_{\ell}^n - P_{\ell}^{n-1}P_i \ge 1$ by \eqref{E:intMiF}, then we get 
	$$
	P_{\ell}^{n-1}\cdot P_i \geq (n+s)(h^0(\wt F,P_i)-2)+2,
	$$
	by applying  Proposition \ref{P:Castelnuovo3}.
	We conclude in both cases using that $h^0(\wt F,M_{i}|_{\wt F})\geq r_{i}$ by Proposition \ref{P:Mnef}\eqref{P:Mnef2} and (for the second case) the fact that $r_{i+1}\geq r_i+1$ by \eqref{E:ri}.
		
	\vspace{0.1cm}
	
	By substituting the inequalities given by the above Claim  into \eqref{E:1ineq}, we get  
	\begin{equation}\label{E:1ineqbisX}
	\begin{aligned}
	M_{\ell}^{n+1}& \geq  \sum_{i=1}^{p-2}[(n+s)(r_i-2+r_{i+1}-2)+4](\mu_i-\mu_{i+1}) + [(n+s)(r_{p-1}-2 + r_{p}+\ell-p-n-1)+4](\mu_{p-1} -\mu_{p}) + \\
	& + \sum_{ i=p}^{\ell-1}[(n+s)(r_i-n-1 +(\ell-i) +r_{i+1}-n-1+(\ell-i-1))+4](\mu_i-\mu_{i+1})  + 2[(n+s)(r_{\ell}-n-1)+2]\mu_{\ell} \\
	&\underbrace{\ge}_{r_{i+1}\geq r_i+1}   \sum_{i=1}^{p-2}[(n+s)(2r_i-3)+4](\mu_i-\mu_{i+1}) + [(n+s)(2r_{p-1}-1+\ell-p-n-1)+4](\mu_{p-1} -\mu_{p}) +\\ 
	& +2\sum_{ i=p}^{\ell-1}[(n+s)(r_i-n-1 +\ell-i)+2](\mu_i-\mu_{i+1}) +   2[(n+s)(r_{\ell}-n-1)+2]\mu_{\ell} \\
	&\underbrace{=}_{\textrm{Lemma } \ref{formuletta}}  2(n+s) \deg f_*\O_X(L) -[3(n+s)-4](\mu_1 -\mu_{p-1}) +[(n+s)(\ell-p-n-2)+4](\mu_{p-1}-\mu_p) +\\
	&+ 2[(n+s)(\ell-p-n-1)+2]\mu_p-2(n+s)\mu_{p+1}-\ldots -2(n+s)\mu_{\ell}\\
	&\ge (n+s)[ 2\deg f_*\O_X(L) -3\mu_1 +3\mu_{p-1} + (\ell-p-n-2)\mu_{p-1} - (\ell-p-n-2)\mu_p \\
	&+ 2(\ell-p-n-1)\mu_p-2\mu_{p+1}-\ldots -2\mu_{\ell}] \\
	&=  (n+s)[ 2\deg f_*\O_X(L) -3\mu_1 + (\ell-p-n+1)\mu_{p-1} + (\ell-p-n)\mu_p-2\mu_{p+1}-\ldots -2\mu_{\ell}] \\
	&\underbrace{\ge}_{\mu_i\geq \mu_{i+1}}  (n+s)[ 2\deg f_*\O_X(L) -3\mu_1 + (-n+1)\mu_{p-1} +(\ell-p)\mu_{p} + (\ell-p-n)\mu_p-2(\ell-p)\mu_p]\\
	& =  (n+s)[ 2\deg f_*\O_X(L) -3\mu_1 + (-n+1)\mu_{p-1} -n\mu_p]\\
	& \underbrace{\ge}_{\mu_1\geq \mu_{p-1} \ge \mu_{p}} (n+s)[2\deg f_*\O_X(L) -(n+2) (\mu_1+\mu_p)].
	\end{aligned}
	\end{equation}

	Note that in the above inequalities we have used that $\mu_{\ell}\geq 0$ by \eqref{E:Nef-HN}  and the assumption that $f_*\O_X(L)$ is nef.
	
	We next apply  Corollary \ref{C:spec-ineq}\eqref{C:spec-ineq2} with $N_{\ell+1}:=N_{\ell}=M_{\ell}-\mu_{\ell}\wt F$, $\mu_{\ell+1}=\mu_{\ell}$, and either $q=2$ and $s_1=1<s_2=p<s_3=\ell+1$ if $1<p$, or $q=1$ and $s_1=1=p<s_2=\ell+1$ if $p=1$:

\begin{equation}\label{E:2ineqX}
N_{\ell}^{n+1}=M_{\ell}^{n+1}-(n+1)\mu_{\ell}P_{\ell}^n \ge P_{\ell}^{n-1} \left[(P_1+P_p)(\mu_1-\mu_p) + (P_p+P_{\ell})(\mu_p-\mu_{\ell}) \right]=
\end{equation}
$$ 
= P_{\ell}^{n-1} P_1 (\mu_1-\mu_p)+P_\ell^n (\mu_1+\mu_p-2\mu_{\ell})\geq P_\ell^n (\mu_1+\mu_p-2\mu_{\ell}),
$$ 
where we used in the first equality that $P_{\ell}^{n-1}P_p=P_{\ell}^n$ by the definition of $p$, and in the last equality that $P_{\ell}^{n-1}P_1 (\mu_1-\mu_p) \ge 0$ since $P_{\ell}$ and $P_1$ are nef by Proposition \ref{P:Mnef}\eqref{P:Mnef1} and $\mu_1\geq  \mu_p$ by \eqref{E:ai}.
The  inequality \eqref{E:2ineqX} implies that 
\begin{equation}\label{E:2ineqbisX}
M_{\ell}^{n+1} \geq P_\ell^n (\mu_1+\mu_p),
\end{equation}
with the equality that can occur only if either $n=1$ or $n\geq 2$ and $\mu_{\ell}=0$.

	Now we conclude as follows: 
	\begin{itemize}
		\item if 
		$$
		\mu_1+\mu_p \le\frac{2(n+s)\deg f_*\O_X(L)}{P_\ell^n +(n+s)(n+2)}
		$$
		then the conclusion follows from \eqref{E:1ineqbisX}; 
		\item if
		$$
		\mu_1+\mu_p \ge\frac{2(n+s)\deg f_*\O_X(L)}{P_\ell^n +(n+s)(n+2)}
		$$
		then the conclusion follows from \eqref{E:2ineqbisX}. 
	\end{itemize}	
\end{proof}

\begin{corollary}\label{C:Xiao-bir1}
	Suppose that $L$ and $f_*\O_X(L)$ are nef, $\phi_{L_F}$ is birational and $n\ge 2$.  Assume that the singularities of the general fibre $F$ are canonical and let $s \in \mathbb N$ such that $K_{ F}-sL_F \ge 0$.
	Then 
$$
	L^{n+1} \ge    2(n+s)\frac{h^0(F,L_{F})-n-2}{h^0( F,L_{F})}\deg f_*\O_X(L) .
$$
  If $n\geq 2$ and the inequality is an equality, then $\mu_-(f_*\O_X(L))=0$ and hence $f_*\O_X(L)$ is not ample.
\end{corollary}
\begin{proof}
By Equation \eqref{E:Pi} we know that $Z:=\mu^*L_F - P_{\ell}$ is effective. Since the singularities of $F$ are canonical, we get
$$
K_{\tilde F} -sP_\ell = \mu^*K_{F} + E - s(\mu^*L_F -Z) = \mu^*(K_F -sL_F) + E +sZ \ge 0.
$$

We can hence apply Theorem \ref{T:Xiao-n}, using that  $L^{n+1}\geq M_{\ell}^{n+1}$ by Lemma \ref{L:Lnef}\eqref{L:Lnef2} and that 
	$h^0(F,L_F)=h^0(\wt F, P_{\ell})$ by Proposition \ref{P:Mnef}\eqref{P:Mnef4}. 
\end{proof}

\begin{corollary}\label{C:Xiao-bir2}
	Suppose that $L$ and $f_*\O_X(L)$ are nef,  $L_F$ is Cartier, globally generated, $\phi_{L_F}$ is birational and $n\ge 2$.  Assume that the singularities of the general fibre $F$ are canonical and let $s \in \mathbb N$ such that $K_{ F}-sL_F \ge 0$. Then
	$$
	L^{n+1} \ge 2(n+s)\frac{L_{F}^n}{L_{F}^n +(n+s)(n+2)} \deg f_*\O_X(L). 
	$$
 If $n\geq 2$ and the inequality is an equality, then $\mu_-(f_*\O_X(L))=0$ and hence $f_*\O_X(L)$ is not ample.
\end{corollary}
\begin{proof}
Since	the singularities of the $F$ are canonical, we can apply Theorem \ref{T:Xiao-n} (see the proof of Corollary \ref{C:Xiao-bir1}), using that  $L$ is nef, and hence that $L^{n+1}\geq M_{\ell}^{n+1}$ by Lemma \ref{L:Lnef}\eqref{L:Lnef2}, and that $L_F^n=P_{\ell}^n$ by Proposition \ref{P:Mnef}\eqref{P:Mnef5}. 
\end{proof}

The third (and last) slope inequality that we prove is independent of  the numerical invariants of the polarized general fiber $(F,L_F)$, and it is contained in the following

\begin{theorem}\label{T:Barja} Assume we are in the set-up  \eqref{setup} and suppose that  $L$ and $f_*\O_X(L)$ are nef.
		\begin{enumerate}
		\item  \label{T:Barja1}  Assume  there exists a  $q \in \N_{>0}$ such that  at least one of the following two conditions holds true 
		\begin{itemize}
			\item $\phi_{qL_F}$ is generically finite, i.e. condition ($B_q$) holds true;
			\item $qL_F$ is Cartier and big, i.e. condition ($C_q$) hods true.
		\end{itemize}
		Then 
		$$
		L^{n+1} \ge \frac{\deg f_*\O_X(L)}{q^n}.
		$$
		\item \label{T:Barja2} Assume that there exists a  $q \in \N_{>0}$ such that $\phi_{qL_F}$ is generically finite  (i.e. condition ($B_q$) holds true), and either $n=\dim F\geq 2$ and $\kappa(F)\geq 0$ (i.e. condition $(a)$ holds true) or $\dim F=1$ and $qL_F$ is special (i.e.  condition $(b_q)$ holds true), then
		$$
		L^{n+1} \ge 2\frac{\deg f_*\O_X(L)}{q^n}.
		$$
	\end{enumerate}

In particular, if the assumptions of either item (\ref{T:Barja1}) or item (\ref{T:Barja2}) hold and $\deg f_*\O_X(L)>0$, then $L$ is big.

\end{theorem}
The proof of the above Theorem is inspired by \cite[Page 69, Claim]{Barja_Phd} (see however Remark \ref{R:Barja}).
\begin{proof}
With the notation of \S\ref{Sec:HN}, we  define a partition of the set $[\ell]:=\{1,\ldots, \ell\}$ as it follows
$$
A_i:=\{j\in [\ell]\: : \dim \phi_{P_j}(\wt F)=i\}\quad \text{ for any } 0\leq i \leq n.
$$
Define the sequence of integers $1=m_0\leq m_1\leq \ldots \leq m_{n}\leq m_{n+1}:=\ell+1$ by
\begin{equation}\label{E:mi}
m_i:=
\begin{cases}
\min\{j\in A_i\} & \text{ if } A_i\neq \emptyset, \\
m_{i+1} & \text{ if } A_i= \emptyset. 
\end{cases}
\end{equation}
Equivalently, the above sequence of integers is such that 
$$
A_i=\{m_i,\ldots, m_{i+1}-1\} \quad \text{ for any } 0\leq i \leq n.
$$	

We now apply Corollary \ref{C:spec-ineq}\eqref{C:spec-ineq1} to the above sequence of integers $1=m_0\leq m_1\leq \ldots \leq m_{n}\leq m_{n+1}:=\ell+1$ by choosing the nef  $\Q$-divisor $N_{\ell+1}:=\mu^*(L)$ on $\wt X$ (which implies that $P_{\ell+1}=\mu_{\wt F}^*(L_F)$) and $\mu_{\ell+1}=0$  (which satisfy the assumptions of loc. cit. by Remark \ref{R:choices} since $L$ is nef by hypothesis), we thus have

\begin{equation}\label{E:disu1}
\mu^*(L)^{n+1}= N_{\ell+1}^{n+1}\geq  \sum_{i=0}^{n}\sum_{j=m_i}^{m_{i+1}-1} \left(\sum_{k=0}^{i} P_{j}^k P_{j+1}^{i-k} \right)\cdot \left(P_{m_{i+1}}\cdots P_{m_{n}}\right)(\mu_{j}-\mu_{j+1}).
\end{equation}
Note that $\mu_{\ell}\geq \mu_{\ell+1}=0$ by \eqref{E:Nef-HN} and the assumption that $f_*\O_X(L)$ is nef.

We now want to prove a lower bound on the right hand side of \eqref{E:disu1}.
Using that the line bundles $\{P_1,\ldots, P_{\ell}\}$ are nef by Proposition \ref{P:Mnef}\eqref{P:Mnef1} and they form a non-decreasing sequence  by \eqref{E:Pi}, we get the inequalities of cycles $P_{j}^k P_{j+1}^{i-k}\geq P_j^i$ for any $0\leq k \leq i$, which then imply the following inequality 
\begin{equation}\label{E:disu2}
\left(\sum_{k=0}^{i} P_{j}^k P_{j+1}^{i-k} \right)\left(P_{m_{i+1}}\cdots P_{m_{n}}\right)\geq (i+1) P_j^{i}\cdot \left(P_{m_{i+1}}\cdots P_{m_{n}}\right).
\end{equation}  
We now make the following 

\un{Claim:} For any $m_i\leq j\leq m_{i+1}-1$ (or equivalently $j\in A_i$, i.e. $\dim \phi_{P_j}(\wt F)=i$), we have that 
\begin{equation}\label{E:disu3}
P_j^{i}\cdot \left(P_{m_{i+1}}\cdots P_{m_{n}}\right)\geq 
\begin{cases}
\frac{h^0(\wt F,P_j)-i}{q^n} &  \text{ if we are in case (\ref{T:Barja1}) of the theorem}; \\
\frac{2h^0(\wt F,P_j)-2i}{q^n}& \text{ if we are in case (\ref{T:Barja2}) of the theorem}.\\
\end{cases}
\end{equation}

Indeed, if we set 
$
0\leq d:=\dim \phi_{P_{\ell}}(\wt F)\leq n,
$
then it follows from the definition \eqref{E:mi} of the integers $m_i$ that 
\begin{equation}\label{E:mi-prop}
\begin{sis}
\dim \phi_{P_{m_i}}(\wt F)\geq i & \quad \text{ if } i\leq d, \\
P_{m_i}=P_{m_{n+1}}=P_{\ell+1}=\mu_{\wt F}^*(L_F) & \quad \text{ if } i>d. 
\end{sis}
\end{equation}

Therefore, if $d=n$ (which happens precisely when $\phi_{L_F}$ is generically finite by Proposition \ref{P:Mnef}\eqref{P:Mnef4}), we get that 
\begin{equation}\label{E:disu3A}
P_j^{i}\cdot \left(P_{m_{i+1}}\cdots P_{m_{n}}\right)\geq 
\begin{cases}
2h^0(\wt F,P_j)-2i & \text{ if  either } \dim F\geq 2 \text{ and } \kappa(F)\geq 0, \\
	  & \text{ or } \dim F=1\: \text{ and } \lfloor L_F \rfloor \: \text{ is special},\\
 h^0(\wt F,P_j)-i  &  \text{ otherwise}. \\
\end{cases}
\end{equation}
by applying
\begin{itemize}
\item  if $\dim F\geq 2$ and $\kappa(F)\geq 0$: Proposition \ref{P:Noether1} with $k=i$, $H=P_j$, $L_t=P_{m_t}$ for any $k+1=i+1\leq t \leq n$;   
\item if $\dim F=1$ and  $\lfloor L_F \rfloor$  is special: Clifford's theorem for $P_{m_1}$ if $i=0$ or $P_j$ if $i=1$, using that $P_{m_1}, P_j\leq P_{\ell}\leq \lfloor L_F\rfloor \leq q\lfloor L_F\rfloor$ are special;
\item otherwise: Proposition \ref{P:Noether1bis} with $k=i$, $h=n-i$, $H=P_j$, $L_t=P_{m_t}$ for any $k+1=i+1\leq t \leq n$ (and any big and nef Cartier divisor $M$).
\end{itemize}

On the other hand, if $d<n$ (i.e. if $\phi_{L_F}$ is not generically finite),  then using \eqref{E:mi-prop}, we get that 
\begin{equation}\label{E:disu3B1}
P_j^{i}\cdot \left(P_{m_{i+1}}\cdots P_{m_{n}}\right)=P_j^{i}\cdot \left(P_{m_{i+1}}\cdots P_{m_d}\right)\cdot \mu_{\wt F}^*(L_F)^{n-d}=\frac{P_j^{i}\cdot \left(P_{m_{i+1}}\cdots P_{m_d}\right)\cdot \mu_{\wt F}^*(qL_F)^{n-d}}{q^{n-d}}.
\end{equation}
We can now apply 
\begin{itemize}
\item  if $\dim F\geq 2$ and $\kappa(F)\geq 0$: Proposition \ref{P:Noether1} with $k=i$, $H=P_j$, $L_t=P_{m_t}$ for any $k+1=i+1\leq t \leq d$ and $L_s=\mu_{\wt F}^*(qL_F)$ for any $d+1\leq s \leq n$;   
\item if $\dim F=1$ and  $\lfloor q L_F \rfloor$  is special: Clifford's theorem for $\lfloor \mu_{\wt F}^*(qL_F)\rfloor=\lfloor qL_F\rfloor$;
\item  Proposition \ref{P:Noether1bis} with  $k=i$, $h=d-i$, $H=P_j$, $L_t=P_{m_t}$ for any $k+1\leq t \leq k+h=d$  and $M=q \mu_{\wt F}^*(L_F)= \mu_{\wt F}^*(qL_F)$, if $qL_F$ is Cartier and big;
\item Proposition \ref{P:Noether1bis} with  $k=i$, $h=n-i$, $H=P_j$, $L_t=P_{m_t}$ for any $k+1\leq t \leq d$  and $L_t=q \mu_{\wt F}^*(L_F)= \mu_{\wt F}^*(qL_F)$ for any $d+1\leq t \leq n$, if $\phi_{qL_F}$ is generically finite;
\end{itemize} 
in order to get that 
\begin{equation}\label{E:disu3B2}
P_j^{i}\cdot \left(P_{m_{i+1}}\cdots P_{m_d}\right)\cdot \mu_{\wt F}^*(L_F)^{n-d}\geq
\begin{cases}
2 \frac{h^0(\wt F,P_j)-i}{q^{n-d}} \ge 2 \frac{h^0(\wt F,P_j)-i}{q^{n}} & \text{ if  either } \dim F\geq 2 \text{ and } \kappa(F)\geq 0, \\
	  & \text{ or } \dim F=1\: \text{ and } \lfloor q L_F \rfloor \: \text{ is special},\\
 \frac{h^0(\wt F,P_j)-i}{q^{n-d}} \ge \frac{h^0(\wt F,P_j)-i}{q^{n}}  &  \text{ otherwise}. \\
 \end{cases}
\end{equation}

By putting together \eqref{E:disu3A}, \eqref{E:disu3B1} and \eqref{E:disu3B2}, the Claim follows.

\smallskip

Finally, observe that for any $m_i\leq j\leq m_{i+1}-1$ (or equivalently $j\in A_i$), we have that 
\begin{equation}\label{E:disu4}
(i+1)[h^0(\wt F, P_j)-i]=h^0(\wt F, P_j)+i[h^0(\wt F, P_j)-i-1]\geq h^0(\wt F, P_j)\geq r_j,
\end{equation}
where the first inequality follows from the fact that $\dim \phi_{P_j}(\wt F)=i$ (since  $j\in A_i$) and the second inequality follows from Proposition \ref{P:Mnef}\eqref{P:Mnef2}. 

We are now ready to conclude. If we set
$$
e=
\begin{cases}
\frac{1}{q^n} & \text{ if we are in case (\ref{T:Barja1}) of the theorem}; \\
\frac{2}{q^n} & \text{ if we are in case (\ref{T:Barja2}) of the theorem}.
\end{cases}
$$
then, by putting together the inequalities \eqref{E:disu1}, \eqref{E:disu2}, \eqref{E:disu3}, \eqref{E:disu4}, we get
\begin{equation*}
\begin{aligned}
L^{n+1}= & \mu^*(L)^{n+1}\geq  \sum_{i=0}^{n}\sum_{j=m_i}^{m_{i+1}-1} \left(\sum_{k=0}^{i} P_{j}^k P_{j+1}^{i-k} \right)\cdot \left(P_{m_{i+1}}\ldots P_{m_{n}}\right)(\mu_{j}-\mu_{j+1})\\
& \geq  \sum_{i=0}^{n}\sum_{j=m_i}^{m_{i+1}-1} (i+1) P_j^{i}\cdot \left(P_{m_{i+1}}\ldots P_{m_{n}}\right)(\mu_{j}-\mu_{j+1})\\
& \geq  \sum_{i=0}^{n}\sum_{j=m_i}^{m_{i+1}-1} e (i+1) [h^0(\wt F, P_i)-i]  (\mu_{j}-\mu_{j+1})\\
& \geq  \sum_{i=0}^{n}\sum_{j=m_i}^{m_{i+1}-1} e r_j (\mu_{j}-\mu_{j+1})=e\sum_{j=1}^l r_j(\mu_j-\mu_{j+1})=e\deg f_*\O_X(L),\\
\end{aligned}
\end{equation*}
where we have used $\mu_i-\mu_{i+1}\geq 0$, $\mu_{\ell+1}=0$, and, in the last equality, Lemma \ref{formuletta}. Remark that, in particular, for $i=\ell$, we are using that $\mu_{\ell}\geq 0$, which is equivalent to $f_*\O_X(L)$ being nef.
\end{proof}

\begin{remark}\label{R:Barja}
When $q=1$, Theorem \ref{T:Barja}\eqref{T:Barja1} was asserted in \cite[Page 69, Claim]{Barja_Phd}, without any assumption on $L_F$. However, consider a fibration $f:X\to T$ as in \ref{setup} such that $L$ is nef and $L_F$ is not big. In particular, also $L$ is not big, and  then $L^{n+1}=0$. Let $A$ be an ample divisor on $T$ and consider the divisor $L_q=L+qf^*A$ for $q \in \mathbb N$. Then $L_q$ is a nef divisor on $X$, $(L_q)_F=L_{F}$, $L_q^{n+1} = 0$ and 
$$
\deg f_*\O_X(L_q)=\deg f_*\O_X(L)+h^0(F,L_F)q\deg A >0
$$
for $q \gg 0$, which shows that the bigness assumption on $L_F$ is necessary. 
In works such as \cite{Bar1,BS2}, \cite[Page 69, Claim]{Barja_Phd} is applied only for divisors $L$ which are $f$-ample, so the lack of convenient assumptions on $L_F$ in \cite[Page 69, Claim]{Barja_Phd} does not affect the applications.

\end{remark}

The inequalities in Theorem \ref{T:Barja} are  sharp  in any relative dimension $n\geq 1$. Let us first examine the well-known case of families of curves.

\begin{example}\label{Ex:curves}
\noindent 
\begin{enumerate}
\item \label{Ex:curves1} Let $(f:X\to T, \sigma)$ be a normal family of stable  curves of genus $g=2$  (over a smooth projective irreducible curve $T$) such that every node of each fiber of $f$ is non-separating and consider $L=K_{X/T}$. Then it follows from \cite[Thm. 4.12]{CH} that 
\begin{equation*}
K_{X/T}^2= \frac{4g-4}{g} \deg f_*\O_X(K_{X/T})=2 \deg f_*\O_X(K_{X/T}),
\end{equation*}
which shows that this example realises the equality in Theorem \ref{T:Barja}\eqref{T:Barja2}.

\item  \label{Ex:curves2} Let $(f:X\to T, \sigma)$ be a family of stable elliptic curves (over a smooth projective irreducible curve $T$), i.e. $f:X\to T$ is a  family of nodal integral curves of arithmetic genus one and $\sigma:T\to X$ is a section of $f$ such that $\sigma(s)$ is a smooth point of the fiber $X_s:=f^{-1}(s)$ for any $s\in T$. Let $D:=\Im(\sigma)$, and consider the Cartier divisor $L:=K_{X/T}+D$ on $X$. Note that $L_F$ is a divisor of degree one on a general (and indeed any) fiber $F$ of $f$, and hence $L_F$ is Cartier and big but $\phi_{L_F}$ is not generically finite. As explained in  \cite[Chap. XIII, Example 7.11]{GAC2} and  \cite[Chap. XIV, Cor. 5.14]{GAC2}, both $L$ and $f_*\O_X(L)$ are nef and we have that 
$$L^2=\deg f_*\O_X(K_{X/T})=\deg f_*\O_X(L)\, ,$$
which shows that this example realises the equality in Theorem \ref{T:Barja}\eqref{T:Barja1} for $q=1$.

\end{enumerate}
 \end{example}

\begin{remark}\label{R:sharp-Barja}
The inequalities in Theorem \ref{T:Barja} are sharp in any relative dimension $n\geq 1$ (at least for $q=1$) as it follows from Example  \ref{Ex:Pn}. 
%\begin{itemize}
%\item Example \ref{Ex:Pn} (in which $L_F$ is very ample) realises the equality in part \eqref{T:Barja1}. 
%\item Example \ref{Ex:doubPn}  (in which $L_F$ is globally generated with $\phi_{L_F}$ being of degree two finite cover) realises the equality in part \eqref{T:Barja2}. 
%\end{itemize}
It would also be interesting to find examples attaining  the equality in Theorem \ref{T:Barja}\eqref{T:Barja1} and with the property that $\phi_{L_F}$ is not generically finite, generalising Example \ref{Ex:curves}\eqref{Ex:curves2} from $n=1$ to higher dimension. 
\end{remark}

\begin{remark}[Nef sub-bundles]
With the same spirit of \cite{BScanproj,BS2}, the proofs of this section can be adapted to the case where $f_*\O_X(L)$ is not nef, but it does contain some non-trivial nef subsheaf $\mathcal{G}$. The slope inequalities obtained in this way will involve invariants of $\mathcal{G}$. We do not purse this direction as we do not have any new geometrical application to propose.
\end{remark}

\section{Canonically polarized varieties}\label{Sec:KSB}

\subsection{Slope inequalities for generic slc families}\label{sec:slopeKSB}

The aim of this subsection is to obtain some slope inequalities for certain fibrations whose generic fiber is semi-log canonical and whose relative dualizing divisor satisfies some positivity property. A special case of such families are the families of KSB stable pairs over curves. We refer to \cite{Kollar_book} for the definition of singularities that we are going to introduce. 

\begin{setup}\label{N:KSBgen}
	
Let $X$ be a deminormal (i.e.\ $S_2$ and nodal in codimension one) projective variety of equidimension $n+1$. 

Let $f: X \to T$ be a fibration (i.e. $f_*\O_X=\O_T$), where $T$ is a smooth projective irreducible curve such that every irreducible component of $X$ is dominant onto $T$. In particular, $f$ is projective, flat and with connected fibers, and $X$ is connected. Let $\Delta$ be an effective $\Q$-divisor on $X$ such that no irreducible component of the support of $\Delta$ is contained in the singular locus of $X$.   

We say that $f: (X,\Delta) \to T$ is a \emph{generic lc-family} (resp. \emph{generic canonical family}, resp \emph{generic klt-family}) if $X$ is normal (hence irreducible), $K_{X/T}+\Delta$ is $\Q$-Cartier, and the general fiber  $(F,\Delta_F)$ is log canonical (resp. canonical, resp. klt).

We say that $f: (X,\Delta) \to T$ is a \emph{generic slc-family} if $K_{X/T}+\Delta$ is $\Q$-Cartier, and the general fibre $(F,\Delta_F)$ is semi-log canonical.

We say that $f: (X,\Delta) \to T$ is a \emph{KSB-stable family} if  $K_{X/T}+\Delta$ is $\Q$-Cartier and relatively ample, $\Delta$ does not contain any irreducible component of $X_t$ and none of the irreducible components of $X_t \cap \supp(\Delta)$ is contained in the singular locus of $X_t$ for any $t\in T$ 
(so that the restriction $\Delta_t:=\Delta_{|X_t}$ is well-defined for any $t\in T$), and any fiber $(X_t,\Delta_t)$ is semi-log canonical.
\end{setup}

Note that if $f: (X,\Delta) \to T$ a generic lc-family then, for a general fiber $(F,\Delta_F)$, we have that
\begin{equation}\label{E:rest-KD}
(m(K_{X/T}+\Delta))_{|F}=m(K_F+\Delta_F) \text{ for any } m\in \N.
\end{equation}
Indeed, the above equality holds on the smooth locus of $X$, and then it extends to $X$ since $X$ is normal.
 
For generic slc families, there are powerful results of Fujino \cite{Fujino}, which determine the nefness of the push-forward of the powers of the relative dualizing sheaf.

\begin{theorem}[{\cite[Theorems 1.10 and 1.11]{Fujino}}]\label{T:Fuj}
Let $f\colon (X,\Delta)\to T$ be a generic slc-family as in \ref{N:KSBgen}. 
\begin{enumerate}
	\item \label{T:Fuj1}  Assume that $\Delta$ is a reduced Weil divisor. Then $f_*\O_X(K_{X/T} + \Delta)$ is nef. 
	
	\item \label{T:Fuj2}  Assume that $\Delta$ is a reduced Weil divisor and that $K_{X/T}+\Delta$ is $f$-semiample. Then $f_*\O_X(m(K_{X/T}+\Delta))$ is nef for all integers $m\geq 1$.

	\item \label{T:Fuj3}  Assume that $m(K_{X/T}+\Delta)$ is Cartier and  $f$-globally generated for some positive integer $m$.  Then $f_*\O_X(m(K_{X/T}+\Delta))$ is nef.

\end{enumerate}

\end{theorem}

\begin{corollary}\label{C:Fuj}
	\noindent 
	\begin{enumerate}
		\item \label{C:Fuj1} Let $f\colon (X,\Delta)\to T$ be a generic lc-family as in \ref{N:KSBgen}. Then $f_*\O_X(K_{X/T}+\Delta)$ is nef.
		
		\item \label{C:Fuj2} Let $f\colon (X,\Delta)\to T$ be a generic slc-family as in \ref{N:KSBgen} such that $K_{X/T}+\Delta$ is $f$-semiample. Then  $K_{X/T}+\Delta$ is nef. 
		
	\end{enumerate}
\end{corollary}

\begin{proof}
	Proof of \eqref{C:Fuj1}. When $\Delta$ is integral and reduced, this is exactly Theorem \ref{T:Fuj}(\ref{T:Fuj1}). An easy case is when $K_{X/T}+\lfloor \Delta \rfloor$ is $\Q$-Cartier restricted to the generic fiber and the coefficients of $\Delta$ are at most 1:  in this case, $f\colon (X,\lfloor \Delta \rfloor)\to T$ is a generic lc family, and we can apply Theorem \ref{T:Fuj}(\ref{T:Fuj1}) to prove that $f_*\O_X(K_{X/T}+\lfloor \Delta \rfloor)$ is nef. This last sheaf is equal to $f_*\O_X(K_{X/T}+\Delta)$ as $K_{X/T}$ is integral. We now give an argument to reduce the general case to Theorem \ref{T:Fuj}(\ref{T:Fuj1}). 
	
First we reduce to the case where all coefficients of $\Delta$ are at most one. As the family is generically lc, the divisor $\Delta_{>1}$ is vertical. We thus have an inclusion
$$ f_*\O_X(K_X+\Delta_{\leq 1})\hookrightarrow f_*\O_X(K_X+\Delta) \,, $$
which is an isomorphism on an open dense subset of $T$. By the well-know \cite[Lemma 2.2]{Fujino}, if the left hand side is nef, the right hand side is nef too, so we can assume without loss of generality that all the coefficients of $\Delta$ are smaller than or equal to one.

Let $h\colon Z \to X$ be a log-resolution of  $(X,\Delta)$ so that $\Delta_Z:= h_*^{-1}\Delta + E$ is simple normal crossing, where $E$ is the reduced exceptional divisor of $h$. Let $U$ be the open dense subset of $T$ over which $(X,\Delta)$ is lc. We have a natural inclusion
$$
i\colon	f_*h_*\O_Z(K_{Z/T} + \lfloor \Delta_Z \rfloor) \hookrightarrow f_*\O_X(K_{X/T} + \lfloor \Delta \rfloor)
$$
which is an isomorphism over $U$ by  \cite[Proposition 2.18]{Kollar_book}. 
	
Since $(Z, \Delta_Z) \to T$ is a generic lc-family, we can apply Theorem \ref{T:Fuj}(\ref{T:Fuj1}) to conclude that $f_*h_*\O_Z(K_{Z/T} + \lfloor \Delta_Z \rfloor)$ is nef.

\medskip

Proof of \eqref{C:Fuj2}. Let $k$ be a positive integer such that $\O_X(k(K_{X/T}+\Delta))$ is a locally free and $f$-globally generated.  
Then the evaluation map $f^* f_*\mathcal{O}_X(k(K_{X/T}+\Delta))\to \O_X(k(K_{X/T}+\Delta))$ is surjective.  Since $f_*\mathcal{O}_X(k(K_{X/T}+\Delta))$ is nef by Theorem \ref{T:Fuj}(\ref{T:Fuj3}), we conclude that $K_{X/T}+\Delta$ is nef.
\end{proof}
Under a stronger assumption on the singularities, we also have the following positivity result.
\begin{theorem}[{\cite[Theorem 7.1]{KP}}]\label{T:ampleness}
Let $f\colon (X,\Delta)\to T$ be a generic klt family as in \ref{N:KSBgen}. If $f$ is not isotrivial, then, for all sufficiently divisible positive integers $m$, the vector bundle $f_*\O_X(m(K_{X/T}+\Delta))$ is ample\footnote{The statement of \cite[Theorem 7.1]{KP} claims bigness rather than ampleness. The definition of big vector bundles used in loc. cit. is sometimes referred to as V-bigness, or Viehweg-bigness in the literature. On a curve, V-big is is equivalent to ample, i.e. V-big is equivalent to the strict positivity of all the slopes of the Harder-Narasimhan filtration.}, i.e. $\mu_-(f_*\O_X(m(K_{X/T}+\Delta))) >0$.
\end{theorem}
Theorem \ref{T:ampleness} is false if one replaces klt with lc, see \cite[Examples 7.5, 7.6 and 7.7]{KP}.

We now combine the results of Section \ref{Sec:slope} with the above nefness results of Fujino in order to obtain some slope inequalities for generic lc families. 

\begin{proposition}\label{P:slopepair}
Let $f\colon (X,\Delta)\to T$ be a generic lc-family as in \ref{N:KSBgen}. 
\begin{enumerate}
\item \label{P:slopepair1}  
Assume that there exists  $m \in \N_{>0}$ such that  at least one of the following conditions holds true 
\begin{itemize}
\item $m(K_{X/T}+\Delta)$  is Cartier, $f$-globally generated and $f$-big;
\item $K_{X/T}+\Delta$ is $f$-semiample, $\Delta$ is a reduced Weil divisor and $m(K_{F}+\Delta_F)$  is Cartier, globally generated and big.
\end{itemize}
Then
$$ 
m^{n+1}(K_{X/T}+\Delta)^{n+1} \geq 
\begin{cases}
\frac{4m^n(K_F+\Delta_F)^n}{m^n(K_F+\Delta_F)^n+2n} \deg f_*\O_X(m(K_{X/T}+\Delta)) &  \text{ if  either } \dim F\geq 2 \text{ and } \kappa(F)\geq 0, \\
	  & \text{ or } \dim F=1\: \text{ and } \Delta_F=0,\\
\frac{2m^n(K_F+\Delta_F)^n}{m^n(K_F+\Delta_F)^n+n} \deg f_*\O_X(m(K_{X/T}+\Delta)) & \text{ otherwise.} \\
\end{cases}
$$ 
 If $n\geq 2$, the family is generically klt, and one of the above inequalities is an equality for all $m$ divisible enough, then $f$ is isotrivial.
	
\item \label{P:slopepair2} 
\begin{enumerate}[(a)]
\item  \label{P:slopepair2a}  
Assume that there exist  $m,q \in \N_{>0}$ such that  $K_{X/T} + \Delta$ is $f$-semiample, $\Delta$ is a reduced Weil divisor and $\phi_{mq(K_{F} + \Delta_F)}$ is generically finite.
Then 
$$
m^{n+1}(K_{X/T}+ \Delta)^{n+1}  \geq 
\begin{cases}
2\frac{\deg f_*\O_X(m(K_{X/T}+\Delta))}{q^n} &  \text{ if  either } \dim F\geq 2 \text{ and } \kappa(F)\geq 0, \\
	  & \text{ or } \dim F=1\: \text{ and } \Delta_F=0,\\
\frac{\deg f_*\O_X(m(K_{X/T}+\Delta))}{q^n} & \text{ otherwise.} \\
\end{cases}
$$

\item  \label{P:slopepair2b}
Assume that there exists  $m, q \in \N_{>0}$ such that at least one of the following conditions holds true
\begin{itemize}
\item $m(K_{X/T}+\Delta)$  is Cartier, $f$-globally generated and $f$-big;
\item $K_{X/T}+\Delta$ is $f$-semiample, $\Delta$ is a reduced Weil divisor and $mq(K_{F}+\Delta_F)$  is Cartier, globally generated and big.
\end{itemize}
 Then 
$$
m^{n+1}(K_{X/T}+ \Delta)^{n+1} \ge \frac{\deg f_*\O_X(m(K_{X/T}+\Delta))}{q^n}.
$$

\end{enumerate}
\item \label{P:slopepair3} Assume that $K_{X/T} + \Delta$ is nef.
\begin{enumerate}[(a)]
\item \label{P:slopepair3a} If   $q \in \N_{>0}$ is such that $\phi_{q(K_{F} + \Delta_F)}$ is generically finite then 
$$(K_{X/T}+ \Delta)^{n+1} \ge 
\begin{cases}
2\frac{\deg f_*\O_X(K_{X/T}+\Delta)}{q^n}. &  \text{ if  either } \dim F\geq 2 \text{ and } \kappa(F)\geq 0, \\
	  & \text{ or } \dim F=1\: \text{ and } \Delta_F=0,\\
\frac{\deg f_*\O_X(K_{X/T}+\Delta)}{q^n} & \text{ otherwise.} \\
\end{cases}
$$
\item \label{P:slopepair3b} If   $q \in \N_{>0}$ is such that $q(K_{F} + \Delta_F)$ is Cartier and big then 
$$(K_{X/T}+ \Delta)^{n+1} \ge \frac{\deg f_*\O_X(K_{X/T}+\Delta)}{q^n}.$$
\end{enumerate}
\end{enumerate}

\end{proposition}

\begin{proof}
Proof of \eqref{P:slopepair1}.	Note that we are in the set-up \ref{setup} with $L=m(K_{X/T}+\Delta)$ and we can apply either Theorem  \ref{T:Fuj}\eqref{T:Fuj2} or Theorem  \ref{T:Fuj}\eqref{T:Fuj3} to get that $f_*\O_X(L)$ is nef and Corollary \ref{C:Fuj}\eqref{C:Fuj2} to get that $L$ is nef. The conclusion follows then from  Corollary \ref{C:Xiao-high2}, using \eqref{E:rest-KD}. The  equality case follows from the above mentioned results combined with Theorem \ref{T:ampleness}.
\medskip

Proof of \eqref{P:slopepair2}. We are in the set-up \ref{setup} with  $L=m(K_{X/T}+\Delta)$ and we can apply either Theorem  \ref{T:Fuj}\eqref{T:Fuj2} or Theorem  \ref{T:Fuj}\eqref{T:Fuj3} to get that $f_*\O_X(L)$ is nef and  Corollary \ref{C:Fuj}\eqref{C:Fuj2} to get that $L$ is nef.
 The conclusion follows then from Theorem \ref{T:Barja}, using \eqref{E:rest-KD}. 

\medskip

Proof of \eqref{P:slopepair3}.	We are in the set-up \ref{setup} with $L=K_{X/T}+\Delta$ and we can apply Corollary \ref{C:Fuj}\eqref{C:Fuj1} to get that $f_*\O_X(K_{X/T}+\Delta)$ is nef. The conclusion follows then from Theorem \ref{T:Barja}, using \eqref{E:rest-KD}.
\end{proof}

We now extend some of the above slope inequalities from generic lc families to generic slc families. 

\begin{theorem}\label{T:Nef-Fam}
Let $f\colon (X,\Delta)\to T$ be a generic slc-family as in \ref{N:KSBgen}.
\begin{enumerate}
\item \label{T:Nef-Fam1}
Assume that there exists  $m \in \N_{>0}$ such that  at least one of the following conditions hold true 
\begin{itemize}
\item $m(K_{X/T}+\Delta)$  is Cartier, $f$-globally generated and $f$-big;
\item $K_{X/T}+\Delta$ is $f$-semiample, $\Delta$ is a reduced Weil divisor and $m(K_{F}+\Delta_F)$  is Cartier, globally generated and big.
\end{itemize}
Let $w\in \Q_{>0}$ such that the volume of the pull-back of $K_F+\Delta_F$ to any irreducible component of the normalisation of $F$ is at least $w$. Then
$$ m^{n+1}(K_{X/T}+\Delta)^{n+1}\geq \frac{2wm^n}{wm^n+n}\deg \left(f_*\mathcal{O}_X(m(K_{X/T}+\Delta))\right).$$

\item \label{T:Nef-Fam2} 
Assume  there exist  $m,q \in \N_{>0}$ such that at least one of the following conditions holds true 
\begin{itemize}
        \item $m(K_{X/T} + \Delta)$ is Cartier, $f$-globally generated and $f$-big.
	\item $K_{X/T} + \Delta$ is $f$-semiample, $\Delta$ is a reduced Weil divisor and either $\phi_{qm(K_{F} + \Delta_F)}$ is generically finite or $qm(K_{F} + \Delta_F)$ is Cartier, globally generated and big. 
	\end{itemize}
Then 
$$
m^{n+1}(K_{X/T}+ \Delta)^{n+1} \ge \frac{\deg f_*\O_X(m(K_{X/T}+\Delta))}{q^n}.
$$

\item \label{T:Nef-Fam3}  Assume that $K_{X/T} + \Delta$ is nef and let $q\in \N_{>0}$ such that either $q(K_F+\Delta_F)$ is Cartier and big or $\phi_{q(K_F+\Delta_F)}$ is generically finite. Then
$$(K_{X/T}+\Delta)^{n+1}\geq \frac{\deg\left(f_*\mathcal{O}_X(K_{X/T}+\Delta)\right)}{q^n}.$$

\end{enumerate}
\end{theorem}
\begin{proof}
We first make some considerations useful for all the three cases, and then we focus on the specific inequalities.

Let $\nu \colon Y\to X$ be the normalisation of $X$ and set $g:Y\xrightarrow[]{\nu} X \xrightarrow[]{f} T$. 
We have that 
\begin{equation}\label{E:pullK}
 \nu^*( K_{X/T}+\Delta)=K_{Y/T}+\Delta', 
 \end{equation}
for some  boundary $\Q$-divisor $\Delta'$ on $Y$ with the property that $g:(Y,\Delta')\to T$ is a disjoint union of generic lc-families $g_i:(Y_i,\Delta_i)\to T$, see \cite[Section 5.1 and Definition-Lemma 5.10]{Kollar_book}. The divisor $K_{X/T}+\Delta$ is nef either by assumption or by Corollary \ref{C:Fuj}\eqref{C:Fuj2}; combining this with \eqref{E:pullK} we obtain that $K_{Y_i/T}+\Delta_i$ is nef for every $i$. Moreover, if $\Delta$ is integral, $\Delta'$ is integral too.

By \eqref{E:pullK} and the projection formula, we have that 
\begin{equation}\label{E:inter-Delta}
(K_{X/T}+\Delta)^{n+1}=(K_{Y/T}+\Delta')^{n+1}=\sum_i (K_{Y_i/T}+\Delta_i)^{n+1}.
\end{equation}

Using the natural injection $\O_X\hookrightarrow \nu_*\O_Y$ coming from the normalisation map $\nu$ and \eqref{E:pullK}, we get the following injection of reflexive sheaves on  $X$
$$
\O_X(m(K_{X/T}+\Delta))\hookrightarrow \nu_*\left(\nu^*(\O_X(m(K_{X/T}+\Delta))) \right)=\nu_*\left(\O_Y(m(K_{Y/T}+\Delta')) \right)=\nu_*\left(\bigoplus_i \O_{Y_i}(m(K_{Y_i/T}+\Delta_i)) \right).
$$
By taking the push-forward along $f$, we get the following injection of locally free sheaves on $T$ 
$$V:=f_* \mathcal{O}_X(m(K_{X/T}+\Delta))\hookrightarrow  W:=\bigoplus_i g_{i*}\left(\mathcal{O}_{Y_i}(m(K_{Y_i/T}+\Delta_i))\right).$$
The locally free sheaf $W$ is nef in each of the cases we are considering by either Theorem \ref{T:Fuj} or Corollary \ref{C:Fuj}\eqref{C:Fuj1} (see also the proof of Proposition \ref{P:slopepair}); hence also the quotient $W/V$ is nef and  it has non-negative degree. Therefore,  by taking degrees in the previous inclusion we get that 
\begin{equation}\label{E:push-Delta}
\deg f_*\O_X(m(K_{X/T}+\Delta)) \leq  \sum_i \deg g_{i*}\left(\mathcal{O}_{Y_i}(m(K_{Y_i/T}+\Delta_i))\right).
\end{equation}

\medskip

Part \eqref{T:Nef-Fam1}  follows now from Proposition \ref{P:slopepair}\eqref{P:slopepair1}  applied to the generic lc-families $g_i:(Y_i,\Delta_i)\to T$, noting that 
$$
\frac{2m^n(K_{F_i}+\Delta_{F_i})^n}{m^n(K_{F_i}+\Delta_{F_i})^n+2n}\geq \frac{2m^nw}{m^nw+2n},
$$
where $(F_i,\Delta_{F_i})$ is a general fiber of $g_i:(Y_i,\Delta_i)\to T$.

\medskip

Part \eqref{T:Nef-Fam2} follows from Proposition \ref{P:slopepair}\eqref{P:slopepair2}  applied to the generic lc-families $g_i:(Y_i,\Delta_i)\to T$, and using that 
if $\phi_{m(K_F+\Delta)}$ is generically finite, the same is true for $\phi_{m(K_{F_i}+\Delta_i)}$ by Lemma \ref{L:gen_fin}.  

%We then conclude by Proposition \ref{P:slopepair}\eqref{P:slopepair2} to the generic lc-families $g_i:(Y_i,\Delta_i)\to T$.
%If $m(K_{X/T} + \Delta)$ is Cartier and $f$-globally generated, we apply Proposition \ref{P:slopepair}\eqref{P:slopepair2} to the generic lc-families $g_i:(Y_i,\Delta_i)\to T$.

\medskip

Part \eqref{T:Nef-Fam3}  follows from Proposition \ref{P:slopepair}\eqref{P:slopepair3} applied to the generic lc-families $g_i:(Y_i,\Delta_i)\to T$.
\end{proof}

An interesting consequence of the above results is the following 

\begin{corollary}[Existence of slope inequalities]\label{C:existence}
	Fix an integer $n\geq 1$ and a subset $I$ of $[0,1]$ satisfying the DCC,  then there exists a  constant $s(n,I)>0$ (which depends only on $n$ and $I$) such that
	$$(K_{X/T}+\Delta)^{n+1}\geq s(n,I)\deg f_*\O_X(K_{X/T}+\Delta) $$
	for every  generic slc family $f\colon (X,\Delta)\to T$  such that $\dim X=n+1$,  the coefficients of $\Delta_{\leq 1}$ belong to $I$ and $K_{X/T}+\Delta$ is $f$-semiample and $f$-big. 
\end{corollary}
\begin{proof}
We can assume, without loss of generality, that $I$ contains $1$. By \cite[Theorem 1.3]{ACC}, there exists a constant $b(n,I)>0$ (which depends only on $n$ and $I$) such that $b(n,I)(K_Z+\Delta_Z)$ gives a birational morphism for all lc pairs $(Z,\Delta_Z)$ such that the dimension of $Z$ is $n$, the coefficients of $\Delta_Z$ belong to $I$, and $K_Z+\Delta_Z$ is big.  We claim that we can take
$$
s(n,I) = \frac{1}{b(n,I)^n}.
$$ 

Indeed, let $f \colon (X,\Delta) \to T$ be a generic slc family as in the statement. Following the construction of the first part of the proof of Theorem \ref{T:Nef-Fam}, consider the lc families $g_i\colon (Y_i,\Delta_i)\to T$ obtained by normalising $X$. Since $g_{i*}\left(\mathcal{O}_{Y_i}(K_{Y_i/T}+\Delta_i)\right)$ is nef by Corollary \ref{C:Fuj}\eqref{C:Fuj1},  by arguing as in the proof of Theorem \ref{T:Nef-Fam}, we get that 
\begin{equation}\label{E:push-Delta2}
\deg f_*\O_X(K_{X/T}+\Delta) \leq  \sum_i \deg g_{i*}\left(\mathcal{O}_{Y_i}(K_{Y_i/T}+\Delta_i)\right).
\end{equation}
Similarly, we get that 
\begin{equation}\label{E:inter-Delta2}
(K_{X/T}+\Delta)^{n+1}=\sum_i (K_{Y_i/T}+\Delta_i)^{n+1}.
\end{equation}

As $f$ is a generically slc family, the coefficients of $\Delta_{>1}$ are vertical, hence, if we restrict $\Delta_i$ to a generic fiber of $g_i$, we obtain a boundary divisor with coefficients in $I$ (we had to make sure that $1$ is in $I$, as $\Delta_i$ might contain divisors with coefficients one coming from the conductor). By assumption, $b(n,I)(K_{Y_i/T}+\Delta_i)$ gives a birational morphism when restricted to a generic fiber of $g_i$. Moreover, $K_{Y_i/T}+\Delta_i$ is nef by Corollary \ref{C:Fuj}\eqref{C:Fuj2} and the assumption that $K_{X/T}+\Delta$ is $f$-semiample. 
Hence, we can apply Proposition \ref{P:slopepair}\eqref{P:slopepair3} in order to get that 
\begin{equation}\label{E:ine-gi}
(K_{Y_i/T}+\Delta_i)^{n+1}\geq \frac{\deg g_{i*}\left(\mathcal{O}_{Y_i}(K_{Y_i/T}+\Delta_i)\right)}{b(n,I)^n}.
\end{equation}
We conclude by putting together \eqref{E:push-Delta2}, \eqref{E:inter-Delta2} and \eqref{E:ine-gi}. 
\end{proof}

We conclude with an application of Theorem \ref{T:Xiao-n} to generic canonical families with no boundary.

\begin{corollary}\label{C:can_sing}
Let $f\colon X\to T$ be a generic  canonical family such that $K_{X/T}$ is $f$-semiample, $\phi_{K_F}$ is a birational map and $n\geq 2$. Then 
\begin{equation*}
K_{X/T}^{n+1} \ge 2(n+1)\frac{h^0(F,K_F)-n-2}{h^0(F,K_F)}\deg f_*\O_X(K_{X/T}).
\end{equation*}
If, furthermore, $K_F$ is Cartier and globally generated, then we have 
$$
K_{X/T}^{n+1} \ge 2(n+1)\frac{K_{F}^n}{K_{F}^n +(n+1)(n+2)} \deg f_*\O_X(K_{X/T}). 
$$
\end{corollary}
\begin{proof}
The divisor $K_{X/T}$ and the vector bundle $f_*\O_X(K_{X/T})$ are nef by Corollary \ref{C:Fuj}. We conclude by applying Corollaries \ref{C:Xiao-bir1} and \ref{C:Xiao-bir2} with $L= K_{X/T}$ and $s=1$.
%Since the singularities of $F$ are canonical, the canonical bundle of $\wt F$ is effective and induces a birational map , so the hypotheses of Theorem \ref{T:Xiao-n} with $L=K_{X/T}$ and $s=0$ are satisfied. As $L$ is nef, by Lemma \ref{L:Lnef} item \ref{L:Lnef2} we have $L^{n+1}\geq M_{\ell}^{n+1}$. We conclude noticing that $h^0(\wt F, P_{\ell})=h^0(F,K_F)$ by Proposition \ref{P:Mnef} item \ref{P:Mnef4}.
\end{proof}

\subsection{Application to the moduli space of KSB stable varieties}\label{sec:KSBmoduli}

Let $\M_{n,v}$ be the Deligne-Mumford algebraic stack of KSB stable varieties of dimension $n$ and volume $v$, i.e. varieties $V$ with slc singularities, ample dualizing $\Q$-divisor $K_V$, dimension $n$ and volume $v:=K_V^n\in \Q_{>0}$.  We denote by  $\M_{n,v}^o$ the open substack of $\M_{n,v}$ parametrizing normal stable varieties. We refer to \cite{Kollar_moduli} for the definition and main properties of families of KSB stable varieties.

Using the universal family $\pi\colon \X_{n,v}\to \M_{n,v}$, we can define  the following $\Q$-Cartier $\Q$-divisors (well-defined up to $\Q$-linear equivalence) on $\M_{n,v}$:
\begin{itemize}
\item the Chow-Mumford divisor  $\lambda_{CM}:=\pi_*K_{\pi}^{n+1}$;
\item the determinant divisors $\lambda_m=c_1(\det \left(\pi_*\mathcal{O}_{\mathcal{X}}(mK_{\pi}))\right)$ for any $m\geq 0$. 

\end{itemize}
Recall that $\lambda_{CM}$ is ample on $\M_{n,v}$ by \cite{PX}, while $\lambda_m$  are nef on $\M_{n,v}$ for any $m$ big and divisible enough by \cite{Fujino} (see also Theorem \ref{T:Fuj}). The nefness of $\lambda_m$ combined with Koll\'{a}r's ampleness lemma \cite{Kollar_ampleness}, shows that $\lambda_m$ is ample on $\M_{n,v}$ if  $m$ is big and divisible enough. Let us stress that for low values of $m$, $\lambda_m$ does not need to be ample: for instance, when $n=1$, $\lambda_m$ is ample for $m\geq 2$, but just semi-ample for $m=1$.  For $n\geq 2$, we are not aware of any effective lower bound on $m$ that guarantees the ampleness of $\lambda_m$.

The aim of this section is to determine some explicit rational functions $f(m)\in \Q(m)$ for which the  $\Q$-divisor $m^{n+1}\lambda_{CM}-f(m) \lambda_m$ is \emph{nef on $\M_{n,v}$}, i.e. it intersects non-negatively all the projective curves of  $\M_{n,v}$, or \emph{nef on $\M_{n,v}$ away from the boundary},  i.e. it intersects non-negatively all the projective curves of  $\M_{n,v}$ not entirely contained in the boundary $\partial \M_{n,v}:=\M_{n,v}\setminus \M_{n,v}^o$. 
The same definitions can be given for any closed substack of $\M_{n,v}$, for example for any irreducible or connected component.  

\vspace{0.1cm}

Let us first describe some $\Q$-divisors of the form $m^{n+1}\lambda_{CM}-f(m) \lambda_m$ that are nef on $\M_{n,v}$ away from the boundary.

\begin{theorem}\label{T:Nef-Away}
Fix $n\in \N_{>0}$ and $v\in \Q_{>0}$. Let $m$ and $q$ be two positive integers.
\begin{enumerate}
\item\label{T:Nef-Away1} Assume that $mK_V$ is Cartier and globally generated for  any $V\in \M_{n,v}^o$. Then:
\begin{itemize}
\item If either $n\geq 2$ or $n=m=1$ then 
$$m^{n+1} \lambda_{CM}-\frac{4vm^n}{vm^n+2n}\lambda_m$$
is nef on $\M_{n,v}$ away from the boundary;
\item If $n=1$ and $m\geq 2$, then 
$$m^{n+1} \lambda_{CM}-\frac{2vm^n}{vm^n+n}\lambda_m$$
is nef on $\M_{n,v}$ away from the boundary
\end{itemize}

\item \label{T:Nef-Away2} If $\phi_{mqK_V}$ is generically finite for any $V\in \M_{n,v}^o$ and either $n\geq 2$ or $n=q=m=1$, then 
$$m^{n+1}\lambda_{CM}-2 \frac{\lambda_m}{q^n}$$
is nef on $\M_{n,v}$ away from the boundary.
%VECCHIO TESTO 
%Assume either $n=1$ and $m=1$, or $n\geq 2$ and let $m\in \N$ such that  $\phi_{mK_V}$ is generically finite for any $V\in \M_{n,v}^o$. Then the $\Q$-divisors 
%$$m^{n+1}\lambda_{CM}-2 \lambda_m \quad \textrm{and} \quad m^n\lambda_{CM}-2\lambda_1 $$
%are nef on $\M_{n,v}$ away from the boundary.
\item\label{T:Nef-Away3} If either $mqK_V$ is Cartier  or $\phi_{mqK_V}$ is generically finite for any $V\in \M_{n,v}^o$, then 
$$
m^{n+1}\lambda_{CM}-\frac{\lambda_m}{q^n}
$$
is nef on $\M_{n,v}$ away from the boundary.
\end{enumerate}
\end{theorem}
The same conclusions hold if we replace $\M_{n,v}$ by any closed substack, for example any irreducible or connected component.  

\begin{proof}
We have to prove that the degree of the given divisors are non-negative on any irreducible smooth projective curve $T\subset \M_{n,v}$ not entirely contained in $\partial \M_{n,v}$. 

Let $f:X\to T$ be the restriction of the universal family $\pi:\X_{n,v}\to \M_{n,v}$ to $T$. 
By assumption, a general fiber $F$ of $f$ is normal (because it belongs to $\M_{n,v}^o$), which implies, using that all the fibers of $f$ are reduced, that $X$ is normal. Moreover, the divisor $K_{X/T}$ is $\Q$-Cartier by \cite[Theorem 2.3]{Kollar_moduli}.
Hence the fibration $f$ is  a generic lc-family as in \ref{N:KSBgen}.

Now observe that 
\begin{equation}\label{E:res-div-T}
\begin{sis}
& \deg (\lambda_{CM})_{|T}=K_{X/T}^{n+1}, \\
& \deg (\lambda_m)_{|T}=\deg f_*\O_X(mK_{X/T}),
\end{sis}
\end{equation}
where in the second equality we have used the base change property $\pi_*\O_{\X_{n,v}}(qK_{\pi})|_T=f_*\O_X(qK_{X/T})$, which follows, using the Koll\'ar condition on families of KSB stable varieties,   from the fact that  $\pi_*\O_{\X_{n,v}}(qK_{\pi})$ is locally free (see \cite[Rmk. 3.4]{Fujino} for $q\geq 2$ and \cite[Corollary 2.71]{Kollar_moduli} for $q=1$). 

The results follow by applying Proposition \ref{P:slopepair} with $\Delta=0$ and using that $K_{X/T}$ is $f$-semiample (being $f$-ample) and $K_F$ is big (being ample).  More precisely:
\begin{itemize}
\item part \eqref{T:Nef-Away1} follows from Proposition \ref{P:slopepair}\eqref{P:slopepair1};
\item part \eqref{T:Nef-Away2} follows from Proposition \ref{P:slopepair}\eqref{P:slopepair2a};
\item part \eqref{T:Nef-Away3} follows from Proposition \ref{P:slopepair}\eqref{P:slopepair2b}. 
\end{itemize}

\end{proof}

\begin{remark}\label{R:AwayCurves}
Some of the results of Theorem \ref{T:Nef-Away} are sharp for the moduli space $\ov\M_g(=\M_{1,2g-2})$ of stable curves of genus $g\geq 2$. Indeed, it follows from \cite[Prop. 4.3, Thm. 4.12]{CH} that 
\begin{equation}\label{E:AwayCurves}
\lambda_{CM}- s \lambda_1 \: \text{ is nef on } \ov\M_g \: \text{ away from the boundary } \Leftrightarrow s\leq \frac{4g-4}{g}.
\end{equation}
This shows that if $n=1$ then 
\begin{itemize}
\item  Theorem \ref{T:Nef-Away}\eqref{T:Nef-Away1} is sharp if  $m=1$ (for any volume $v$);
 \item Theorem \ref{T:Nef-Away}\eqref{T:Nef-Away2} is sharp if $m=q=1$ and $v=2$ (i.e. genus $g$ equal to $2$). 
\end{itemize}
\end{remark}

We now describe some $\Q$-divisors of the form $m^{n+1}\lambda_{CM}-f(m) \lambda_m$ that are nef on $\M_{n,v}$.

\begin{theorem}\label{T:Nef-Mod}
Fix $n\in \N_{>0}$ and $v\in \Q_{>0}$. Consider two positive integers $m$ and $q$. 
\begin{enumerate}
\item \label{T:Nef-Mod1}
Assume that $mK_V$ is Cartier and globally generated for any $V\in \M_{n,v}$ and let $w\in \Q_{>0}$ such that the volume of the pull-back of $K_F$ to any irreducible component of the normalisation of $V$ is at least $w$. Then the $\Q$-divisor 
$$m^{n+1}\lambda_{CM}-\frac{2wm^n}{wm^n+n}\lambda_m  $$
is nef on $\M_{n,v}$. 
\item \label{T:Nef-Mod2} 
If either $mqK_V$ is Cartier  or $\phi_{mqK_V}$ is generically finite for any $V\in \M_{n,v}$, then the $\Q$-divisor
$$
m^{n+1}\lambda_{CM}-\frac{\lambda_m}{q^n}
$$
is nef on $\M_{n,v}$.
\end{enumerate}
\end{theorem}
The same conclusions hold if we replace $\M_{n,v}$ by any closed substack, for example any irreducible or connected component.
\begin{proof}
It is enough to prove that the degree of the given divisors are non-negative on any irreducible smooth projective curve $T\subset \M_{n,v}$. 

Let $f:X\to T$ be the restriction of the universal family $\pi:\X_n\to \M_n$ to $T$. 
By  \cite[Theorem 2.3]{Kollar_moduli}, $X$ is deminormal and the divisor $K_{X/T}$ is $\Q$-Cartier. Hence, the fibration $f$ is  a generic slc-family as in \ref{N:KSBgen}.

Now observe that 
\begin{equation}\label{E:res-div-T2}
\begin{sis}
& \deg (\lambda_{CM})_{|T}=K_{X/T}^{n+1}, \\
& \deg (\lambda_m)_{|T}=\deg f_*\O_X(mK_{X/T}),
\end{sis}
\end{equation}
where in the second equality we have used the base change property $\pi_*\O_{\X_{n,v}}(qK_{\pi})|_T=f_*\O_X(qK_{X/T})$ (see the proof of Theorem \ref{T:Nef-Away}). 
%As the relative canonical bundle satisfies base-change, we have $\deg\left(\lambda_{CM}|_T\right)=K_{X/T}^{n+1} $. By \cite[Corollary 2.71]{Kollar_book}, we have $\lambda_m|_T= f_*\O_X(mK_{X/T}) $. 

The results follow by applying Theorem \ref{T:Nef-Fam}\eqref{T:Nef-Fam1} or   \ref{T:Nef-Fam}\eqref{T:Nef-Fam3} with $\Delta=0$ and using that $K_{X/T}$ is $f$-semiample (being $f$-ample) and $K_F$ is big (being ample). 
\end{proof}

\begin{remark}\label{R:NefCurves}
Some of the results in Theorem \ref{T:Nef-Mod} are sharp for the moduli space $\ov\M_g(=\M_{1,2g-2})$ of stable curves of genus $g\geq 2$. Indeed, it follows from \cite[Prop. 4.3, Thm. 4.12]{CH} that 
\begin{equation}\label{E:NefCurves}
\lambda_{CM}- a \lambda_1 \: \text{ is nef on } \ov\M_g  \Leftrightarrow a\leq 1.
\end{equation}
This shows that if $n=1$ then 
\begin{itemize}
\item  part \eqref{T:Nef-Mod1} is sharp if $m=1$ and $w=1$, which is also the minimum possible value of $w$ that satisfies the assumptions of part \eqref{T:Nef-Mod1} for $m=1$, since the volume (=degree) of the canonical divisor of a stable curve $C$ is $1$ on any elliptic tail of $C$. 
 \item  part \eqref{T:Nef-Mod2} is sharp if $m=q=1$. 
\end{itemize}
\end{remark}

\begin{remark}\label{R:asympo}
For $m$ very large, the results of Theorem \ref{T:Nef-Mod} are far from being sharp. Indeed, let us define 
\begin{equation}\label{E:s(q)}
s(m):=\sup\{t\in \R_{\geq 0} \: : \: m^{n+1}\lambda_{CM}-t \lambda_m \quad \text{ is nef on } \M_{n,v}\}.
\end{equation}
By applying the Grothendieck-Riemann-Roch formula to the universal family $\pi:\X_{n,v}\to \M_{n,v}$, it can be shown (see \cite[Sec. 2.3]{PX} or \cite[Lemma A.2]{CP}) that 
\begin{equation}\label{E:lambda-asin}
\lambda_m\sim_{\Q} \left(\frac{m^{n+1}}{(n+1)!}-\frac{m^n}{2n!}\right)\lambda_{CM}+P^{n-1}(m) \quad \text{ for all sufficiently divisible } m\in \N, 
\end{equation}
where  $P^{n-1}(m)$ is a polynomial of degree at most $n-1$ in $m$ with coefficients being $\Q$-divisors on $\M_{n,v}$.  The divisibility condition on $m$ is needed to guarantee that the relative canonical bundle $mK_{\pi}$ is Cartier, and we do not know if it is really necessary.

Hence, from the above asymptotic formula for $\lambda_m$, it follows that 
\begin{equation}\label{E:nef-asin}
s(m) \sim \frac{2(n+1)! m}{2m-(n+1)}\sim (n+1)! \quad \text{ as } m\to \infty \; \text{ and } m \text{ is sufficiently divisible}
\end{equation}

However, note that  the asymptotic formula \eqref{E:nef-asin} is not effective while the results of Theorem \ref{T:Nef-Mod}, although asymptotically worse, are however effective.
\end{remark}

Let us define the \emph{lambda Neron-Severi space} $\NS^{\Lambda}_{\Q}(\M_{n,v})$ as the linear subspace of the rational Neron-Severi vector space $\NS_{\Q}(\M_{n,v})$ spanned by the Chow-Mumford line bundle $\lambda_{CM}$ and the classes $\lambda_m$ for any $m\geq 1$, and consider the \emph{lambda nef cone}  $\Nef^{\Lambda}(\M_{n,v})\subset \NS^{\Lambda}_{\Q}(\M_{n,v})$, which is the closed convex cone equal to the intersection of the nef cone $\Nef(\M_{n,v})\subset \NS_{\Q}(\M_{n,v})$ with $\NS^{\Lambda}_{\Q}(\M_{n,v})$.

In the case of the moduli space of stable curves, i.e $n=1$, the following facts are well-known:
\begin{itemize}
\item the space $\NS^{\Lambda}_{\Q}(\M_{n,v})$ is two dimensional: a basis is given by $\lambda_{CM}$ and $\lambda_1$, while the other classes are equal to (see \cite[Chap. XIII, Thm. (7.6)]{GAC2})
$$
\lambda_m=\binom{m}{2}\lambda_{CM}+\lambda_1.
$$ 
\item the two extremal rays of the two dimensional cone $\Nef^{\Lambda}(\M_{n,v})$ are spanned by, respectively, $\lambda_1$, which is semi-ample and gives the map towards the Satake compactification (see \cite[pages 435-437]{GAC2}), and $\lambda_{CM}-\lambda_1$ (which is the class given by either Theorem \ref{T:Nef-Mod}\eqref{T:Nef-Mod1} with $m=w=1$ or Theorem \ref{T:Nef-Mod}\eqref{T:Nef-Mod2} with $m=q=1$), which is semi-ample and it gives the morphism toward the moduli stack of pseudo-stable curves (see \cite[Thm. 1.1]{HH}).

\end{itemize}

Let us conclude this section with the following two questions.

\begin{question}\label{Q:dim}
Is the dimension of $\NS^{\Lambda}_{\Q}(\M_{n,v})$ equal to $n+1$?
\end{question}

Note that, if relative canonical divisor $K_{\pi}$ of the universal family $\pi:\X_{n,v}\to \M_{n,v}$ is Cartier, then formula \eqref{E:lambda-asin} implies that $\NS^{\Lambda}_{\Q}(\M_{n,v})$ has dimension at most  $n+1$. 

\begin{question}
For which values of $n$ and $v$, do the classes $\lambda_m$ and the classes from Theorem \ref{T:Nef-Mod} give all extremal rays of $\Nef^{\Lambda}(\M_{n,v})$?
\end{question}

\section{Fano varieties}\label{Sec:Fano}

\subsection{Slope inequalities for families of K-stable and K-polystable log-Fano pairs}\label{sec:slopeFano}

In this subsection we prove some slope inequalities for families of Fano variety. We will need some stability assumption on the general member of the family.

We refer to \cite{CP,Xu_Zhuang}, to the survey \cite{survey} and the recent breakthrough \cite{LXZ} for background results about K-stability, and a comprehensive bibliography, here we recall just some notations and properties.

\begin{setup}\label{N:Kstab}
Let  $f\colon X \to T$ be a fibration from a normal projective irreducible variety $X$ of dimension $n+1$ to  a smooth projective irreducible curve $T$. 

Let $\Delta$ be a divisor on $X$ such that $-K_{X/T}-\Delta$ is $\Q$-Cartier and $f$-ample. We say that $f$ is a \emph{$\Q$-Gorenstein family of anti-canonically polarized pairs}  of dimension $n$. We denote by $(X_t,\Delta_t)$ the fiber of $f$ over $t\in T$, and we denote by  
 $v:=(-K_{X_t}-\Delta_{X_t})^n=((-K_{X/T}-\Delta)_{|X_t})^n$ the relative volume of $f$.

If the generic fiber of $f$ has klt singularities, then $f$ is a \emph{generic $\Q$-Gorenstein family of log-Fano pairs}. If all fibers of $f$ has klt singularities, then $f$ is a \emph{$\Q$-Gorenstein family of log-Fano pairs}. 

The Chow-Mumford (CM) $\Q$-divisor on $T$ is defined (up to $\Q$-linear equivalence) as
$$\lambda_{CM}(X/T):=-f_*(-K_{X/T}-\Delta)^{n+1}.$$

For a log-Fano pair $(F,\Delta_F)$, we denote by $\delta(F,\Delta_F)$ its stability threshold. A log-Fano pair is K-semistable if $\delta(F,\Delta_F)\geq 1$, it is K-stable if $\delta(F,\Delta_F)>1$. Both K-semistability and K-stability are open properties in families.

K-semistability implies that the pair is klt, so a $\Q$-Gorenstein family of anti-canonically polarized pairs with K-semistable generic fiber is a generic $\Q$-Gorenstein family of log-Fano pairs.
 
If the generic fiber of $f:X\to T$ is K-semistable, then we have
$$\deg \lambda_{CM}(X/T)=-(-K_{X/T}-\Delta)^{n+1}\geq 0.$$

\end{setup}

\begin{remark}
A K-semistable log-Fano pair, contrary to a KSB stable pair, is always klt and hence normal. Therefore, a family of anti-canonically polarized pairs with reduced fibers and K-semistable generic fiber has automatically normal total space. In other words, assuming that $X$ is normal, we are not ruling out any $\Q$-Gorenstein family of K-semistable log-Fano pairs. 
\end{remark}

\begin{remark}[Negativity in the Fano case]\label{rem:neg} Observe that, by \cite[Theorem A.13]{CPZ}, if $-K_{X/T}-\Delta$ is nef, then $f$ is locally \'{e}tale isotrivial. Combining this result with Lemma \ref{L:Lnef}\eqref{L:Lnef1}, if $f$ is not locally \'{e}tale isotrivial, then $f_*\O_X(m(-K_{X/T}-\Delta))$ is \emph{not} nef for all $m$ sufficiently divisible. This means that we do not expect to apply our results to $-K_{X/T}-\Delta$. 
However, assuming K-stability of the general fiber, we can prove some slope inequalities for convenient line bundles. The coefficients in these inequalities involve the delta invariant $\delta(F,\Delta_F)$ of $(F,\Delta_F)$, which  provides a quantitative description of the negativity of $f_*\O_X(m(-K_{X/T}-\Delta))$. 
\end{remark}

\begin{theorem}\label{T:F_slope1}
Let $f$ be a $\Q$-Gorenstein family of anti-canonically polarized pairs as in the set-up \ref{N:Kstab}, assume that a geometric fiber $(F,\Delta_F)$ is K-stable, i.e. $\delta:=\delta(F,\Delta_F)>1$. 

For any rational number $C>1$ consider the $\Q$-Cartier $\Q$-divisor on $X$
\begin{equation}\label{E:HC}
H_C:=-K_{X/T}-\Delta+C\frac{\delta}{(\delta-1)v(n+1)}f^* \lambda_{CM}(X/T).
\end{equation}

\begin{enumerate}
\item \label{T:F_slope1i} Let $q\geq \frac{1}{C-1}$ be a positive integer such that $qH_C$ is Cartier. Then  
$$q^{n+1}H_C^{n+1}\geq \deg f_*\mathcal{O}_X(qH_C).$$

 \item  \label{T:F_slope1ii} Let $q\geq \frac{1}{C-1}$ be a positive integer such that $qH_C$ is  Cartier and  $-q(K_F+\Delta_F)$ gives a generically finite map. Then
 $$q^{n+1}H_C^{n+1}\geq 2\frac{h^0(F,q(-K_{F}-\Delta_F))-n}{h^0(F,q(-K_{F}-\Delta_F))}\deg f_*\O_X(qH_C).$$
 
  \item \label{T:F_slope1iii} Let $q\geq \frac{1}{C-1}$ be a positive integer such that $qH_C$ is Cartier and $-q(K_{F}+\Delta_F)$ is globally generated. Then
	$$q^{n+1}H_C^{n+1}\geq 2\frac{q^nv }{q^nv +n} \deg f_*\O_X(qH_C)  \,.$$
\end{enumerate}
\end{theorem}

Before giving the proof, let us remark that one could rephrase the above results as inequalities between $(-K_{X/T}-\Delta)^{n+1}$ and  $\deg f_*\O_X(-q(K_{X/T}+\Delta))$, at least if $q$ is sufficiently divisible,  using the following formulas
$$
\begin{sis}
& q^{n+1}H_C^{n+1}=-q^{n+1}(-K_{X/T}-\Delta)^{n+1}\frac{\delta(C-1)+1}{\delta-1}, \\
& \deg f_*\mathcal{O}_X(qH_C)=\deg f_*\O_X(-q(K_{X/T}+\Delta))-C\frac{qh^0(-q(K_{F} + \Delta_F))}{v(n+1)}\frac{\delta}{\delta-1}(-K_{X/T}-\Delta)^{n+1},
\end{sis}
$$
where in the second formula we assumed that $q$ is sufficiently divisible so that $qC\frac{\delta}{(\delta-1)v(n+1)} \lambda_{CM}(X/T)$ is integral (and hence Cartier since $T$ is smooth).
\begin{remark}[Stability threshold in families] The stability threshold of the fiber is a lower-semicontinuous function on the base. A priori, it can take countably many values. This means that, at least if the base field is uncountable, the maximum value of $\delta$ (which gives also the best slope inequality) is obtained taking a very general fiber. If the base field is countable, to obtain the best slope inequality, we can make a base field extension and then choose a very general fiber over the greater field. The slope inequality obtained in this way holds true also for the family over the original field. The recent result \cite[Corollary 3.7]{LXZ}  shows that the minimum between the stability threshold and $(n+1)/n$ is constructible for families over a normal base, hence it attains a minimum.
\end{remark}

\begin{proof}
The $\Q$-divisor $H_C$ is nef by \cite[Thm 1.20]{CP} (or  \cite[Prop. 4.9]{Xu_Zhuang}, or \cite[Thm 1.3]{CP2}) and the fact that $\lambda_{CM}(X/T)$ is nef. 
Moreover, from \cite[Prop. 6.4]{CP} we infer that $f_*\mathcal{O}_X(qH_C)$ is nef since $qH_C$ is Cartier and the divisor 
$$
\begin{aligned}
qH_C-K_{X/T}-\Delta&=(q+1)(-K_{X/T}-\Delta)+qC\frac{\delta}{(\delta-1)v(n+1)}f^* \lambda_{CM}(X/T)=\\
&=(q+1)\left[-K_{X/T}-\Delta+\frac{qC}{q+1} \frac{\delta}{(\delta-1)v(n+1)}f^* \lambda_{CM}(X/T)\right]\\
\end{aligned}
$$
is $f$-ample (because $-K_{X/T}-\Delta$ is $f$-ample) and nef by \cite[Thm 1.20]{CP} (or  \cite[Prop. 4.9]{Xu_Zhuang}, or \cite[Thm 1.3]{CP2}), using the assumption that $q\geq \frac{1}{C-1}$ and the fact that $\lambda_{CM}(X/T)$ is nef.

Hence, by applying Theorems \ref{T:Barja}\eqref{T:Barja1}, Corollary \ref{C:Xiao-high1} and Corollary \ref{C:Xiao-high2} to $f:X\to T$ with $L=qH$ we obtain, respectively, the  inequalities in \eqref{T:F_slope1i}, \eqref{T:F_slope1ii} and \eqref{T:F_slope1iii}. 
\end{proof}

\begin{remark}
We do not know if items \eqref{T:F_slope1ii} and \eqref{T:F_slope1iii}  of Theorem \ref{T:F_slope1} holds without the assumption that $qH_C$ is Cartier, and if we can replace in item \eqref{T:F_slope1i} the assumption that $qH_C$ is Cartier with the weaker assumption that $q(-K_F-\Delta_F)$ is Cartier or with the alternative assumption that $-q(K_F+\Delta_F)$ gives a generically finite map. We do use the assumption that  $qH_C$ is Cartier when we apply  \cite[Prop. 6.4]{CP}.
\end{remark}

Theorem \ref{T:F_slope1} can be extended to the case when there is a geometric fiber $(F,\Delta_F)$ which is only K-polystable, provided that we slightly modify the family, as we know briefly indicate.  

The stability threshold of a K-polystable log-Fano pairs is always equal  to one. Given a maximal torus $\mathbb{T}$ in $\Aut(F,\Delta_F)$, we can however introduce a twisted stability threshold  $\delta_{\mathbb{T}}(F,\Delta_F)$, so that $(F,\Delta_F)$ is K-polystable if and only if $\delta_{\mathbb{T}}(F,\Delta_F)>1$, see \cite[Appendix A]{Xu_Zhuang}  (in loc. cit. the authors speak about reduced uniform K-stability, which by now is known to be equivalent to K-polystability, see \cite{LXZ}). K-polystable log-Fano pairs have a reductive automorphism group, so all maximal tori are conjugate and, in for K-polystable log-Fano pair, the twisted stability threshold does not depend on the choice of $\mathbb{T}$.

\begin{theorem}\label{T:F_slope2}
Keeping the notation as in Theorem \ref{T:F_slope1}, assuming that a geometric fiber $(F,\Delta_F)$ is K-polystable rather than K-stable. Let $\mathbb{T}$ be a maximal torus in $\Aut(F,\Delta)$, and assume that $\mathbb{T}$ acts fiberwise on $(X,\Delta)$. Replacing the stability threshold $\delta$ with the twisted stability threshold $\delta_{\mathbb{T}}:=\delta_{\mathbb{T}}(F,\Delta_F)$, the same inequalities of Theorem \ref{T:F_slope1} hold up to a base change, and up to  a birational modification of $X$ which does not modify $(F,\Delta_F)$. 

\end{theorem}
\begin{proof}

Following  \cite[Section 5]{Xu_Zhuang}, for every one parameter subgroup $\xi$ of $\mathbb{T}$, we can define a new family $f_{\xi}:(X_{\xi},\Delta_{\xi})\to T$, which is Zariski locally isomorphic to  $f:(X,\Delta)\to T$, and it is called the twist of $(X,\Delta)$ by $\xi$.  In particular, $qH_C$ remains Cartier on the new family.

From \cite[Remark A.2, Proposition 4.5 and its proof]{Xu_Zhuang}, after a base change, we can see that there exists a $\xi$ such that the Harder-Narasimhan filtration of the family twisted by $\xi$ has non-negative $\beta_{\delta_{\mathbb{T}}}$ invariant (we refer to \cite{Xu_Zhuang} for the definition of this invariant). We can now start arguing as in the proof of Theorem \ref{T:F_slope1}. The nefness of $H_C$, with $\delta$ replaced by $\delta_{\mathbb{T}}$ and the original family replaced by the twisted one, is now guaranteed by \cite[Proposition 4.9]{Xu_Zhuang} rather than \cite[Theorem 1.20]{CP}. From now on, the argument is verbatim as in the proof of Theorem \ref{T:F_slope1}.
\end{proof}

If $\mathbb{T}$ acts only on $(F,\Delta_F)$ but not on $(X,\Delta)$, arguing similarly to \cite[Thm. 6]{LiXu}, one can make a base change and a birational modification to construct a new family as in the set-up \ref{N:Kstab} where $\mathbb{T}$ acts on all fibers. After these operations, $qH_C$ remains $\Q$-Cartier but might stop being Cartier, so one might have to increase $q$ before start arguing as in the proof of Theorem \ref{T:F_slope1}.

The following example shows that the conclusion of Theorem \ref{T:F_slope1} may fail if $(F,\Delta_F)$ is only K-polystable but not K-stable, and that indeed  the birational modification in Theorem \ref{T:F_slope2} is sometimes unavoidable.

\begin{example}\label{ex:twist}
Consider the Hirzebruch surface $\mathbb{F}_e$ (for $e\geq 0$), with its natural projection $f_e:\bF_e\to \P^1$ (and take $\Delta=0$). The fibers of this family are isomorphic to $\P^1$, which is K-polystable but not K-stable.

For every $e\geq 0$, we have 
$$\deg \lambda_{CM}(\bF_e/\P^1)=-(-K_{\bF_e/\P^1})^{2}=0.$$ 
In particular, for any rational $C>1$ the $\Q$-divisor $H_C$ of \eqref{E:HC} is equal to $-K_{\bF_e/\P^1}$. 

If $e=0$, i.e. $\bF_0$ is the trivial family $\P^1\times \P^1$, then  $\deg f_*\O_X(-qK_{\bF_e/\P^1})=0$  for every $q\geq 1$, and hence Theorem \ref{T:F_slope2} holds true without the need of a twist.
 
On the other hand, if $e\geq 1$,  the degree of the locally free sheaf $f_*\O_X(-qK_{\bF_e/\P^1})$ is strictly greater than zero for every $q\geq 1$ (more precisely, fixed $q$, it tends to infinity when $e$ grows). This shows that the twist in Theorem \ref{T:F_slope2} is in this case necessary, and indeed $f_0:\mathbb{F}_0\to \P^1$ is a twist of $f_e:\mathbb{F}_e\to \P^1$ for every $e\geq 1$, see \cite[Example 5.2]{Xu_Zhuang}.
\end{example}

\subsection{Application to the moduli space of K-semistable Fano varieties}\label{sec:Fanomoduli}

Let $\mathcal{M}_{n,v}^K$ be the algebraic stack of K-semistable Fano varieties with dimension $n$ and volume $v$. Note that $\M_{n,v}^K$ is an Artin stack of finite type. Let $\pi \colon \mathcal{X}_{n,v}\to \mathcal{M}_{n,v}^K$ be the universal family. Define the Chow-Mumford (CM) $\Q$-divisor  
(well-defined up to $\Q$-linear equivalence) 
$$\lambda_{CM}=-\pi_*(-K_{\mathcal{X}_{n,v}/\M_{n,v}^K })^{n+1}.$$
This $\Q$-divisor is $\Q$-Cartier and nef by \cite[Theorem 1.1(a)]{CP}.

Let $\mathcal{M}_{n,v}^{K,s}$ the open Deligne-Mumford substack of $\mathcal{M}^K_{n,v}$ parametrizing K-stable Fano varieties.
The stability threshold is a lower semi-continuous function, strictly greater than $1$ on $\mathcal{M}_{n,v}^{K,s}$, which in principle can assume countably many values. We can however consider as invariant the minimum between the stability threshold and $(n+1)/n$. This second invariant is lower semi-continuous  and constructible for the Zariski topology by \cite[Corollary 3.7]{LXZ}, and we call $\delta_{n,v}>1$ its minimum on $\mathcal{M}_{n,v}^{K,s}$ (we apply \cite[Corollary 3.7]{LXZ} to the normalisation of an atlas of $\mathcal{M}_{n,v}^{K,s}$).  

\begin{theorem}\label{thm:Fano_moduli}
With the above notations, let $T$ be a normal projective curve in $\mathcal{M}_{n,v}$  intersecting $\mathcal{M}_{n,v}^{K,s}$. 
For any rational number $C>1$ consider the $\Q$-Cartier $\Q$-divisor on $\X_{n,v}$
$$
\H_C:=-K_{\X_{n,v}/\M_{n,v}^K}+C\frac{\delta_{n,v}}{(\delta_{n,v}-1)v(n+1)}\pi^* \lambda_{CM}.
$$

\begin{enumerate}
\item \label{thm:Fano_moduli1} Let $q\geq \frac{1}{C-1}$ be a positive integer such that $q\H_C$ is Cartier. Then  
$$q^{n+1}\pi_*(\H_C^{n+1})- c_1 \pi_*\mathcal{O}_{\X_{n,v}}(q\H_C)$$
  intersects non-negatively $T$.

  \item \label{thm:Fano_moduli2} Let $q\geq \frac{1}{C-1}$ be a positive integer such that $q\H_C$ is Cartier and $-qK_{F}$ is globally generated for any $F\in  \mathcal{M}_{n,v}^{K,s}$. Then
	$$q^{n+1}\pi_*(\H_C^{n+1})- 2\frac{q^nv }{q^nv +n} c_1 \pi_*\O_{\X_{n,v}}(q\H_C)$$
	  intersects non-negatively $T$. 
\end{enumerate}
%VECCHIA VERSIONE 
% and let
%$$
%\H:=-K_{\mathcal{X}/\M_{n,v}^K }+2\frac{\overline{\delta}}{(\overline{\delta}-1)v(n+1)}\pi^* \lambda_{CM}
%$$
%\begin{enumerate}

%\item Let $q$ be a number such that $qH$ is Cartier, and such that the formation of $\det\left(\pi_*\mathcal{O}_{\mathcal{X}}(-qK_{\mathcal{X}/\mathcal{M}_{n,v}})\right)$ commutes with base-changes, then the line bundle
%$$q^{n+1}\pi_*H^{n+1}- \det \left( \pi_*\mathcal{O}_{\mathcal{X}}(qH)\right)$$
% intersect non-negatively $T$.
 
%\item Let $q$ be a number such that $qH$ is Cartier and $\pi$-globally generated, and such that the formation of $\det \left(\pi_*\mathcal{O}_{\mathcal{X}}(-qK_{\mathcal{X}/\mathcal{M}_{n,v}})\right)$ commutes with base-changes, then
%	$$q^{n+1}\pi_*H^{n+1}-2\frac{q^nv }{q^nv +n} \det\left( \pi_*\O_{\mathcal{X}}(qH)  \right)$$
%	 intersect non-negatively $T$.
%\end{enumerate}
\end{theorem}

\begin{proof}
Let $f\colon X\to T$ be the pull-back of the universal family. Note that: $X$ is normal since all the fibers of $f$ are normal (being K-semistable Fano varieties), $-K_{X/T}$ is $\Q$-Cartier and $f$-ample by the Koll\'ar condition on families of K-semistable Fano varieties and a general fiber $F$ is K-stable by the assumption that $T$ 
intersects $\mathcal{M}_{n,v}^{K,s}$.
Hence, the results follows from Theorem \ref{T:F_slope1} (with $\Delta=0$) using that 
\begin{equation}\label{E:res-div-Fano}
\begin{sis}
& \pi_*(\H_C^{n+1})\cdot T=H_C^{n+1} , \\
& c_1 \pi_*\O_{\X_{n,v}}(q\H_C)\cdot T=\deg f_*\O_X(qH_C),
\end{sis}
\end{equation}
In the second equality we have used the base change property $\pi_*\O_{\X_{n,v}}(q\H_C)_{|T}=f_*\O_X(qH_C)$, which holds because all fibers of $\pi$ are Fano varieties with klt singularities and then by Kawamata vanishing we have $R^i \pi_*\O_{\X_{n,v}}(q\H_C)=0$ for all $i>0$.

\end{proof}

The coarse moduli space of the DM stack $\M_{n,v}^{K,s}$ is a quasi-projective variety that it is not  proper in general. However, 
 if  $V\subset \M_{n,v}^{K,s}$ is a proper (and hence projective) closed subscheme, then the Chow-Mumford $\Q$-divisor $\lambda_{CM}$  is ample on $V$ (see \cite[Theorem 1.1]{CP} and \cite[Theorem 1.1]{Xu_Zhuang}), and Theorem \ref{thm:Fano_moduli} provides new nef line bundles on $V$.

It is worth mentioning that $\mathcal{M}^K_{n,v}$ has a projective good moduli space $M_{n,v}^K$, but neither
$\det\left( \pi_*\O_{\mathcal{X}}(q\H_C)  \right)$ nor $\det\left( \pi_*\O_{\mathcal{X}}(-qK_{\mathcal{X}/\mathcal{M}_{n,v}^K})  \right)$ descend, in general, to $\Q$-line bundles on $M_{n,v}^K$;
hence we do not see how to use slope inequalities to construct nef $\Q$-line bundles on $M_{n,v}^K$. On the other hand, the Chow-Mumford  $\Q$-divisor $\lambda_{CM}$ descends to an ample $\Q$-Cartier $\Q$-divisor on $M_{n,v}^K$, see \cite{Xu_Zhuang,LXZ}.

\section{Slope inequalities under other stability conditions}\label{Sec:BS}

The aim of this section is to collect some slope inequalities that are true under GIT or slope (semi)stability conditions. 
These slope inequalities are formulated in term of the following notion of positivity for divisors on the total space of a fibration  over a curve (as in the setup \ref{setup}), which was studied in detail by Barja and Stoppino in a series of papers \cite{BS1,BS2,BS3}.

\begin{definition}\label{D:f-positive}(\cite[Def. 1.3]{BS1})
Assume that we are in the set-up  \ref{setup}. We say that the divisor $L$ is \emph{$f$-positive} if 
$$
L^{n+1}\geq (n+1)\frac{L_F^n}{h^0(F,\lfloor L_F\rfloor)} \deg f_*\O_X(L). 
$$
\end{definition}

The above notion can be slightly rephrased if $f_*\O_X(L)$ has positive degree. More precisely, if  $\deg f_*\O_X(L)\neq 0$, then we define the \emph{slope} of $L$ to be
$$
s(L)=\frac{L^{n+1}}{\deg f_*\O_X(L)}.
$$ 
Then, under the assumption that $\deg f_*\O_X(L)> 0$, we have that 
\begin{equation}\label{E:posBS}
L \text{ is f-positive } \Longleftrightarrow s(L)\geq \BS(L_F):=(n+1)\frac{L_F^n}{h^0(F,\lfloor L_F\rfloor)},
\end{equation}
where $\BS$ stands for the \emph{Barja-Stoppino} invariant of the $\Q$-Cartier $\Q$-divisor $L_F$ on $F$.

Under some positivity assumption on $L_F$, the Barja-Stoppino invariant of $L_F$ admits the following lower bounds, which should be compared with the slope inequalities in Corollaries \ref{C:Xiao-high1} and \ref{C:Xiao-high2}. 
 
\begin{remark}\label{R:ineqBS}
Assume that we are in the setup \ref{setup} and that $L_F$ is nef and with generically finite associated map $\phi_{L_F}$.
\begin{enumerate}
\item \label{R:ineqBS1} We have that
$$
\BS(L_F)\geq (n+1)\frac{L_F^n}{L_F^n+n}\geq (n+1) \frac{h^0(F,\lfloor L_F\rfloor)-n}{h^0(F,\lfloor L_F\rfloor)}\geq (n+1)\frac{h^0(F,\lfloor L_F\rfloor)-n}{L_F^n+n},
$$
with equalities if and only if $L_F^n=h^0(F,\lfloor L_F\rfloor)-n$.

This follows from the fact that  $L_F^n\geq h^0(F,\lfloor L_F\rfloor)-n$, see Corollary \ref{C:Noether1}\eqref{C:Noether1i}.

\item \label{R:ineqBS2}
If, moreover, either $\dim F\geq 2$ and $\kappa(F)\geq 0$ or $\dim F=1$ and $L_F$ is special, then we have that 
$$
\BS(L_F)\geq 2(n+1)\frac{L_F^n}{L_F^n+2n}\geq 2(n+1) \frac{h^0(F,\lfloor L_F\rfloor)-n}{h^0(F,\lfloor L_F\rfloor)}\geq 4(n+1)\frac{h^0(F,\lfloor L_F\rfloor)-n}{L_F^n+2n},
$$
with equalities if and only if $L_F^n=2h^0(F,\lfloor L_F\rfloor)-2n$.

This is a consequence of the inequality  $L_F^n\geq 2h^0(F,\lfloor L_F\rfloor)-2n$, which follows from Corollary \ref{C:Noether1}\eqref{C:Noether1ii} if $\dim F\geq 2$ and from Clifford's theorem if $\dim F=1$.
\end{enumerate}
\end{remark}

If we modify $L$ by the pull-back of more and more positive divisors from the base, then the slope of $L$ will become closer and closer to $\BS(L_F)$, as we now show.

\begin{proposition}\label{P:convBS}
Assume that we are in the setup \ref{setup} and that $\deg f_*\O_X(L)> 0$. Let $A$ be an ample divisor on $T$.
Then the slopes $s(L+qf^*A)$ converge monotonically to $\BS(L_F)$ as $\N\ni q\to +\infty$.
\end{proposition}
\begin{proof}
Using that $(f^*A)^2=f^*(A^2)=0$ since $A$ is a divisor on a curve and the projection formula, we compute 
\begin{equation*}
\begin{sis}
& (L+qf^*A)^{n+1}=L^{n+1}+q(n+1)L_F^n\deg A, \\
& f_*\O_X(L+qf^*A)=f_*\O_X(L)\otimes \O_T(qA)\Rightarrow \deg f_*\O_X(L+qf^*A)= \deg f_*\O_X(L)+q h^0(F,\lfloor L_F\rfloor) \deg A. 
\end{sis}
\end{equation*}
From the above formulas we get that 
\begin{equation*}
\begin{sis}
& s(L+qf^*A)-\BS(L_F)=\frac{\deg f_*\O_X(L)}{\deg f_*\O_X(L+qf^*A)} \left[s(L)-\BS(L_F)\right],\\
& s(L+qf^*A)-s(L)=\frac{q\deg A\cdot h^0(F,\lfloor L_F\rfloor)}{\deg f_*\O_X(L+qf^*A)}\left[\BS(L_F)- s(L)\right].
\end{sis}
\end{equation*}
Hence, we conclude that 
\begin{equation}\label{E:sandwitch}
\begin{sis}
s(L)\leq \BS(L_F) &\Longrightarrow s(L)\leq s(L+qf^*A)\leq \BS(L_F), \\
\BS(L_F)\leq s(L) & \Longrightarrow \BS(L_F)\leq s(L+qf^*A)\leq s(L).\\
\end{sis}
\end{equation}
Formulas \eqref{E:sandwitch} imply that the sequence $s(L+qf^*A)$ converges monotonically to $\BS(L_F)$ as $q\to +\infty$.
\end{proof}

\begin{remark}
The above Proposition shows that the lower bound $s(L)\geq \BS(L_F)$  (i.e. the $f$-positivity of $L$) is the best possible lower bound we can hope for $s(L)$ in terms of numerical invariants of the general fiber $F$. 

However that there are examples of  fibrations $f$ endowed with (sufficiently positive) divisors  $L$ that are not  $f$-positive: for example, Hu-Zhang have constructed in \cite[\S 2.1, \S 2.4, \S 3]{HZ} families of smooth canonically polarized varieties of any dimension $n\geq 2$ such that the relative canonical bundle is not $f$-positive.

In subsection \ref{Sec:fam-mindeg}, we investigate the $f$-positivity for families of polarized varieties of minimal degrees and polarized hyperelliptic varieties. 
\end{remark}

The $f$-positivity of $L$ holds true provided that either $(F,L_F)$ is  Chow semistable (e.g. if it is  Hilbert semistable) or $f_*\O_X(L)$ is semistable.

\begin{theorem}(\cite[Thm. 3.3]{Bost})\label{T:BS-pos}
Assume that we are in the set-up  \ref{setup} and that, moreover,  $L$ is $f$-nef. 

If $L_F$ is a globally generated Cartier divisor and the cycle $(\phi_{L_F})_*(F)\subset \P(H^0(F,L_F)^\vee)$ is Chow semistable then $L$ is $f$-positive. 
\end{theorem}
The  special case of the above result when $L$ is $f$-ample, $L_F$ is a very ample Cartier divisor and the subvariety $\phi_{L_F}(F)\subset \P(H^0(F,L_F)^\vee)$ is Hilbert semistable was proved earlier by Cornalba-Harris \cite[Thm. 1.1]{CH} (see also \cite[Cor. 2.3]{BS1})

\begin{theorem}\label{T:ss-push}
Assume that we are in the set-up  \ref{setup} and that, moreover, $L$ is $f$-nef and $L_F$ is Cartier and globally generated. 

If $f_*\O_X(L)$ is semistable of non-negative degree then $L$ is $f$-positive. 

\end{theorem}

The above result was proved for a Cartier divisor $L$ in \cite[Thm.1.3]{BS2} under the further assumption that  either $L$ is $f$-globally generated or $L$ is nef.
\begin{proof}
Denote by 
$$\mu:=\mu(f_*\O_X(L))=\frac{\deg f_*\O_X(L)}{\rk f_*\O_X(L)}=\frac{\deg f_*\O_X(L)}{h^0(F,\lfloor L_F\rfloor)}$$
 the slope of the locally free sheaf $f_*\O_X(L)$. Using that $F^2=0$, we compute that 
$$(L-\mu F)^{n+1}=L^{n+1}-(n+1)\frac{L_F^n}{h^0(F,\lfloor L_F\rfloor)} \deg f_*\O_X(L).$$
In particular
\begin{equation}\label{E:fpos-top}
L \quad \text{ is $f$-positive } \Longleftrightarrow (L-\mu F)^{n+1}\geq 0. 
\end{equation}

Using the notation of \S \ref{Sec:HN} (with $\ell=1$ since $f_*\O_X(L)$ is semistable) and the fact that $f_*\O_X(L)$ is nef since it is semistable of non-negative degree 
(see \eqref{E:Nef-HN}), our assumptions imply that
\begin{equation}\label{E:disug-ss}
\begin{sis}
M_1-\mu \wt F \text{ is nef } & \: \text{ by Proposition \ref{P:Nnef}} \Longrightarrow 0\leq (M_1-\mu \wt F)^{n+1}=M_1^{n+1}-(n+1)\mu P_1^n,\\
L^{n+1}\geq M_1^{n+1} & \: \text{ by  Lemma \ref{L:Lnef}\eqref{L:Lnef2},  }\\
L_F^n=P_1^n & \: \text{ by  Proposition \ref{P:Mnef}\eqref{P:Mnef5}}.\\
\end{sis}
\end{equation}
By putting together  the above formulas \eqref{E:disug-ss}, we get that 
$$
(L-\mu F)^{n+1}\geq (M_1-\mu \wt F)^{n+1}\geq 0, 
$$
and hence that $L$ is $f$-positive by \eqref{E:fpos-top}.
\end{proof}

\begin{remark}
When $L$ is $f$-globally generated, so that $M_{\ell}=L$, and $f_*\O_X(L)$ is semistable,  the proof of Theorem \ref{T:ss-push} shows that the nef threshold of $L$ with respect to a fiber $F$ is at least the slope $\mu(f_*\O_X(L))$ of $f_*\O_X(L)$. It is worth recalling that \cite[Remark 2, Section 4]{Xiao} gives an example where $f_*\O_X(L)$ is not semistable, $L-\mu(f_*\O_X(L)) F$ is not nef, but still the family is $f$-positive.
\end{remark}

\section{Examples}\label{Sec:Examples}

In this section we will compute the slope of some natural families of polarised varieties. At the end of each example, we will discuss why it is relevant for the purposes of this paper.

\subsection{Families of varieties of minimal degree and of polarized hyperelliptic varieties}\label{Sec:fam-mindeg}

In this subsection, we are going to compute slopes of several natural families  of 
\begin{itemize}
\item \emph{varieties of minimal degree}, i.e. polarized varieties $(F,L_F)$ such that $L_F$ is very ample and it embeds $F$ as a (non-degenerate) variety in $\P(H^0(F,L_F)^\vee)$ of degree $\deg F=\codim F+1$ (see \cite{EH}).
Note that, for any such family, we have that $L_F^n=h^0(F,L_F)-n$, where $n=\dim F$, so that all the inequalities in Remark \ref{R:ineqBS}\eqref{R:ineqBS1} are equalities. 

\item \emph{hyperelliptic polarized  varieties}, i.e.  polarized varieties $(\wt F,\wt L_F)$ such that $\wt L_F$ is base point free and it induces a double finite cover of a variety of minimal degree (see \cite{Fuj}).
Note that, for any such family, we have that $\wt L_F^n=2h^0(\wt F,\wt L_F)-2n$, where $n=\dim \wt F$, so that all the inequalities in Remark \ref{R:ineqBS}\eqref{R:ineqBS2} are equalities. 
\end{itemize}

In all the examples, we will use the following notation. Given a locally free sheaf $E$ of rank $r$ on a scheme $S$, we will denote by $\pi:\P_S(E)\to S$ the projective bundle of quotients of $E$ and by $H=H_E$ any tautological divisor on $\P_S(E)$, i.e. any 
effective Cartier divisor on $\P_S(E)$ such that $\O_{\P_S(E)}(H_E)=\O_{\P_S(E)}(1)$. With these convention, we have that 
\begin{equation}\label{E:form-Proj}
\begin{sis}
& \pi_*\O_{\P_S(E)}(dH_E)=
\begin{sis} 
\Sym^d(E) & \text{ for any } d\geq 0,  \\
0 & \text{ for any } d< 0.  \\
\end{sis}
\\
& H^r=\sum_{i=1}^{r} (-1)^{i-1}\pi^*(c_i(E))H^{r-i} \in A^r(\P_S(E)). 
\end{sis}
\end{equation}

Moreover, we are going to use frequently the following two nefness results. Let $E$ be a locally free sheaf on an irreducible smooth and projective curve $T$ and denote by $\mu_-(E)$ the lowest slope in the Harder-Narasimhan filtration of $E$. Then we have that (see \cite[Lemma 2.1]{Ful} and  \cite[Thm. 6.4.15]{Laz2}):
\begin{equation}\label{E:Nef-PE}
\begin{aligned}
E \text{ is nef on } T \Leftrightarrow \mu_-(E)\geq 0 \quad \text{[Hartshorne's theorem]}, \\ 
dH_E+f^*(A) \text{ is nef on } \P_T(E) \Leftrightarrow d\geq 0 \text{ and } d\mu_-(E)+\deg A \geq 0 \quad \text{[Miyaoka's theorem]}.
\end{aligned}
\end{equation}

\begin{example}\label{Ex:Pn}[Families of projective spaces and their double covers]

Let $T$ be a smooth irreducible projective curve  and let $E$ be a nef locally free sheaf on $T$ of rank $n+1$ and of positive degree. 

\un{Family of projective spaces}

Consider the projective bundle  $f:X:=\P_T(E)\to T$ and let $L:=H$ be any tautological divisor on $\P_T(E)$. 
Note that $L$ is nef since $E$ is nef and its restriction to a general (and indeed any) fiber $F\cong \P^n$ is a hyperplane divisor, and hence it is very ample.
Using \eqref{E:form-Proj}, we compute  
\begin{equation}\label{E:Pn-for1}
\begin{sis}
& f_*\O_{X}(L)=E \quad (\Rightarrow f_*\O_{X}(L) \text{ is nef}),\\
& L^{n+1}=\deg(E),
\end{sis}
\end{equation}
from which we deduce that
\begin{equation}\label{E:Pn-slope}
s(L)=1=\BS(L_F).
\end{equation}
Note that this example realises the equality in Theorem \ref{T:Barja}\eqref{T:Barja1} with $q=1$ and also it provides an example where the $f$-positivity is sharp (see \eqref{E:posBS}). 

\un{Double covers}

Fix now an integer $m\geq 2$ and a  divisor $A$ on $T$  such that 
\begin{equation*}
 |2(mH+f^*A)|\neq \emptyset \text{ and the general element } R\in |2(mH+f^*A)| \text{ is smooth.} 
\end{equation*}  

\vspace{0.1cm}

Take a general effective smooth divisor $R\in |2(mH+g^*A)|$ and denote by $\pi:\wt X\to X=\P_T(E)$ the  finite double cover ramified along $R$ and set $\wt f:\wt X\xrightarrow{\pi}X \xrightarrow{f} T$. 
Consider the nef divisor $\wt L=\pi^*(L)$ on $\wt X$.

Note that a general polarized fiber $(\wt F,\wt L_F)$ of $\wt f$ is a double finite cover of $(\P^n,H_{\P^n})$ ramified along a smooth hypersurface of degree $2m\geq 4$, and hence 
%$(\wt F,\wt L_F)$ is a hyperelliptic polarized variety of Type I in the sense of \cite{Fuj} with $\wt L_F^n=2$ and $h^0(\wt F,\wt L_F)=n+1$. 
%In particular, we have that 
\begin{equation}\label{E:BS-doubPn}
\BS(\wt L_F)=2\BS(L_F)=2. 
%(n+1)\frac{\wt L_F^n}{h^0(\wt F,\wt L_F)}=2.
\end{equation}
%Note that, by the Riemann-Hurwitz formula, we have that $K_F=\pi_{|F}^*(K_{\P^n}+mH_{\P^n})=(m-n-1)\pi_{|F}^*(H_{\P^n})$, so that $\kappa(F)\geq 0$ (resp. $F$ is canonically polarized) if and only if $m\geq n+1$ (resp. $m\geq n+2$).   
 
Using  \eqref{E:Pn-for1},  the projection formula and the formula  $\pi_*\O_{\wt X}=\O_{X}\oplus \O_{X}(-mH-f^*A)$, we compute: 
\begin{equation}\label{E:douPn-for1}
\begin{sis}
& \wt f_*\O_{\wt X}(\wt L)=f_*\pi_*\pi^* \O_X(L)=f_*\left(\O_{X}(L)\oplus \O_{X}(H-mH-f^*A) \right)=E \quad (\Rightarrow \wt f_*\O_{\wt X}(\wt L) \text{ is nef}),\\
& \wt L^{n+1}=2L^{n+1}=2\deg E,
\end{sis}
\end{equation}
from which we deduce that
\begin{equation}\label{E:douPn-slope}
s(\wt L)=2=\BS(\wt L_F).
\end{equation}

%VECCHIE FORMULE

%using the projection formula and the second formula in \eqref{E:form-Proj}:
%\begin{equation}\label{E:top-doubPn}
%L^{n+1}=2H^{n+1}=2\deg E.
%\end{equation}

%We now compute the push-forward  $f_*\O_X(L)$ using the the projection formula, the formula for $\pi_*\O_X$ (coming from the structure of the finite double cover $\pi$), and first formula of \eqref{E:form-Proj}:
%\begin{equation*}
%f_*\O_X(L)=g_*\pi_*\O_X(L)=g_*\left(\O_{\P_T(E)}(H)\otimes [\O_{\P_T(E)}\oplus \O_{\P_T(E)}(-mH-g^*A)] \right)=E.
%\end{equation*}
%Hence, $f_*\O_X(L)$ is nef and 
%\begin{equation}\label{E:push-doubPn}
%\deg f_*\O_X(L)=E.
%\end{equation}

Note that example  realises the equality in Theorem \ref{T:Barja}\eqref{T:Barja2} with $q=1$ and also it provides an example where the $f$-positivity is sharp (see \eqref{E:posBS}).

\end{example}

\begin{example}\label{Ex:Vero}[Families of Veronese surfaces and their double covers]

Let $T$ be a smooth irreducible projective curve  and let $E$ be a nef locally free sheaf on $T$ of rank $3$ and of positive degree. 

\un{Families of Veronese surfaces}

Consider the projective bundle $f:X:=\P_T(E)\to T$ and set $L:=2H$ where  $H$ is any tautological divisor on $\P_T(E)$.
Note that $L$ is nef since $E$ is nef and a general (and indeed any) fiber $(F,\O_F(L_F))$ is isomorphic to the Veronese surface $(\P^2,\O_{\P^2}(2))$.

Using \eqref{E:form-Proj}, we compute  
\begin{equation}\label{E:Vero-for1}
\begin{sis}
& f_*(\O_{X}(L))=\Sym^ 2(E) \Rightarrow \deg f_*(\O_{X}(L))=\deg \Sym^ 2(E)=4\deg E, \\
& L^3=(2H)^{3}=8\deg(E),
\end{sis}
\end{equation}
from which we deduce that
\begin{equation}\label{E:Vero-slope}
s(L)=2=\BS(L_F).
\end{equation}

\un{Double covers}

Fix now an integer $m\geq 3$ and a  divisor $A$ on $T$  such that 
\begin{equation*}
 |2(mH+f^*A)|\neq \emptyset \text{ and the general element } R\in |2(mH+f^*A)| \text{ is smooth}.  
\end{equation*}

\vspace{0.1cm}

Take a general effective smooth divisor $R\in |2(mH+g^*A)|$ and denote by $\pi:\wt X\to X=\P_T(E)$ the  finite double cover ramified along $R$ and set $\wt f:\wt X\xrightarrow{\pi}X \xrightarrow{f} T$. 
Consider the nef divisor $\wt L=\pi^*(L)$ on $\wt X$.

Note that a general polarized fiber $(\wt F,\wt L_F)$ of $f$  is a double finite cover of $(\P^2, 2H_{\P^2})$ ramified along a smooth hypersurface of degree $2m\geq 6$, and hence 
%$(\wt F,\wt L_F)$ is a hyperelliptic polarized variety of Type IV in the sense of \cite{Fuj} with $\wt L_F^2=8$ and $h^0(\wt F,\wt L_F)=6$. 
%In particular, we have that 
\begin{equation}\label{E:BS-2Vero}
\BS(\wt L_F)=2\BS(L_F)=4.
\end{equation}
%Note that, by the Riemann-Hurwitz formula, we have that $K_F=\pi_{|F}^*(K_{\P^2}+mH_{\P^2})=(m-3)\pi_{|F}^*(H_{\P^2})$, so that $\kappa(F)\geq 0$ and, moreover,  $F$ is canonically polarized if and only if $m\geq 4$.  
 
 Arguing as in Example \ref{Ex:Pn} and using \eqref{E:Vero-for1}, we get that 
  \begin{equation}\label{E:2Vero-for}
\begin{sis}
& \wt f_*\O_{\wt X}(\wt L)=f_*\O_{X}(L)=\Sym^2 E, \\
%\quad (\Rightarrow \wt f_*\O_{\wt X}(\wt L) \text{ is nef}),\\
& \wt L^{n+1}=2L^{n+1}=16 \deg(E),
\end{sis}
\Rightarrow s(\wt L)=4=\BS(\wt L_F).
\end{equation}
 
Note that these families provide examples where  the $f$-positivity is sharp (see \eqref{E:posBS}). 
 
 %SPIEGAZIONE PIU" LUNGA
% Using  \eqref{E:Vero-for1},  the projection formula and the formula  $\pi_*\O_{\wt X}=\O_{X}\oplus \O_{X}(-mH-f^*A)$, we compute: 
 %\begin{equation}\label{E:2Vero-for1}
%\begin{sis}
%& \wt f_*\O_{\wt X}(\wt L)=f_*\left(\O_{X}(L)\oplus \O_{X}(2H-mH-f^*A)] \right)=\Sym^2 E \Rightarrow \deg f_*(\O_{X}(L))=4\deg E, \\
%& \wt L^{n+1}=2L^{n+1}=16\deg E.
%\end{sis}
%\end{equation}
%Hence, we get that 
%\begin{equation}\label{E:2Vero-slope}
%s(\wt L)=4=\BS(\wt L_F).
%\end{equation}

 %VECCHIE FORMULE
%The top self-intersection of $L$ can be computed using the projection formula and the second formula in \eqref{E:form-Proj}:
%\begin{equation}\label{E:top-2Vero}
%L^{3}=2(2H)^{3}=16\deg E.
%\end{equation}

%We now compute the push-forward  $f_*\O_X(L)$ using the the projection formula, the formula for $\pi_*\O_X$ (coming from the structure of the finite double cover $\pi$), and first formula of \eqref{E:form-Proj}:
%\begin{equation*}
%f_*\O_X(L)=g_*\pi_*\O_X(L)=g_*\left(\O_{\P_T(E)}(2H)\otimes [\O_{\P_T(E)}\oplus \O_{\P_T(E)}(-mH-g^*A)] \right)=\Sym^2 E.
%\end{equation*}
%Hence, $f_*\O_X(L)$ is nef and 
%\begin{equation}\label{E:push-2Vero}
%\deg f_*\O_X(L)=\deg \Sym^2 E=4 \deg E.
%\end{equation}
%Hence, combining \eqref{E:top-2Vero}, \eqref{E:push-2Vero} and \eqref{E:BS-2Vero} we get that 

\end{example}

\begin{example}\label{Ex:Quadrics}[Families of Quadrics and their double covers]

Let $T$ be a smooth irreducible projective curve, let $E$ be a nef locally free sheaf on $T$ of rank $n+2$ and of positive degree, and let 
 $H$ be a tautological divisor on the projective bundle $h:\P_T(E)\to T$. 
 
 \un{Families of Quadrics}
 
 Consider a  divisor $A$ on $T$ such that 
 \begin{equation*}
|2H+h^*A|\neq \emptyset  \text{ and the general element in } |2H+h^*A| \text{ is normal.}\tag{*}
\end{equation*}

Take a general divisor $X\in |2H+h^*A|$ and let $f:X\to T$ be the restriction of $h$, which is a fibration of quadric hypersurfaces in $\P_T(E)$.  
Let $L:=H_{|X}$ which is a nef Cartier (since $E$ is nef) divisor on $X$.
% which restricts to the hyperplane divisor on each fiber of $f:X\to T$. 

Note that a general fiber $F$ of $f:X\to T$ is a quadric inside $\P^{n+1}$ and $L_F$ is a hyperplane divisor on $F$, and hence
\begin{equation}
\BS(L_F)=(n+1)\frac{L_F^n}{h^0(F,L_F)}=(n+1)\frac{2}{n+2}=2-\frac{2}{n+2}.
\end{equation}

The top-self intersection of $L$ on $X$ can be compute inside $\P_T(E)$ as it follows (using also \eqref{E:form-Proj})
\begin{equation}\label{E:Ltop-quad}
L^{n+1}=H^{n+1}\cdot X=H^{n+1}\cdot(2H+h^*A)=2H^{n+2}+H^{n+1}\cdot h^*A=2\deg E+\deg A.
\end{equation}
In order to compute the degree of $f_*\O_{X}(L)$, consider the exact sequence of the divisor $X\subset \P_T(E)$ tensored by $\O_{\P_T(E)}(1)$:
$$ 
0 \to \O_{\P_T(E)}(-H-h^*A)=\O_{\P_T(E)}(1)(-X)\to \O_{\P_T(E)}(1)\to \O_{\P_T(E)}(1)_{|X}=\O_{X}(L)\to 0.
$$
By taking the pushforward along $h$ and using that $h_*\O_{\P_T(E)}(-H-h^*A)=R^1h_*\O_{\P_T(E)}(-H-h^*A)=0$, we get the isomorphism 
\begin{equation}\label{E:pushO-quad}
E=h_*\O_{\P_T(E)}(1)\xrightarrow{\cong} f_*\O_{X}(L).
\end{equation}
In particular, $f_*\O_{X}(L)$ is nef by our assumption on $E$.    From \eqref{E:Ltop-quad} and \eqref{E:pushO-quad}, we get that 
\begin{equation}\label{E:slopeL-quad}
s(L)=2+\frac{\deg A}{\deg E}. 
\end{equation}
Note that 
$$
s(L)\geq \BS(L_F) \Leftrightarrow \deg A\geq -2\mu(E)=-2\frac{\deg E}{\rk E}.
$$

 \un{Double covers}

Fix now an integer $m\geq 2$ and a  divisor $B$ on $T$  such that 
\begin{equation*}
 |2(mL+f^*B)|\neq \emptyset \text{ and the general element in }  |2(mL+f^*B)| \text{ is smooth.} \tag{**}
 \end{equation*}
Take a general effective smooth divisor $R\in |2(mH+f^*B)|$ and denote by $\pi:\wt X\to X$ the  finite double cover ramified along $R$ and set $\wt f:\wt X\xrightarrow{\pi}X\xrightarrow{f} T$. 
Consider the nef divisor $\wt L:=\pi^*(L)$ on $\wt X$. 

Note that a general polarized fiber $(\wt F,\wt L_F)$ of $\wt f$  is a finite double cover of $(F,L_F)$   ramified along a divisor of $|2m L_F|$ (with $m\geq 2$), and hence 
%Hence $(\wt F,\wt L_F)$   a hyperelliptic polarized variety of Type II in the sense of \cite{Fuj} with $\wt L_F^n=4$ and $h^0(\wt F,\wt L_F)=n+2$, so that  
\begin{equation}\label{E:BS-2Quadr}
\BS(\wt L_F)=2\BS(L_F)=4-\frac{4}{n+2}. 
\end{equation}
%Note that, by the Riemann-Hurwitz formula for $\pi$ and the adjunction formula for the quadric hypersurface $Q\subset \P^{n+1}$, we have that $K_F=\pi_{|F}^*(K_{Q}+mH_{Q})=\pi_{|F}^*(-nH_{Q}+mH_Q)=(m-n)\pi_{|F}^*(H_Q)$.
%Hence, we have that $\kappa(F)\geq 0$ (resp. $F$ is canonically polarized)  if and only if $m\geq n$ (resp. $m\geq n+1$).  
 
  Arguing as in Example \ref{Ex:Pn} and using formulas \eqref{E:Ltop-quad} and \eqref{E:pushO-quad}, we get that 
\begin{equation}\label{E:2Quadr-for}
\begin{sis}
\wt f_*\O_{\wt X}(\wt L)=f_*\O_X(L)=E, \\
\wt L^{n+1}=2L^{n+1}=2(2\deg E+\deg A),
\end{sis}
\Rightarrow s(\wt L)=4+2\frac{\deg A}{\deg E}. 
\end{equation}
 
 %SPIEGAZIONE PIU" DETTAGLIATA
% The top self-intersection of $M$ can be computed using the projection formula and  \eqref{E:Ltop-quad}:
%\begin{equation}\label{E:top-2Quadr}
%M^{n+1}=2(H_{X})^{n+1}=2(2\deg E+\deg A).
%\end{equation}

%We now compute the push-forward  $g_*\O_X(M)$ using the the projection formula, the formula for $\pi_*\O_X$ (coming from the structure of the finite double cover $\pi$), and  \eqref{E:pushO-quad}:
%\begin{equation}\label{E:pushO-2Quadr}
%g_*\O_X(M)=f_*\pi_*\O_X(M)=f_*\left(\O_{X}(H_{X})\otimes [\O_{X}\oplus \O_{X}(-mH-f^*B)] \right)=f_*\left(\O_{X}(H_{X})\right)= E.
%\end{equation}
%In particular $g_*\O_X(M)$ is nef (by the assumption of Example \ref{Ex:Quadrics}). From \eqref{E:top-2Quadr} and \eqref{E:pushO-2Quadr}, we get that 
%\begin{equation}\label{E:slopeL-2Quad}
%s(M)=4+2\frac{\deg A}{\deg E}. 
%\end{equation}

%Note that 
%$$
%s(M)\geq \BS(M_F) \Leftrightarrow \deg A\geq -2\mu(E)=-2\frac{\deg E}{\rk E}.
%$$

\un{Examples of small slopes}

In order to obtain examples of small slope, we can take $T=\P^1$ and 
$$
 E=\O_{\P^1}^{\oplus n+2-r}\oplus \O_{\P^1}(d)^{\oplus r} \text{ with } 3\leq r\leq n+2 \text{ and } d\geq 1 \quad \text{ and } \deg A=-2d.
$$
With these choices, the general element $X$ of 
$$|2H+h^*(A)|=\P(H^0(\P^1, \Sym^2(E)\otimes \O_{\P^1}(A)))$$
 is a family of quadrics over $T$ of generic rank $r$ (and hence it is normal since $r\geq 3$) and, using \eqref{E:slopeL-quad}, its slope is 
 $$
 s(L)=2-\frac{2d}{rd}=2-\frac{2}{r}\leq 2-\frac{2}{n+2}=\BS(L_F),
 $$
 with equality if and only if $r=n+2$, i.e. the family $f:X\to \P^1$ is generically smooth. And the same thing is true for a double finite cover $\wt f:\wt X\to T$ of $f:X\to T$ as above. 
 
 In particular, this example shows that  Theorem \ref{T:BS-pos} can fail without the Chow semistability of the general fiber and that Theorem   \ref{T:ss-push} can fail without the slope semistability of $f_*\O_X(L)$.

\end{example}

\begin{example}\label{Ex:Scroll}[Families of Rational Normal Scrolls and their double covers]

Let $T$ be a smooth irreducible projective curve, let $E$ be a locally free sheaf of rank $2$ on $T$ and denote by $\mu_{-}(E)$ the smallest slope in the Harder-Narasimhan filtration of $E$. 
Consider the projective bundle $h:S:=\P_T(E)\to T$  and let $H_S$ be a tautological divisor on $S$. 
Fix a $n$-tuple of integers $d_1\geq \ldots \geq d_n\geq 0$ and an $n$-tuple $\{A_1,\ldots, A_n\}$ of divisors on $T$ of degree $a_i:=\deg A_i$ subject to the following assumptions
\begin{equation}\label{E:ai-assum}
a_i+d_i\mu_-(E)\geq 0 \text{ for any } 1\leq i \leq n. \tag{$\dagger$}
\end{equation}
Consider the projective bundle  $g: \P_S\left(\bigoplus_{i=1}^n \O_S(d_i)\otimes h^*\O_T(A_i) \right)\to S$ and let $H$ be a tautological divisor on it.

\un{Families of rational normal scrolls}

Consider the polarized family 
$$f: \P_S\left(\bigoplus_{i=1}^n \O_S(d_i)\otimes h^*\O_T(A_i) \right)=:X\xrightarrow{g} S\xrightarrow{h} T \text{ and } L=H. $$
Note that $L=H$ is nef on $X$ since, for each $1\leq i \leq n$, the line bundle $\O_S(d_i)\otimes h^*\O_T(A_i)$ is nef on $S$ by Miyaoka's theorem, using the assumption \eqref{E:ai-assum}.

The general (and indeed any) fiber of $f$ is the rational normal scroll 
$$F=\P_{\P^1}\left(\bigoplus_{i=1}^n \O_{\P^1}(d_i)\right)$$ 
and $L_F$ is a tautological divisor on $F$. Hence we have that 
\begin{equation}\label{E:BS-scroll}
\BS(L_F)=(n+1)\frac{L_F^n}{h^0(F,L_F)}=(n+1)\frac{\sum_i d_i}{\sum_i d_i+n}. 
\end{equation}

We now compute the sheaf $f_*\O_X(L)$ using \eqref{E:form-Proj} as it follows
$$
f_*\O_X(L)=h_*\left(g_*\left(\O_X(L)\right) \right)=h_*\left( \bigoplus_{i=1}^n \O_S(d_i)\otimes h^*\O_T(A_i)\right) =\bigoplus_{i=1}^n \Sym^{d_i}(E)\otimes \O_T(A_i).
$$
Since $\mu_-(\Sym^{d_i}(E)\otimes \O_T(A_i))=d_i\mu_-(E)+a_i\geq 0$ because of the assumption \eqref{E:ai-assum}, we conclude that $f_*\O_X(L)$ is nef by Hartshorne's theorem. 
Moreover, taking the degree in the above formula, we get 
\begin{equation}\label{E:pushO-scroll}
\deg f_*\O_X(L)=\sum_{i=1}^n \left[\deg \Sym^{d_i}(E)+(\rk \Sym^{d_i}(E))\deg \O_T(A_i)\right]= \sum_{i=1}^n\left[\binom{d_i+1}{2}\deg E+(d_i+1)a_i \right]. 
\end{equation}
Observe that, using the above formula, the assumption \eqref{E:ai-assum} and the fact that $\deg E\geq 2\mu_-$ with equality if and only if $E$ is semistable, we deduce that 
$\deg f_*\O_X(L)>0$ if and only if either  $E$ is not semistable or one of the inequalities in \eqref{E:ai-assum} is strict.

Let us now compute the top self-intersection of $L$. The non-zero Chern classes of the locally free sheaf  $\bV=\bigoplus_i \O_S(d_i)\otimes h^*\O_T(A_i)$ on $S$ are
\begin{equation*}
\begin{sis}
& c_1(\bV)=\sum_{i=1}^n \left(d_iH_S+h^*A_i\right)\in A^1(S),\\
& c_2(\bV)=\sum_{1\leq i <j\leq n}\left(d_id_j\deg E+d_ia_j +d_ja_i \right)\in \Z\cong A^2(S),
\end{sis}
\end{equation*}
where we used the formula $H_S^2=\deg E$ (see \eqref{E:form-Proj}).
Using the above formulas and by applying  \eqref{E:form-Proj} first  to the projective bundle $g:\P_S(\bV)\to S$ and then to its restriction to the divisors  $H_S$ and $h^*(A_i)$ of $S$, we get that 
$$
L^{n+1}=H^{n+1}=g^*(c_1(\bV))\cdot H^n-g^*(c_2(\bV))\cdot H^{n-1}=\sum_{i=1}^n \left[g^*\left(d_iH_S+h^*A_i\right)\cdot H^n\right]- c_2(\bV)=
$$ 
$$
=  \sum_{i=1}^n d_i (H_{|g^*{H_S}})^n + \sum_{i=1}^n (H_{|g^*(h^*(A_i))})^n-c_2(\bV)
= \sum_{i=1}^n \left[d_i  c_1(\bV)\cdot H_S)\right] + \sum_{i=1}^n \left[ c_1(\bV)\cdot h^*(A_i))\right] -c_2(\bV)=
$$
$$
= \sum_{i=1}^n \left[d_i \sum_{j=1}^n \left(d_j\deg E+a_j \right)\right] + \sum_{i=1}^n \sum_{j=1}^n a_i d_j - \sum_{1\leq i <j\leq n}\left(d_id_j\deg E+d_ia_j +d_ja_i \right)=
$$
\begin{equation}\label{E:Ltop-scroll}
=\left(\sum_{i=1}^n d_i^2+\sum_{1\leq i <j \leq n} d_id_j\right) \deg E+\sum_{i=1}^n 2d_ia_i+\sum_{1\leq i\neq j \leq n} d_ia_j.
\end{equation}
From \eqref{E:pushO-scroll} and \eqref{E:Ltop-scroll}, we get that (assuming $\deg f_*\O_X(L)>0$) 
\begin{equation}\label{E:slopeL-scroll}
s(L)=\frac{\left(\sum_{i=1}^n d_i^2+\sum_{1\leq i <j \leq n} d_id_j\right) \deg E+\sum_{i=1}^n 2d_ia_i+\sum_{1\leq i\neq j \leq n} d_ia_j}{\sum_{i=1}^n\binom{d_i+1}{2}\deg E+\sum_{i=1}^n (d_i+1)a_i}
\end{equation}

\un{Double covers}

Fix now two integers $\alpha, \beta$ such that $\alpha\geq 2$ and $\alpha d_n+\beta>0$, and a  divisor $B$ on $T$  such that 
\begin{equation*}
 |2(\alpha L+\beta g^*H_S+f^*B)|\neq \emptyset \text{ and the general element in }   |2(\alpha L+\beta g^*H_S+f^*B)| \text{ is smooth.} \tag{$\dagger$$\dagger$}
 \end{equation*}
 
 Note that the above assumption is realised by a sufficiently positive divisor $B$ on $T$ since the divisor $\alpha L+\beta g^*H_S$ is $f$-relatively very ample under the above assumptions on $\alpha$ and $\beta$ (see \cite[(5.7)]{Fuj}). 
Take a general effective smooth divisor $R\in |2(\alpha H+\beta g^*H_S+f^*B)|$ and denote by $\pi:\wt X\to X$ the  finite double cover ramified along $R$ and set $\wt f:\wt X\xrightarrow{\pi}X\xrightarrow{f} T$. 
Consider the nef divisor $\wt L:=\pi^*(L)$ on $\wt X$. 

Note that a general polarized fiber $(\wt F,\wt L_F)$ of $\wt f$  is a finite double cover of $(F,L_F)$   ramified along a divisor of $|2(\alpha L_F+\beta (H_S)_{|F})|$, and hence 
%Hence $(\wt F,\wt L_F)$   a hyperelliptic polarized variety of Type II in the sense of \cite{Fuj} with $\wt L_F^n=4$ and $h^0(\wt F,\wt L_F)=n+2$, so that  
\begin{equation}\label{E:BS-2Scroll}
\BS(\wt L_F)=2\BS(L_F)=2(n+1)\frac{\sum_i d_i}{\sum_i d_i+n}.
\end{equation}
%Note that, by the Riemann-Hurwitz formula for $\pi$ and the adjunction formula for the quadric hypersurface $Q\subset \P^{n+1}$, we have that $K_F=\pi_{|F}^*(K_{Q}+mH_{Q})=\pi_{|F}^*(-nH_{Q}+mH_Q)=(m-n)\pi_{|F}^*(H_Q)$.
%Hence, we have that $\kappa(F)\geq 0$ (resp. $F$ is canonically polarized)  if and only if $m\geq n$ (resp. $m\geq n+1$).  
 
  Arguing as in Example \ref{Ex:Pn} and using formulas \eqref{E:pushO-scroll}, \eqref{E:Ltop-scroll} and \eqref{E:slopeL-scroll}, we get that 
\begin{equation}\label{E:2Scroll-for}
\begin{sis}
\wt f_*\O_{\wt X}(\wt L)=f_*\O_X(L), \\
\wt L^{n+1}=2L^{n+1},
\end{sis}
\Rightarrow s(\wt L)=2s(L). 
\end{equation}

\un{Special cases}

Note that if either $d_1=\ldots= d_n:=d$ (which implies that $(F,L_F)$ is Chow stable) or $\deg E=2\mu_(E)$ and $a_i+d_i\mu_-(E)=C$ for some positive constant $C$ and any $i$ (which is equivalent to the semistability of $f_*\O_X(L)$), then we have that (for all choices of $a_i$ subject to \eqref{E:ai-assum})
$$
s(L)=(n+1)\frac{d}{d+1}=\BS(L_F).
$$

In particular, we get examples where Theorem \ref{T:BS-pos} and Theorem   \ref{T:ss-push}  are sharp. And the same thing is true for a double finite cover $\wt f:\wt X\to T$ of $f:X\to T$ as above.

On the other hand, if not all the integers $d_i$ are equal among themselves then, by fixing some numbers $\{a_2,\ldots, a_n\}$ subject to \eqref{E:ai-assum} and letting $a_1\to +\infty$, we get that
$$
s(L)\xrightarrow{a_1\to +\infty} \frac{d_1+\sum_i d_i}{d_1+1}<(n+1)\frac{\sum_i d_i}{\sum_i d_i+n}=\BS(L_F),
$$ 
where we used that $d_1>\frac{\sum_i d_i}{n}$ (which follows from the fact that $d_1\geq \ldots \geq d_n$ and that not all of the $d_i$'s are equal). 
As an extreme case, if $d_1:=d\geq 2>d_2=\ldots =d_n=0$ and $a_2=\ldots =a_n=0$ then we have that 
$$
s(L)=\frac{d^2\deg E+2d a_1}{\binom{d+1}{2}\deg E+(d+1)a_1}=2\frac{d}{d+1}<(n+1)\frac{d}{d+n}=\BS(L_F).
$$ 

 In particular, this example shows that  Theorem \ref{T:BS-pos} can fail without the Chow semistability of the general fiber and that Theorem   \ref{T:ss-push} can fail without the slope semistability of $f_*\O_X(L)$. And the same thing is true for a double finite cover $\wt f:\wt X\to T$ of $f:X\to T$ as above.

\end{example}

\subsection{Families of hypersurfaces in weighted projective spaces}\label{Sec:hyperweight}

The aim of this subsection is to compute the slope of families of hypersurfaces inside a weighted projective space over $\P^1$.  

Let $\un a=(a_0,\ldots,a_{n+1})$ be a collection of positive natural numbers (for some $n\geq 1$) and consider the ($n+1$)-dimensional weighted projective space $\P(\un a):=\Proj S(\un a)$, where $S(\un a)$ is the graded polynomial algebra $k[X_0,\ldots,X_{n+1}]$  such that $X_i$ has weight  $a_i$. Without loss of generality (see \cite[Sec. 1.3]{Dol}), we can assume that $\un a$ is \emph{reduced} (or well-formed), i.e. 
$$
1=\gcd(a_0,\ldots,a_{i-1},a_{i+1}, \ldots, a_{n+1}) \quad \text{ for any } 0\leq i \leq n+1.
$$
Denote by $H_{\un a}$ any tautological divisor, i.e. Weil $\Q$-Cartier divisor such that $\O_{\P(\un a)}(H_{\un a})=\O_{\P(\un a)}(1)$. For any $m\geq 0$, denote by $S(\un a)_m$ the (finite dimensional) $k$-vector space of homogeneous elements of $S(\un a)$ of degree $m$. Set $|\un a|:=\sum_i a_i$. 

Recall the following well-known facts (see \cite[Sec. 1.4, Sec. 2.1]{Dol}):
\begin{equation}\label{E:form-Pw}
\begin{sis}
& H_{\un a}^{n+1}=\frac{1}{\prod_i a_i},\\
& H^i(\P(\un a),\O_{\P(\un a)}(m))= 
\begin{sis}
\dim S(\un a)_m & \quad \text{ if } 0=i, \\
0 & \quad \text{ if } 0<i<n+1, \\
\dim S(\un a)_{-m-|\un a|} &\quad  \text{ if } i=n+1, \\
\end{sis}\\
& mH_{\un a} \text{ is Cartier } \Leftrightarrow mH_{\un a} \text{ is Cartier and base point free } \Leftrightarrow \lcm(a_0,\ldots, a_{n+1})\vert m,\\
& K_{\P(\un a)}=-|\un a| H_{\un a}. \\
\end{sis} 
\end{equation}

Consider $\P(\un a)\times \P^1$, denote by $p_1: \P(\un a)\times \P^1\to \P(\un a)$ and $p_2: \P(\un a)\times \P^1\to \P^1$ the two projections, and set $H_1:=p_1^*H_{\un a}$ and by $H_2$ the pull-back of a tautological divisor on $\P^1$ (i.e. a fiber of $p_2$).
Fix integers $d,e, h> 0$ and $l\geq 0$ such that 
\begin{equation}\label{E:ass-int}
\begin{sis}
& \dim S(\un a)_{e}-\dim S(\un a)_{e-d}>0,\\
&\lcm(a_0,\ldots, a_{n+1})\vert d. \\
\end{sis}
\end{equation}
From the second assumption in \eqref{E:ass-int}, the divisor $dH_1+lH_2$ is Cartier and base point free; hence, the general divisor in  $|dH_1+lH_2|$ is normal and connected by Bertini's theorems. 
Fix now a normal connected hypersurface $X\in |dH_1+lH_2|$, which is endowed with the fibration $f=(p_2)_{|X}:X\to \P^1$. Consider the ample $\Q$-Cartier Weil divisor $L:=(eH_1+hH_2)_{|X}$ on $X$.

\begin{remark}\label{R:wps}
\begin{enumerate}
\noindent 
\item \label{R:wps1} The general fiber $F$ of $f$ is a normal connected $n$-dimensional hypersurface in $\P(\un a)$ which is a general element of the linear system $|dH_{\un a}|$ and the restriction of the polarisation $L$ is equal to  $L_F=(eH_{\un a})_{|F}$.    By our assumptions \eqref{E:ass-int} on $d$, $F$ is well-formed \cite[6.10]{IaFl} and quasi-smooth 
\cite[Thm. 8.1]{IaFl}; hence by the adjunction formula for  $F\subset \P(\un a)$ (see \cite[6.14]{IaFl}) and \eqref{E:form-Pw}, the canonical divisor of $F$ is equal to 
\begin{equation}\label{E:KF-hyper}
K_F=(K_{\P(\un a)}+F)_F=(d-|\un a|)(H_{\un a})_{|F}. 
\end{equation}
In particular, we have the following trichotomy 
$$
\begin{sis}
& F \text{ is Fano } \Longleftrightarrow d<|\un a|,\\
& F \text{ is CY } \Longleftrightarrow d=|\un a|,\\
& F \text{ is canonically polarized } \Longleftrightarrow d>|\un a|.\\
\end{sis}
$$
\item \label{R:wps2} By the adjunction formula for  $X\subset \P(\un a)\times \P^1$ and \eqref{E:form-Pw}, the relative canonical divisor of $f$ is equal to 
\begin{equation}\label{E:relf-hyper}
K_{X/\P^1}=(K_{\P(\un a)\times \P^1/\P^1}+X)_{|X}=\left((d-|\un a|)H_1+l H_2\right)_{|X}. 
\end{equation}
Hence we have that 
$$
L=K_{X/\P^1} \Longleftrightarrow e=d-|\un a| \text{ and } h=l.
$$
In particular, in this case the general fiber $F$ is canonically polarized. 
\end{enumerate}
\end{remark}

We now compute the numerical invariants of $L$, i.e. $L^{n+1}$ and $\deg f_*\O_X(L)$.  
Using that $H_2^2=0$ and  $H_1^{n+1}\cdot H_2=H_{\un a}^{n+1}$ and formula \eqref{E:form-Pw}, we compute the  top self-intersection of $L$  as it follows
 \begin{equation}\label{E:top-hyper}
L^{n+1}=(eH_1+hH_2)^{n+1} \cdot (dH_1+l H_2)=[e^{n+1}l+ (n+1)e^nhd] H_1^{n+1}\cdot H_2=\frac{e^{n+1}l+ (n+1)e^nhd}{\prod_i a_i}. 
\end{equation}

In order to compute $f_*\O_X(L)$, we take the exact sequence of $X\subset \P(\un a)\times \P^1$, we  twist by $\O_{\P(\un a)\times \P^1}(eH_1+hH_2)$ and then take the reflexive hulls to get the exact sequence:
$$
0 \to \O_{\P(\un a)\times \P^1}((e-d)H_1+(h-l)H_2)\to   \O_{\P(\un a)\times \P^1}(eH_1+hH_2)\to  \O_{\P(\un a)\times \P^1}(eH_1+hH_2)_{|X}=\O_X(L)\to 0. 
$$
By taking the push-forward  along $f$ of the above exact sequence, we get the following exact sequence of locally free sheaves on $\P^1$:
\begin{equation*}
0 \to \O_{\P^1}(h-l)^{\dim S(\un a)_{e-d}} \to   \O_{\P^1}(h)^{\dim S(\un a)_{e}} \to f_*\O_F(L_F)\to 0,
\end{equation*}
from which we deduce that 
\begin{equation}\label{E:push-hyper}
\deg f_*\O_X(L)= \deg  \O_{\P^1}(h)^{\dim S(\un a)_{e}} -  \O_{\P^1}(h-l)^{\dim S(\un a)_{e-d}} =h(\dim S(\un a)_{e}-\dim S(\un a)_{e-d})+l\dim S(\un a)_{e-d}.
\end{equation}
In particular, by the assumptions \eqref{E:ass-int}, we have that $\deg f_*\O_X(L)>0$, and the slope of $L$ is equal to 
\begin{equation}\label{E:slope-hyper}
s(L)=\frac{e^{n+1}l+ (n+1)e^nhd}{\prod_i a_i\cdot [h(\dim S(\un a)_{e}-\dim S(\un a)_{e-d})+l\dim S(\un a)_{e-d}]}. 
\end{equation}

\begin{remark}
	The divisor $L$ is  $f$-positive if and only if  	\begin{equation}\label{E:fpos-wps3}
	\dim S(\un a)_e\geq \left[1+(n+1)\frac{d}{e}\right]\dim S(\un a)_{e-d}.
	\end{equation}
	
	In order to show that, let us compute the Barja-Stoppino invariant of $L_F$. 
	The top-self intersection of $L_F$ is equal to (using \eqref{E:form-Pw})
	\begin{equation}\label{E:topF-hyper}
	L_F^n=(eH_{\un a})^n \cdot dH_{\un a}=\frac{e^nd}{\prod_i a_i}. 
	\end{equation}
	Arguing similarly to the above computation of $\deg f_*\O_X(L)$, it follows that 
	\begin{equation}\label{E:h0F-hyper}
	h^0(F,L_F)=h^0(\P(\un a),  \O_{\P(\un a)}(e))- h^0(\P(\un a),  \O_{\P(\un a)}(e-d))=\dim S(\un a)_{e}-\dim S(\un a)_{e-d}. 
	\end{equation}

	Hence we get that 
	\begin{equation}\label{E:BSF-hyper}
	\BS(L_F)=(n+1)\frac{L_F^n}{h^0(F,L_F)}= \frac{(n+1)e^nd}{\prod_i a_i\cdot  [\dim S(\un a)_{e}-\dim S(\un a)_{e-d}]}.
	\end{equation}
	By combining the formulas \eqref{E:slope-hyper} and \eqref{E:BSF-hyper}, we get \eqref{E:fpos-wps3}.
	
	We finally note that Inequality \eqref{E:fpos-wps3}  is trivially true if $e<d$, while we don't know if it always true for $e\geq d$. 
	 
\end{remark}

The above formula \eqref{E:slope-hyper} simplifies if we are in the following 

\subsubsection{Special case:} $1=e<d$. 

Indeed, by the first assumption in \eqref{E:ass-int}, we must have that $\dim S(\un a)_1\geq 1$, which implies that some of the weights $a_i$ must be equal to one. If we define the natural number  $0\leq u:=\{i: a_i=1\}-1\leq n+1$, then 
we have that $\dim S(\un a)_1=u+1$. Substituting into \eqref{E:slope-hyper}, we get that the slope of $L$ in this special case is equal to 
\begin{equation}\label{E:slope-spec}
s(L)=\frac{(n+1)d+\frac{l}{h}}{(u+1)\prod_i a_i}.
%\geq \frac{(n+1)\lcm(a_0,\ldots, a_{n+1})+\frac{l}{h}}{(u+1)\prod_i a_i}.
\end{equation}

We now consider examples of fibrations of small slopes.

\subsubsection{Example I}

Take 
$$\un a=(1,1, \alpha, \ldots,  \alpha) \quad \text{ with } \alpha \geq 2,  \quad d=m\alpha \quad \text{ with  } m\geq 1.$$ 
Note that assumptions \eqref{E:ass-int} hold true and that, by Remark \ref{R:wps}\eqref{R:wps1},   the general fiber $F$ is  canonically polarized (resp. of non-negative Kodaira dimension) if $m\geq n+2$ (resp. $m\geq n+1$). Formula \eqref{E:slope-spec} gives that 
$$
s(L)=\frac{(n+1)m\alpha+\frac{l}{h}}{2\alpha^n}\xrightarrow[\alpha \to +\infty]{} 0 \quad  \text{ if } n\geq 2.
$$
This example shows that in Theorem \ref{T:Barja}\eqref{T:Barja2} (resp. Theorem \ref{T:Barja}\eqref{T:Barja1}) the hypothesis that $\phi_{L_F}$ is generically finite  (resp. or that $L_F$ is Cartier)   cannot be dropped.

\begin{remark}\label{rem:index}
The construction above gives examples of families in any fixed dimension $n \ge 2$, in which the Gorenstein index of the general fibre goes to infinity. For simplicity, let us take $m=n+1$. Then for any odd integer $\alpha$ bigger than 2,  the Cartier index of $K_F=(\alpha-2){H_{\un a}}_{|F}$ is $\alpha$. In fact, on the one hand $\alpha K_F$ is Cartier and on the other hand $K_F^n = (\alpha-2)^n/\alpha^n$ and we conclude since $\alpha$ and $\alpha-2$ are coprime.
\end{remark}

\subsubsection{Example II}\label{Ex:Double}

%$1=e=d - |\un a |$ and  $h=l$ (which implies that  $L= K_{X/ \P^1}$).

%ESEMPIO di LUCA
%Consider $n \ge 2$ and $\un a=(1,1, \ldots, 1, 2(n+1), 3(n+1))$ where the number of 1's is $n$. Take $d= 6(n+1)$.  Then $K_F= \O_F(1)$. Set $h=l=1$. Then  
%$$
%s(L)=\frac{(n+1)d+ 1 }{6n(n+1)^2}=\frac{6(n+1)^2+ 1 }{6n(n+1)^2}= \frac{1}{n} + \frac{1 }{6n(n+1)^2}. 
%$$	

Consider the Sylvester sequence $\{s_n\}_{n\in \N}$ (see the sequence  \cite[\href{https://oeis.org/A000058}{A000058}]{OEIS}) defined inductively as $s_n=1+\prod_{i=0}^{n-1}s_i$ with the initial condition $s_0=2$.  
Define $b_i:=\prod_{0\leq j\neq i \leq n-1} s_j$ for every $0\leq i \leq n-1$, and take 
$$\un a=(1,1,3b_0,\ldots, 3b_{n-1}), \quad d=3(s_{n}-1)=\prod_{i=0}^{n-1} s_i, \quad e=1, \quad h=l>0.$$
Note that assumptions \eqref{E:ass-int} hold true and that we have 
$$1+|\un a|=3+3b_0+\ldots +3b_{n-1}=3(s_n-1)=d,$$ 
which is proved by induction on $n$ using the formula $s_n=s_{n-1}^2-s_{n-1}+1$.  By Remark \ref{R:wps}\eqref{R:wps2}, we have that $L=K_{X/\P^1}$.  
Formula \eqref{E:slope-spec} gives that 
$$
s(K_{X/\P^1})=\frac{3(n+1)(s_n-1)+1}{2\cdot 3^n(s_n-1)^{n-1}}, 
$$
which is smaller than $1$ if $n\geq 2$ and it decays double exponentially as $n\to +\infty$ since $s_n$ grows doubly exponentially in $n$ (see loc. cit.).  We expect that lower slopes are possible using examples of varieties of general type with small volume as constructed in  \cite{BPT} and \cite{TW}.

This example shows that in Theorem \ref{T:Barja}\eqref{T:Barja2} (resp. Theorem \ref{T:Barja}\eqref{T:Barja1}) the hypothesis that $\phi_{L_F}$ is generically finite  (resp. or that $L_F$ is Cartier)   cannot be dropped even under the assumption that $L_F=K_F$.

\subsubsection{Example III}\label{Ex:slope1}
Take
$$
\un a=(1,1,\ldots, 1, 2,n+3) \quad \text{ with }   e=1, \quad d=2(n+3) \quad \text{ and } h=l >0.
$$ 

By Remark \ref{R:wps}\eqref{R:wps2}  we have that $L=K_{X/\P^1}$ and the general fibre $F$ of $f: X \to \P^1$ is a canonically polarised variety of dimension $n$. By \eqref{E:slope-spec} we get

$$
s(L) = \frac{1+ (n+1) 2(n+3)}{ 2(n+3) n}= \frac{1}{2n(n+3)} + \frac{n+1}{n}	.
$$

When $n$ is even, $X$ is smooth and this shows that for families $f: X \to T$ with smooth general fibre the minimum slope tends to 1 when $n$ grows.  
In \cite[Theorem 1.5]{HZ} it is shown that for $n=2$ the sharp slope is $4/3$.

\subsubsection{Example IV}\label{Ex:Xiao}
Take
$$
\un a=(1,1,8,12) \quad \text{ with }   e=2, \quad d=24 \quad \text{ and } h=l >0.
$$ 
By Remark \ref{R:wps}\eqref{R:wps2}  we have that $L=K_{X/\P^1}$ and the general fibre $F$ of $f: X \to \P^1$ is a canonically polarised (singular) surface.

Note that $\dim S(\un a)_2 = 3$ and so 
$$
4\frac{h^0(F,L_F) - 2}{h^0(F,L_F)}= \frac{4}{3}.
$$

By \eqref{E:slope-hyper} we obtain
$$
s(L) = \frac{2^3+ 3\cdot 2^2 \cdot 24}{8 \cdot 12 \cdot 3}= 1 + \frac{1}{36} < \frac{4}{3}.
$$

This shows that the assumption that $\phi_{L_F}$ is generically finite in Corollary \ref{C:Xiao-high1} can not be dropped.

\subsubsection{Example IV}\label{Ex:Xiao2}

Take
$$
\un a=(1,1,\alpha k,\beta k) \quad \text{ with }   e=k, \quad d=\alpha \beta k \quad \text{ and } h=l >0.
$$ 
where $\alpha, \beta \ge 2$ and $k$ are positive integers.

Note that $\dim S(\un a)_k = k+1$ and so 
$$
4\frac{h^0(F,L_F) - 2}{h^0(F,L_F)}= \frac{4(k-1)}{k+1}.
$$

By \eqref{E:slope-hyper} we obtain
$$
s(L) = \frac{k^3+ 3\cdot k^2 \cdot \alpha \beta k}{\alpha \beta k^2 (k+1)}= \frac{k}{\alpha \beta (k+1)}  + \frac{3k}{k+1}.
$$

If $\alpha, \beta$ and $k$ are pairwise coprime, then $L_F$ is Cartier.  For $\alpha, \beta, k \gg 0$ we have
$$
s(L) < 4\frac{h^0(F,L_F) - 2}{h^0(F,L_F)}.
$$  

This shows that the assumption $\phi_{L_F}$ generically finite in Corollary \ref{C:Xiao-high1} can not be dropped in general, even if $L_F$ is Cartier.

\bibliographystyle{apsr}

\begin{thebibliography}{ACG11}
\bibitem[AFS17]{Alp} J. Alper, M. Fedorchuk, D. Smyth: \emph{Second flip in the Hassett-Keel program: projectivity.} Int. Math. Res. Not. 2017, No. 24, 7375--7419. 
\bibitem[ACG2]{GAC2}  E. Arbarello, M. Cornalba, P.A. Griffiths: \emph{Geometry of algebraic curves. Volume II.} With a contribution by Joseph Daniel Harris. Grundlehren der Mathematischen Wissenschaften 268. Springer, Heidelberg, 2011.
\bibitem[BPT13]{BPT} E. Ballico, R. Pignatelli, L. Tasin: \emph{Weighted Hypersurfaces with Either Assigned Volume or Many Vanishing Plurigenera.} Communications in Algebra 41 (2013), 3745--3752.
\bibitem[BarPhD]{Barja_Phd} M. A.  Barja: \emph{On the slope and geography of fibred surfaces and threefolds.} Ph. D. Thesis, 1998.
\bibitem[Bar00]{Bar1} M. A. Barja: \emph{Lower bounds of the slope of fibred threefolds.} Internat. J. Math. 11 (2000), 461--491.
\bibitem[Bar22]{BarjaUMI} M. Barja: \emph{Slope inequalities for fibrations of non-maximal albanese dimension.}  Boll. Unione Mat. Ital. 15 (2022), 3--15.
\bibitem[BS08]{BScanproj} M. A. Barja, L. Stoppino: \emph{Linear stability of projected canonical curves with applications to the slope of fibred surfaces.} J. Math. Soc. Japan  60 (2008), 171--192.
\bibitem[BS09]{BS1} M. A. Barja, L. Stoppino: \emph{Slopes of trigonal fibred surfaces and of higher dimensional fibrations.} Ann. Sc. Norm. Super. Pisa, Cl. Sci. (5) 8 (2009), 647--658.
\bibitem[BS14]{BS2}  M.A. Barja, L. Stoppino: \emph{Stability conditions and positivity of invariants of fibrations.} Algebraic and complex geometry, 1--40, Springer Proc. Math. Stat., 71, Springer, Cham, 2014.
\bibitem[BS16]{BS3} M. A. Barja, L. Stoppino: \emph{Stability and singularities of relative hypersurfaces.} Int. Math. Res. Not. 2016, No. 4, 1026--1053.
\bibitem[BS]{adjunctiontheory} M.C. Beltrametti, A. Sommese: \emph{The adjunction theory of complex projective varieties.} De Gruyter Expositions in Mathematics, 16. Walter de Gruyter \&  Co., Berlin, 1995.
\bibitem[Bos94]{Bost} J.B. Bost: \emph{Semi-stability and heights of cycles.} Invent. Math. 118 (1994), 223--253.
\bibitem[BFJ09]{BFJ} S. Boucksom, C. Favre, M. Jonsson: \emph{Differentiability of volumes of divisors and a problem of Teissier.} J. Algebraic Geom. 18 (2009), 279--308.
\bibitem[CP21a]{CP} G. Codogni, Zs. Patakfalvi: \emph{Positivity of the CM line bundle for families of K-stable klt Fano varieties}, Invent. Math. 223 (2021), 811--894.
\bibitem[CP21b]{CP2} G. Codogni, Zs. Patakfalvi: \emph{A note on families of K-semistable log-Fano pairs.}  Preprint arXiv:2107.07902.
\bibitem[CPZ19]{CPZ} G. Codogni, Zs. Patakfalvi, M. Zdanowicz:  \emph{Isotriviality of the Albanese morphism of klt varieties with nef anticanonical divisor in characteristic zero}. Appendix to  Zs. Patakfalvi and M. Zdanowicz:  \emph{On the Beauville-Bogomolov decomposition in characteristic $p\geq 0$}. Preprint arXiv:1912.12742
\bibitem[CH88]{CH} M. Cornalba, J. Harris: \emph{Divisor classes associated to families of stable varieties, with applications to the moduli space of curves.} Ann. Sci. \'Ecole Norm. Sup. (4) 21 (1988), 455--475.
\bibitem[Cor93]{Cproj} M. Cornalba:  \emph{On the projectivity of the moduli spaces of curves.} J. Reine Angew. Math. 443 (1993), 11--20.
\bibitem[Dol82]{Dol} I. Dolgachev: \emph{Weighted projective varieties.} Group actions and vector fields (Vancouver, B.C., 1981), 34--71, Lecture Notes in Math., 956, Springer, Berlin, 1982.
\bibitem[EH87]{EH}  D. Eisenbud, J. Harris: \emph{On varieties of minimal degree (a centennial account).} Algebraic geometry, Bowdoin, 1985 (Brunswick, Maine, 1985), 3--13. 
Proc. Sympos. Pure Math., 46, Part 1, Amer. Math. Soc., Providence, RI, 1987. 
\bibitem[Fuj18]{Fujino} O. Fujino: \emph{Semipositivity theorems for moduli problems.} Ann. of Math. (2) 187 (2018), 639--665.
\bibitem[Fuj83]{Fuj} T. Fujita: \emph{On hyperelliptic polarized varieties.} Tohoku Math. J. (2) 35 (1983), 1--44.
\bibitem[Ful11]{Fulger} M. Fulger: \emph{The cones of effective cycles on projective bundles over curves.} Math. Z. 269 (2011), 449--459.
\bibitem[Ful]{Ful} W. Fulton: \emph{Intersection theory.} Second edition. A Series of Modern Surveys in Mathematics [Results in Mathematics and Related Areas. 3rd Series. A Series of Modern Surveys in Mathematics], 2. Springer-Verlag, Berlin, 1998.
\bibitem[HMX14]{ACC}C.D. Hacon, J. McKernan, C. Xu: \emph{ACC for log canonical thresholds.} Ann. of Math. (2) 180 (2014), 523--571.
\bibitem[Har]{Har} R. Hatshorne: \emph{Algebraic Geometry}. Graduate texts in Mathematics 52, Springer-Verlag, 1977.
\bibitem[HH09]{HH} B. Hassett, D. Hyeon: \emph{Log canonical models for the moduli space of curves: the first divisorial contraction.} Trans. Amer. Math. Soc. 361 (2009), 4471--4489.
\bibitem[HZ21]{HZ} Y. Hu, T. Zhang: \emph{Fibered varieties over curves with low slope and sharp bounds in dimension three}. J. Algebraic Geom. 30 (2021), 57--95.
\bibitem[Har81]{Harris} J. Harris: \emph{A bound on the geometric genus of projective varieties.} Ann. Sc. Norm. Super. Pisa, Cl. Sci., IV. Ser. 8 (1981), 35-68.  
\bibitem[IaFl00]{IaFl} A.R. Iano-Fletcher: \emph{Working with weighted complete intersections.} Explicit birational geometry of 3-folds, 101--173. London Math. Soc. Lecture Note Ser., 281, Cambridge Univ. Press, Cambridge, 2000.
\bibitem[Kob92]{Kobayashi} M. Kobayashi: \emph{On Noether's inequality for threefolds.} J. Math. Soc. Japan 44 (1992), 145--156.
\bibitem[Kol90]{Kollar_ampleness} J. Koll\'{a}r: \emph{Projectivity of complete moduli}. J. Differential Geom. 32 (1990), 235--268.
\bibitem[Kol1]{Kollar_book} J. Koll\'{a}r: \emph{Singularities of the Minimal Model Program.}  With the collaboration of S\'{a}ndor Kov\'{a}cs. Cambridge Tracts in Mathematics 200, Cambridge University Press, Cambridge, 2013.
\bibitem[Kol2]{Kollar_moduli} J. Koll\'{a}r: \emph{Families of varieties of general type.} Draft of a book available at 
\\ https://web.math.princeton.edu/~kollar/FromMyHomePage/modbook-final.pdf.
\bibitem[Kon96]{Konno} K. Konno: \emph{A lower bound of the slope of trigonal fibrations.} Internat. J. Math. 7 (1996), 19--27.
\bibitem[Laz2]{Laz2} R. Lazarsfeld: \emph{Positivity in algebraic geometry. II. Positivity for vector bundles, and multiplier ideals.} A Series of Modern Surveys in Mathematics [Results in Mathematics and Related Areas. 3rd Series. A Series of Modern Surveys in Mathematics], 49. Springer-Verlag, Berlin, 2004.
\bibitem[LX14]{LiXu} C. Li, C. Xu: \emph{Special test configuration and K-stability of Fano varieties.} Ann. Math. 180 (2014), 1--36.
%\bibitem[LXZ22]{LXZ} Y. Liu, C. Xu, Z. Zhuang: \emph{Finite generation for valuations computing stability thresholds and applications to K-stability.} Preprint arxiv: 2102.09405
\bibitem[LXZ22]{LXZ} Y. Liu, C. Xu, Z. Zhuang: \emph{Finite generation for valuations computing stability thresholds and applications to K-stability.} Ann. Math. (2) 196 (2022), 507--566. 
\bibitem[Miy87]{Miy} Y. Miyaoka: \emph{The Chern classes and Kodaira dimension of a minimal variety}. Algebraic geometry, Sendai, 1985, 449-476. Adv. Stud. Pure Math., 10, North-Holland, Amsterdam, 1987.
\bibitem[Mum71]{Mum} D. Mumford: \emph{Theta characteristics of an algebraic curve.} Ann. Sci. \'Ecole Norm. Sup. 4 (1971), 181--191.
\bibitem[Nak]{Nak} N. Nakayama: \emph{Zariski-decomposition and abundance.} MSJ Memoirs, 14. Mathematical Society of Japan, Tokyo, 2004.
\bibitem[OEIS]{OEIS} OEIS Foundation Inc. (2020), The On-Line Encyclopedia of Integer Sequences, http://oeis.org.
\bibitem[KP17]{KP} Zs. Patakfalvi, K. S\'{a}ndor: \emph{Projectivity of the moduli space of stable log-varieties and subadditvity of log-Kodaira dimension.} J. Amer. Math. Soc. 30 (2017), 959--1021.
\bibitem[PX17]{PX} Zs. Patakfalvi, C. Xu: \emph{Ampleness of the CM line bundle on the moduli space of canonically polarized varieties}. Algebraic Geometry 4 (2017), 29--39.
%\bibitem[PST17]{PST} M. Pizzato, T. Sano, L. Tasin \emph{Effective nonvanishing for Fano weighted complete intersections}.  Algebra Number Theory 11 (2017), no. 10, 2369--2395.
\bibitem[Pos22]{Posva} Q. Posva: \emph{Positivity of the CM line bundle for K-stable log Fanos}. Trans. Am. Math. Soc. 375 (2022), 4943-4978.
\bibitem[Ohn92]{Ohno} K. Ohno: \emph{Some inequalities for minimal fibrations of surfaces of general type over curves.} J. Math. Soc. Japan 44 (1992), 643--666.
\bibitem[RS21]{Enea} E. Riva, L. Stoppino: \emph{The slope of fibred surfaces: unitary rank and Clifford index.} Preprint arXiv:2102.04142.
\bibitem[RT12]{RT} M. Rossi, L. Terracini: \emph{Linear algebra and toric data of weighted projective spaces.} Rend. Semin. Mat. Univ. Politec. Torino 70 (2012), 469--495.
\bibitem[Shi08]{Shin} D.K. Shin: \emph{Noether inequality for a nef and big divisor on a surface.} Commun. Korean Math. Soc. 23 (2008), 11--18.
\bibitem[TW21]{TW} B. Totaro, C. Wang: \emph{klt varieties of general type with small volume.} Preprint arXiv:2104.12200
%\bibitem[Vie01] {Viehweg} E. Viehweg: \emph{Positivity of direct image sheaves and applications to families of higher dimensional manifolds.} ICTP-Lecture Notes 6 (2001)
\bibitem[Xia87]{Xiao} G. Xiao: \emph{Fibered algebraic surfaces with low slope.} Math. Ann. 276 (1987), 449--466.
\bibitem[XZ20]{Xu_Zhuang} C. Xu, Z. Zhuang: \emph{On positivity of the CM line bundle on K-moduli spaces.} Ann. of Math. (2) 192 (2020), 1005--1068.
\bibitem[Xu21]{survey} C. Xu: \emph{K-stability of Fano varieties: an algebro-geometric approach.} EMS Surv. Math. Sci. 8 (2021), 265--354. 

\end{thebibliography}

\end{document}